\documentclass[12pt]{amsart}
\usepackage{latexsym,fancyhdr,amssymb,color,amsmath,amsthm,graphicx,listings,comment,booktabs}
\usepackage{chessboard,forest,adjustbox}
\usepackage{tikz}
\usepackage[absolute,overlay]{textpos}
\usepackage[section]{placeins}
\newtheorem{thm}{Theorem} 
\newtheorem{lemma}{Lemma}

\let\paragraph\subsection
\newcommand{\pos}[1]{\chessboard[setpieces={#1}]}       
\newcommand{\posINIT}{Qa1,Kb1,Pa2,Pb2,qa4,kb4,pa3,pb3}  
\newcommand{\posMa} {Qa1,Kb1,Pb2,Pb3,qa4,kb4,pa3}       
\newcommand{\posMb} {Qa1,Kb1,Pb2,Pb3,qa4,kb4,pa2}       
\newcommand{\posMc} {Qa2,Kb1,Pb2,Pb3,qa4,kb4}           
\newcommand{\posMd} {Qa2,Kb1,Pb2,Pb3,qa3,kb4}           
\newcommand{\posMATE}{Qa3,Kb1,Pb2,Pb3,kb4}              
\newcommand{\posAxbTwo}{Qa1,Kb1,Pb3,qa4,kb4,pb2}        
\newcommand{\posQxbThree}{Qa1,Kb1,Pb2,pa3,qb3,kb4}      
\newcommand{\posKxbThree}{Qa1,Kb1,Pb2,pa3,qa4,kb3}      
\newcommand{\posDraw}{Qa1,Kb1,Pa2,Pa3,qa4,kb4,pb3}      
\setchessboard{maxfield=b4,boardfontsize=8pt,showmover=false,labelleft=false, labelbottom=false}

\title{God numbers for Graphs, Games and Groups}

\author{ Z. Adams, M.Z. Cassim, C. Hou, O. Knill, V. Seco Roopnaraine, M.H. Saleem}

\date{May 15, 2026}
\address{Department of Mathematics \\ Harvard University \\ Cambridge, MA, 02138 }


\makeindex
\begin{document}
\maketitle

\begin{abstract}
We describe and axiomatize finite solitaire puzzles and zero sum sequential games graph theoretically.
Zermelo's theorem telling that there is a win for one of the players or a draw follows from the definitions. 
The god number is a geometric quantity that quantifies the number of moves necessary to solve 
the puzzle. In the solitaire case, the god number is the minimal distance from the 
initial state $v$ to the solution space $A$. If $v$ and $A$ are not specified, the god number is the graph diameter.
God number computations are related to combinatorial sorting problems and is a NP-complete problem in general 
even when restricted to concrete sliding problems. In the two-player case, the god number is a minimax critical value: 
it minimizes the maximal game event length over the set of all strategies. A strategy is a 
sub-graph of the game graph that contains the initial vertex. 
The definition is done so that a ``mate in k" chess problem has 
god number k. As for examples: in the solitaire case, we look at group games 
like Rubik type problems, transposition games related to sorting, at sliding puzzles 
like the 15 puzzle or rainbow ball, or the tower of Hanoi. For two-player games,
we illustrate the story using examples of small chess games, a small card game or tic-tac-toe type problems.
\end{abstract}

\tableofcontents

\vfill
\pagebreak

\section{In a nutshell}

\paragraph{}
A {\bf solitaire} is a finite simple directed graph $G=(V,E)$ with an initial condition 
$v \in V$ and a {\bf solution set} $A \subset V$, for which the puzzle is considered solved. 
We use the term ``puzzle" as a synonym to ``one-player game" or ``solitaire". 
\index{solitaire}
\index{puzzle}
\index{one-player game}
The game is not continued when the {\bf equilibrium set} $S$ and especially the
{\bf solution set} $A \subset S$ is reached. The {\bf equilibrium set} $S$ 
is defined as the set of vertices in $G$ with only incoming edges.
\index{solution set A}
\index{equilibrium set S}
\index{game graph}
A graph is {\bf simple} if no self loops $v \to v$, nor multiple connections appear. It is possible
to see an edge $(a,b)$ as {\bf undirected}, if both connections $a \to b,b \to a$ are present. 
The {\bf god number of a winnable solitaire} is the minimal path length from $v$ to $A$. 
For undirected graphs, it is the {\bf geodesic graph distance} between $v$ and $A$. 
\index{graph distance}
If a solitaire is non-winnable, it is custom to define the god number to be infinite.
We can still then still look at maximal game events. Also, if  $S$ and so $A$ are empty,
we could look at the {\bf maximal length of a game event}. Finding this is related to the question whether 
the graph is {\bf Hamiltonian}.
\index{Hamiltonian}

\paragraph{}
Here is the graph theoretical axiom system for a two-player game. It will be outlined again in
separate sections and summarized in a slightly more formal way at the end:
a {\bf two-player game} is a {\bf bipartite directed graph} $G=(V \cup W,E)$ with an {\bf initial position} $v$
and a set $A$ of {\bf win positions} in the set of equilibria $S$. The set $V$ consists of all positions of the first player,
while $W$ is the set of positions of the second player. We refer to $G$ as the {\bf game graph}. 
\index{win positions}
\index{two-player game}
\index{game graph}
The first player moves along edges pointing from $V$ to $W$,
the second payer moves on edges pointing from $W$ to $V$. The set $S$ of {\bf end position}
consists of all vertices without outgoing edges. 
\index{end positions} 
The complement of $A$ in $S$ are {\bf draw positions}. These are vertices representing game positions
where the game can not be continued but where none of the players has won. 
\index{draw positions} 
A {\bf game event} is a maximal simple path $\gamma$ in $G$ that starts at $v$. It could be a path of 
length $0$ like if $v \in S$.
\index{game event}
That $\gamma$ is {\bf simple} means that it does not cover any vertex twice.
{\bf Maximal} means that it is not part of a strictly larger game event. 
A game event terminating in $A$ is called {\bf winning}, otherwise it is a {\bf draw game event}: it can
either end up in $S$ or then trap itself, as it needs to be {\bf self-avoiding}.
\index{winning event}
\index{draw event}
\index{self avoiding}
The first player who reaches $A$ is a winner. There are two scenarios for a draw:
the draw event $\gamma$ can either end in $S \setminus A$ or then end in a vertex, where a continuation
is not possible without hitting an already covered position and so violating the simplicity of $\gamma$. 

\begin{figure}[!htpb]
\scalebox{0.25}{\includegraphics{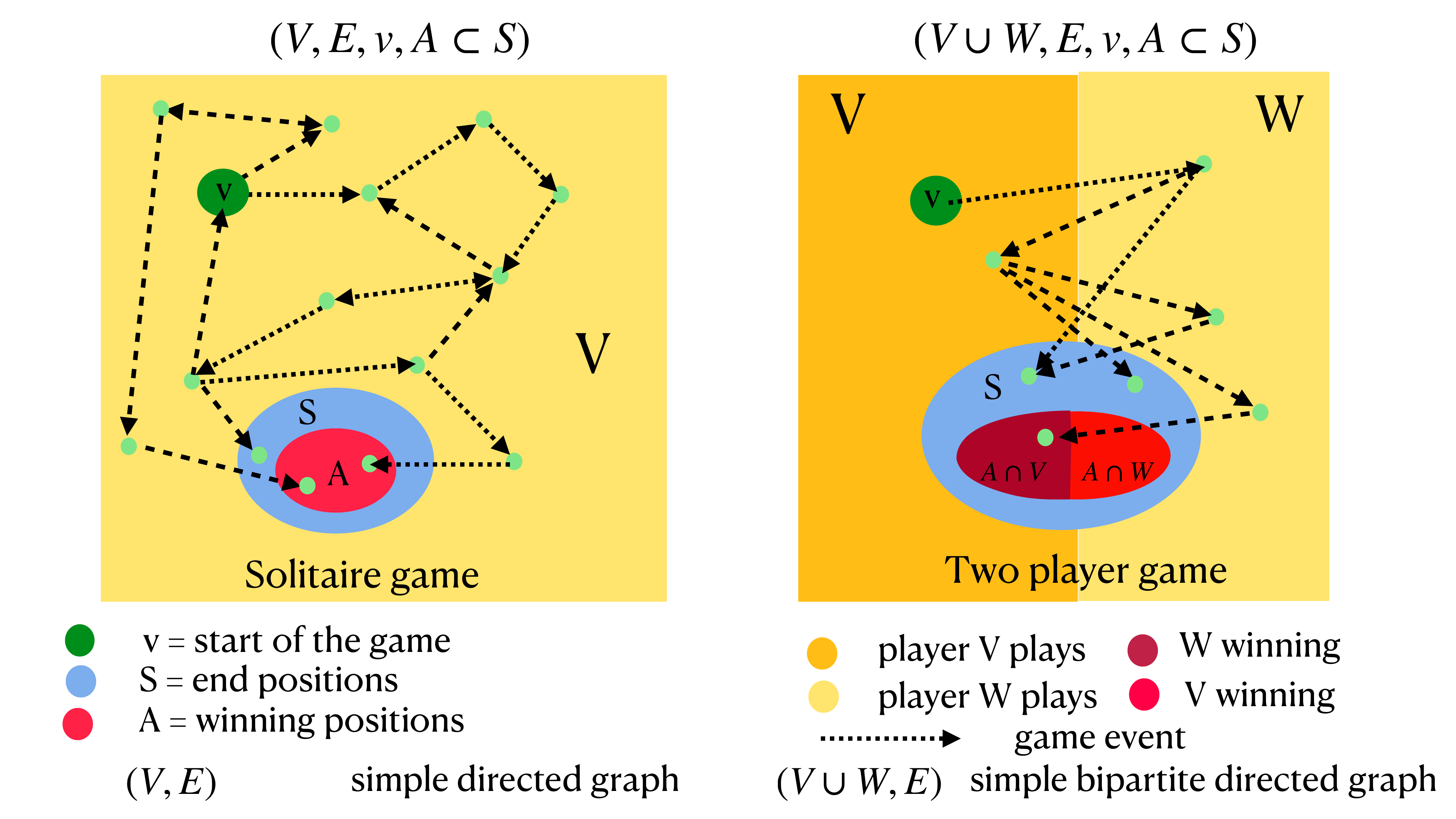}}
\label{game}
\caption{
Both solitaires and two-player games are defined graph theoretically.
In the solitaire case, the puzzle is either solvable or not. 
In the two-player case, there is a {\bf Zermelo trichotomy}:
a game is either a win for V (V can remove
edges leaving V so that every game event lands in $A \cap W$), 
a win for W (W can remove edges leaving W so that every game
event lands in $A \cap V$, or a draw. If both players can remove
edges from their side so that no game event ends in A, we have
a {\bf draw}.
}
\end{figure}
\index{Zermelo trichotomy}
\index{trichotomy of Zermelo}

\paragraph{}
A {\bf strategy} is a sub-graph $H$ of $G$ that contains $v$. \index{strategy} 
\footnote{We discuss this definition more below. Simplicity is intended.}
A {\bf win strategy for V} is a strategy $H$ that does not remove any edges starting from $W$
and which has the property that every game event in $H$ ends in $A \cap W$.  \index{win strategy for V}
A {\bf win strategy for W} is a strategy $H$ that does not remove any edges starting 
from $V$ such that game event in $H$ ends in $A \cap V$.  \index{win strategy for W}
A {\bf draw strategy} is a strategy that is neither a win strategy for $V$, nor a win strategy for $W$. 
\index{draw strategy}
In a draw strategy $H$, both players have limited their move 
options (removed edges) so that the win set $A$ can not be reached in $H$ by neither of the players.

\paragraph{}
In this frame work, {\bf Zermelo's theorem} for a game $(G=(V,W,E),v,A)$ 
logically follows from the definition. Let us prove it nevertheless, 
as traditionally one uses backward induction on the finite game tree \cite{MaschlerSolanZamir}.
\footnote{And because an AI pupil to which we have taught the axioms once did first not believe that it is as simple.}
We also have {\bf game graph}, not a {\bf game tree} so that that argument anyway would 
not work in the graph theoretical frame work that we consider here. 
\footnote{Most games have game graphs, not game trees. There can be various 
ways to reach the same position. Indeed this happens in most games.}

\begin{thm}[Zermelo's theorem]
There is either a win strategy for $V$, or a win strategy for $W$ or a draw strategy. 
\end{thm}
\begin{proof}
As $G$ is a finite graph and game events are maximal simple paths in $G$,
there are only finitely many game events that are possible. To each event, there is a definite value "win" or "loss"
for either player. There are 4 logical possibilities.  The possibility that is a win strategy for both 
can not occur: indeed, either $V$ can limit the options so that no game event from $W$ to $A$ exists, 
or $W$ can limit their options so that no game event from $V$ to $A$ exists. Or then both can 
construct together a strategy $H$, where $V$ is allowed to trim outgoing edges from $V$ and 
$W$ is allowed to trim edges outgoing from $W$ so that $H$ has no game event ending in $A$. 
\end{proof}
\index{Zermelo's theorem}
\index{proof of Zermelo's theorem}

\paragraph{}
Zermelo's theorem have been described in \cite{Zermelo1913}. It needed
adaptations and corrections like \cite{Kalmar1928}. It is discussed in \cite{SchwalbeWalker}. 
A proof of Zermelo's theorem using induction is in \cite{MaschlerSolanZamir} 
(Theorem 1.5). We will later remark that one could still use trees, but then we would 
have to look at the {\bf self-avoiding walk tree} of the game graph. 
\index{Zermelo's theorem}

\paragraph{}
In general, there is no procedure to find a win strategy in general, except
by exhaustive case study \cite{BeckCombinatorialGames}. Indeed, the task of 
finding the god number is in general a {\bf NP complete} problem:
\index{NP complete} 
the solitaire game in which one has to find a maximal 
complete sub-graph in a graph or the problem to find a Hamiltonian path in 
a graph are fundamental NP complete problems \cite{Karp}. 
A concrete case of a solitaire that is NP complete is the $n \times n$ generalization 
of the 15-puzzle. This class of games are discussed later on.
\index{15 puzzle}
\index{NP complete}

\paragraph{}
The {\bf god number} for a given strategy $H$ is
the maximal length $|\gamma|$ that a game event $\gamma$ in $H$ can have. 
The length $|\gamma|$ of $\gamma$ is assumed to be the number of edges in $\gamma$ 
emanating from $V$. It is the number of play moves of player $V$.
\index{god number for two-player games} 
The length of the path as a ``measure of loss" and is defined in such a way 
that {\it if white moves and mates in n} means that V="White" and that the god number is n. 
For a problem {\it black moves and mates in n} means V="Black" and that the god number is n.
The assumption $v \in V$ assumes that $V$ is the {\bf first player}. The situation where the
``second player" starts, can easily be realized by renaming.
\index{first player}
\index{second player}

\paragraph{}
To place this paper into the given mathematical landscape, we remark that
our approach is elementary, geometric and entirely {\bf graph theoretic}. 
It could be formalized in \cite{LeanOrigin}, as we only use notions from graph theory:
directed graphs, subgraphs and paths. 
Other notions exist of course. There has been in particular other set-ups for
"strategy". The book \cite{MaschlerSolanZamir} for example defines
a ``strategy" a rule which tells in each situation where to go. There are
probabilistic set-ups also, where a strategy is a probability measure on the
possible moves. We mention some history in a separate section and especially comment
that traditional game theory developed in the 1950ies by students of Albert Tucker 
focused on various forms of cooperation, which is necessary when looking at multi-player games.
The difficulty of cooperation and collusion through the formation of coalitions already 
starts with three players. 

\paragraph{}
All graphs are simple (no multiple connections, no loops) and unweighted (no additional properties are
attached to edges or vertices).
The graph data $(V,E,v,A)$ or the bipartite directed graph data $(V \cup W,E,v,A \subset S)$ determine
the game entirely. The scope covers all deterministic sequential zero-sum 2-person games. It is especially
part of {\bf combinatorial games} \cite{BeckCombinatorialGames,LessonsInPlay}. 
There is little relation with ``graphical games" \cite{GraphicalGames}. For solitaire games,
the god number problem is a {\bf combinatorial sorting problem}. Many game graphs are Cayley graphs of 
groups so that it is also part of {\bf combinatorial group theory}. 
Puzzles and games deliver an accessible gateway to mathematics. Together with history as a gateway, 
it is a nice selling point. Recent example books to this would be \cite{Moscovich,Orlin}.
\index{topological sorting problem}
\index{gateway to mathematics}
\index{graphical game}
\index{combinatorial game} 
\index{combinatorial group theory}

\paragraph{}
The setup could naturally be generalized to {\bf $n$-player games}: take a finite simple directed 
{\bf $n$-partite graph} $G=(V=V_0 \cup \cdots \cup V_{n-1},E)$, where edges $(a \to b) \in E$ 
have the property that $a \in V_k$ implies $b \in V_{k+1 \; {\rm mod} n}$. Players play cyclically. 
The {\bf initial position} $v \in V_0$ and the {\bf winning set} $A \subset S$ are defined as before. 
The {\bf winning set} $A$ is assumed to be a subset of $S=\{ x \in V, {\rm outdeg}_G(x)=0 \}$.
A {\bf game event} is a maximal simple path $(v_0, \cdots, v_m)$ in $G$ that starts at $v_0=v$. 
A {\bf $V_i$ strategy} is a sub-graph $H$ of $G$ that contains $v$
such that only edges emanating from $V_i$ are removed. Formally, $v \in V(H), E(G) \setminus E(H) \subset E_i$,
where $E_i = \{ (a,b) \in E, a \in V_i \}$.
A {\bf game event} in $H$ is a maximal simple path in $H$ that starts at $v$. A 
non-empty proper subset of players $I$ is called a proper {\bf coalition}.
\footnote{We leave away ``proper" to avoid a pleonasm: the set of all players is not a coalition
in colloquial language.}
A strategy is called a $V_I$-strategy, if $v \in V_I, E(G) \setminus E(H) \subset \bigcup_{i \in I} E_i$. 
The {\bf length} $|\gamma|$ of a game event $\gamma=(v=v_0,v_1, \dots, v_m)$ is the number of $V_0$ edges in $\gamma$. 
A game event represents a {\bf win for player $V_k$} if the final edge $(v_{m-1},v_m)$ in the path $\gamma$ 
satisfies $v_{m-1} \in V_k$ and $v_m \in A_k = A \cap V_{k+1 \; mod n}$. 
Coalition $I$ is winning if it has a strategy, forcing every game event into 
$A_I =\bigcup_{i \in I} A_i = \bigcup_{i \in I} (A \cap V_{(i+1) \; {\rm mod} \; n})$.
It is {\bf minimally winning} if no proper subset $J \subset I$ is winning. 
An $n$-player game is a {\bf draw}, if no coalition has a winning strategy
For $n$-player games, the Zermelo trichotomy needs to be replaced with a study of {\bf winning coalitions}.
By definition, every game is either a draw or then winnable for at least one coalition.
The set of all strategies is richer because {\bf collusion} is possible: it can happen for a 
3-player game already that the game is winnable for any coalition of size $2$ but that is not winnable 
for any single player. Two players can band together to make sure that the third will lose.
Since players can help each other in 3 or more player setups, we stick here to two- player games. 
For a winning coalition $I$, the god number is 
${\rm God}_I(G,v) = {\rm min}_{H \in \mathcal{H}_I^{{\rm win}}} \max_{\gamma \in \mathcal{E}(H)} |\gamma|$,
where $\mathcal{H}_I^{{\rm win}}$ is the set of winning I-strategies and 
$\mathcal{E}(H)$ the set of game events in $H$. 
\index{coalition}
\index{winning coalition}
\index{n-player game}

\paragraph{}
We taught Chat GPT these axioms and asked it to construct a simple example of a {\bf 3 player game},
where only coalitions of two players can win. Here is the example the machine came up with. We could
only think of the {\bf truel} in the movie {\it The good, the bad and the ugly}. So lets call this the ``truel''
example.  It has ${\bf God}(G,v)=1$ for every winning coalition of size $2$. This takes into account that 
only the number of moves of the player $V_0$ are counted.
\index{three player game}
\index{truel}

\paragraph{}
The following paragraph is a machine generated result. We 
slightly edited the answer of the machine. Even so our set up is non-standard, 
the machine understood it well and creatively generate an example to the prompt:
{\it "Please construct a small game example to the attached axiom system,
where collusion occurs, meaning that there is no win for a single player but a coalition of two
can force a win for them"}. The machine was fed the above paragraph about multi-player games and 
also had previously been taught the axioms in the special case of one player or two player games, to 
have more context. 

\paragraph{}
{\bf Example:} (``truel")
\footnote{Generated by GPT 5.4, prompted on May 8, 2026. It was manually slightly formatted and shortened.}
{\it As defined, players $V_0,V_1,V_2$ move cyclically 
$V_0 \longrightarrow V_1 \longrightarrow V_2 \longrightarrow V_0$.
The game is defined by the finite directed $3$-partite graph with
$V_0=\{v,t_2\}, V_1=\{p_1,p_2\},V_2=\{q,t_1\}$ and directed edges 
$v\to p_1, \qquad v\to p_2, p_1\to t_1, \qquad p_2\to q, q\to t_2$. 
The terminal equilibrium set is $S=\{t_1,t_2\}$. We pick $A=S$ (which is allowed).
The game graph is visualized in the Figure below. Under the normal-play convention, a 
terminal vertex in $V_{k+1}$ is a win for the player $V_k$ who made the last move. Hence $t_1 \in V_2$
is a win for $V_1$, because the last move is $p_1\to t_1$ made by $V_1$. Similarly, 
$t_2\in V_0$ is a win for $V_2$, because the last move is $q\to t_2$,
made by $V_2$. Thus the individual terminal win sets are 
$A_0=\emptyset,\qquad A_1=\{t_1\},\qquad A_2=\{t_2\}$. 
a) We first check the coalitions.
(i) First, the coalition $I=\{V_0,V_1\}$ can force a win. Player $V_0$ restricts the first move to
$v\to p_1$.  Then player $V_1$ moves $p_1\to t_1$.
The game ends at $t_1$, which is a win for $V_1$, hence a win for the coalition $\{V_0,V_1\}$.
(ii) Second, the coalition $I= \{V_0,V_2\}$ can force a win. Player $V_0$ restricts the first move to
$v\to p_2$.  Then the move $p_2\to q$ is forced, and player $V_2$ moves
$q\to t_2$. The game ends at $t_2$, which is a win for $V_2$, hence a win for the
coalition $\{V_0,V_2\}$. 
(iii) Third, the coalition $\{V_1,V_2\}$ can force a win. This coalition cannot
restrict the first move of $V_0$, so both possibilities
$v\to p_1 \qquad\text{and}\qquad v\to p_2$ remain. But this is harmless. If the play goes through $p_1$, then
$p_1\to t_1$ ends the game at a win for $V_1$. If the play goes through $p_2$, then
$p_2\to q\to t_2$ ends the game at a win for $V_2$. Therefore every possible game event ends
in a win for one of the coalition members. Hence $\{V_1,V_2\}$ is also a winning coalition.
b) Now, we check that no single player can force a win. The player $V_0$ has no terminal win position in $S$, 
since $A_0= A \cap S = \emptyset$. Thus $V_0$ cannot force a win.
The player $V_1$ cannot restrict the first move of $V_0$. If $V_0$ chooses
$ v\to p_2$ then the play continues $p_2\to q\to t_2$
which ends in a win for $V_2$, not for $V_1$. Hence $V_1$ cannot force a win.
Similarly, the player $V_2$ cannot restrict the first move of $V_0$. If
$V_0$ chooses $v\to p_1$ then the play continues $p_1\to t_1$ which ends in a win for $V_1$, not for $V_2$. 
Hence $V_2$ cannot force a win. 
We have now proven that this is a simple $3$-player game in which no single
player can win, but where coalitions of 2 players can win.}

\begin{figure}[!htpb]
\scalebox{0.19}{\includegraphics{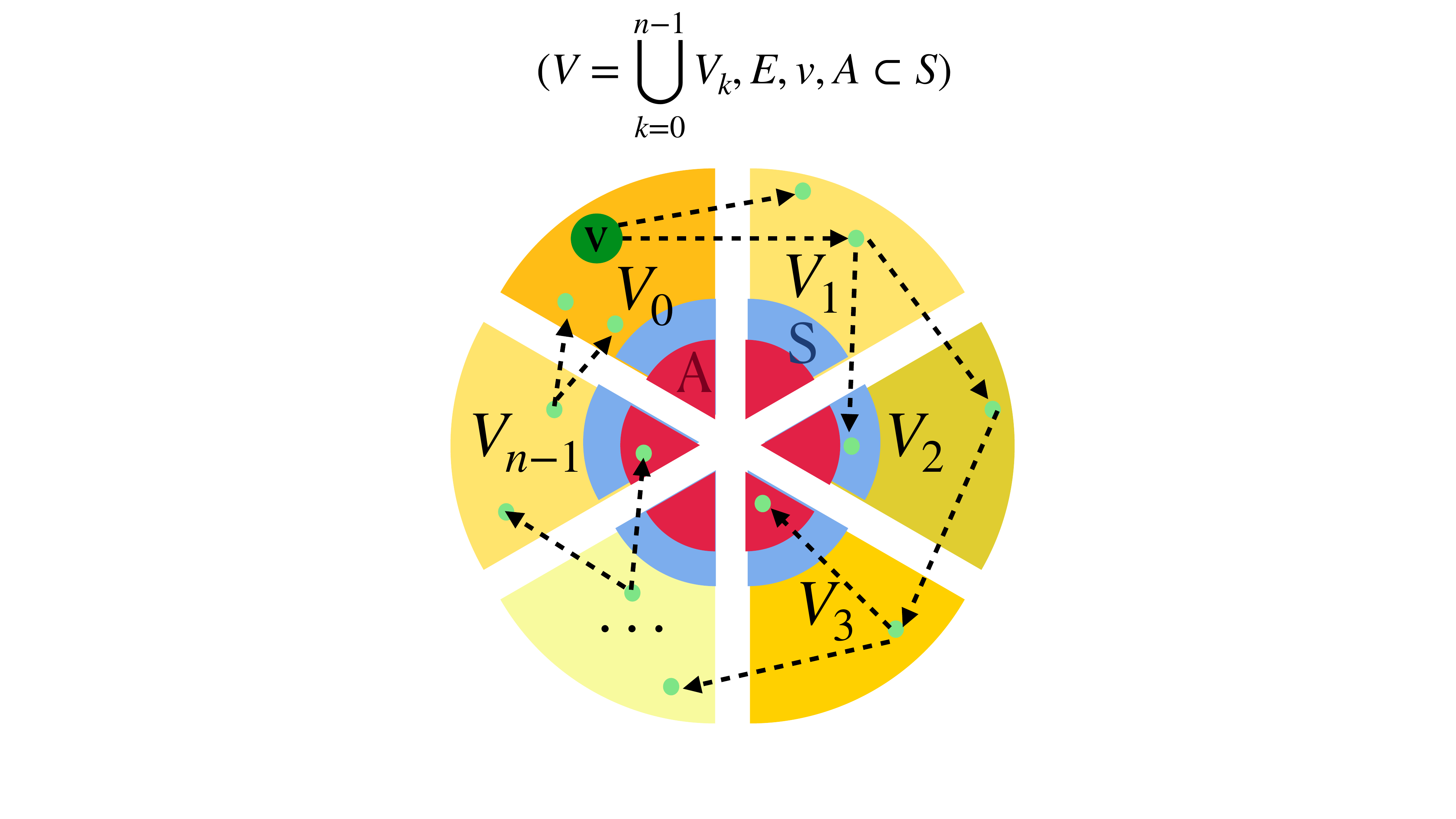}}
\label{n player game}
\caption{ 
For a multi-player game with more than 2 players,
the Zermelo analog of coalitions is already more complicated.
Coalitions can form.
}
\end{figure}

\begin{figure}[!htpb]
\begin{tikzpicture}[
  scale=1.2,
  every node/.style={circle,draw,minimum size=8mm,inner sep=1pt},
  vzero/.style={fill=yellow!35},
  vone/.style={fill=orange!35},
  vtwo/.style={fill=orange!60}
]
  \node[vzero] (v)   at (0,0)   {$v$};
  \node[vzero] (t2)  at (6,0)   {$t_2$};
  \node[vone]  (p1)  at (2,1.5)  {$p_1$};
  \node[vone]  (p2)  at (2,-1.5) {$p_2$};
  \node[vtwo]  (t1)  at (4,1.5)  {$t_1$};
  \node[vtwo]  (q)   at (4,-1.5) {$q$};
  \draw[->,thick] (v) -- (p1);
  \draw[->,thick] (v) -- (p2);
  \draw[->,thick] (p1) -- (t1);
  \draw[->,thick] (p2) -- (q);
  \draw[->,thick] (q) -- (t2);
  \draw[dashed,gray] (1,-2.2) -- (1,2.2);
  \draw[dashed,gray] (3,-2.2) -- (3,2.2);
  \draw[dashed,gray] (5,-2.2) -- (5,2.2);
  \node[draw=none,fill=none] at (0,-2.7) {$V_0$};
  \node[draw=none,fill=none] at (2,-2.7) {$V_1$};
  \node[draw=none,fill=none] at (4,-2.7) {$V_2$};
  \node[draw=none,fill=none] at (6,-2.7) {$V_0$};
  \node[draw=none,fill=none] at (4,2.3) {\small terminal};
  \node[draw=none,fill=none] at (6,0.8) {\small terminal};
\end{tikzpicture}
\label{3 player game}
\caption{ 
A simple 3 player game, where only coalitions of 
2 players can win. This example (as well as the graph) were machine generated.
The machine was only fed the above simple axiom system, without 
giving any additional help. We wanted a simple example. }
\end{figure}
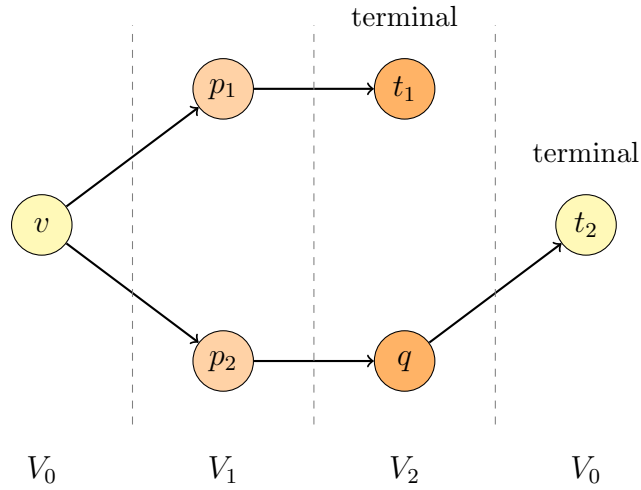

\begin{figure}[!htpb]
\scalebox{0.99}{\includegraphics{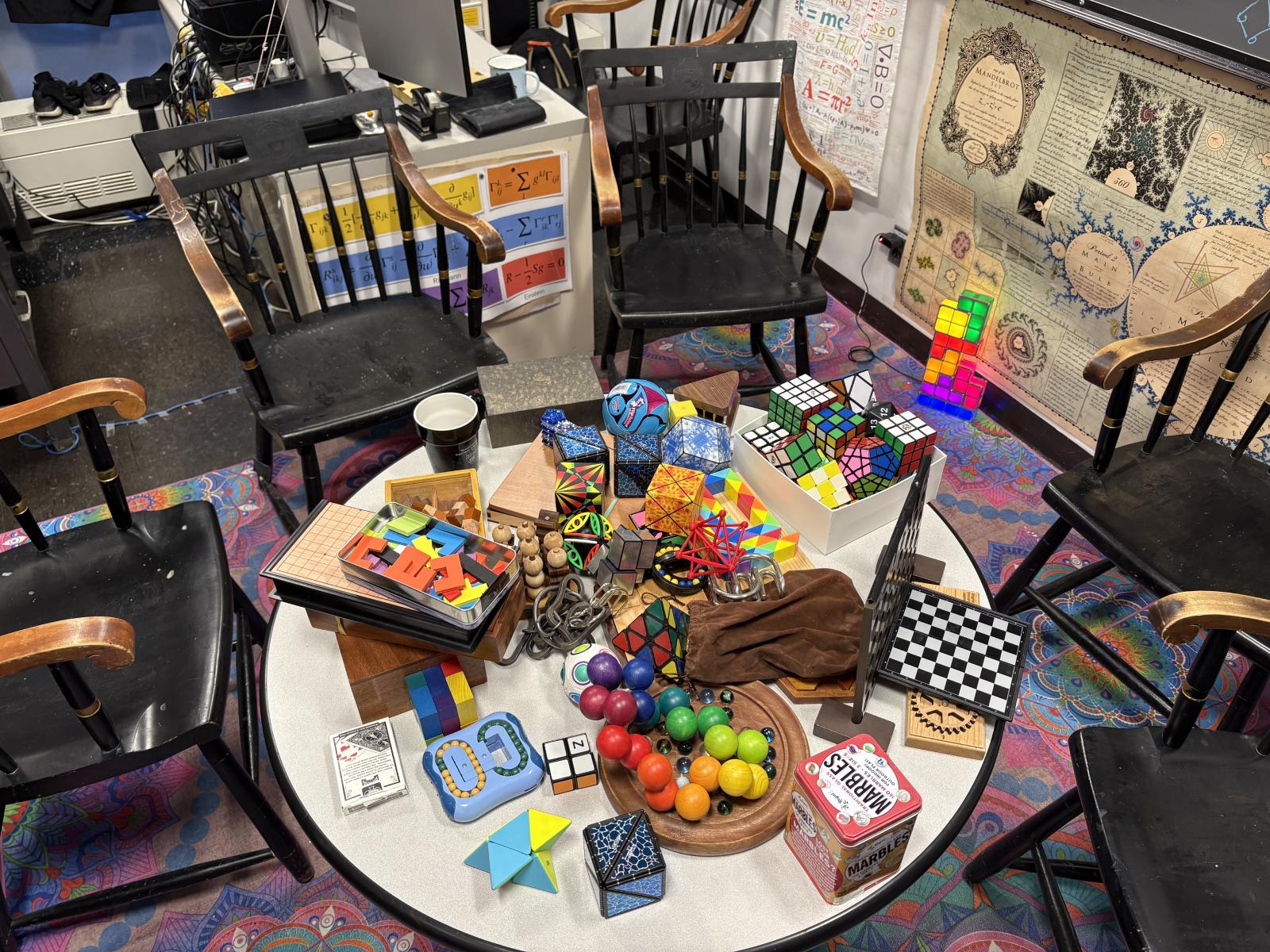}}
\label{game}
\caption{
An exhibit of a few toys and games, we have played with when working
on this project.
}
\end{figure}

\section{Introduction}

\paragraph{} 
{\bf God numbers} quantify optimal solutions for {\bf deterministic}, {\bf finite}, {\bf zero sum},
{\bf sequential} games. The term ``god number" has appeared first in the context of the 
{\bf Rubik cube puzzle} \cite{Singmaster1981,CubedRubik,Joyner}, where the 
number is now known to be $20$ 
\index{Rubik cube} 
in the half turn metric \cite{Rubik20}. The term was popularized in the early 1980ies by ``cubists"
and mathematician {\bf David Singmaster}, who distributed notes on the ``magic cube" in 1979 
already. (They were published in 1981 \cite{Singmaster1981}).
\index{zero sum}
\index{deterministic}
For a two-player game, the {\bf god number} is a
{\bf minimax optimization problem} for the game graph.
The prototype is a {\bf chess problem} ``mate in n" has god number $n$. 
\index{chess}
\index{god number}

\paragraph{}
A problem {\it ``draw in n"} for example tells that for some strategy, the maximal path length
in it leads to a draw. In a game with god number $n$, the analysis of a game
only needs to to be done in a {\bf neighborhood graph}  $B_n(v)$ of $v$. 
\footnote{This is a ball of radius $n$ in the graph distance of the directed graph. Note that
$w \in B_1(v)$ does not imply necessarily that $v \in B_1(w)$. }
``Draw" either forces a game event to reach $S \setminus A$ or then ends in a dead end,
not being able to continue without crossing itself. 
``Draws" can also be a ``theoretical draws" meaning that they are undisputed end positions.
In chess for example, if we reach a position where both players only have a king and a horse. 
The game does not need to be continued. We could bake this into the geometry by 
removing any outgoing edges at such points of the game graph. As more advanced players are, 
as more they can prune the game graph and also label positions $w$ as points in $A$ which 
are clearly winning. Pruning means to remove any outgoing edges from $w$.
\index{pruning the game graph}

\paragraph{}
The class of games that can be described as such is close to \cite{Zermelo1913}, who 
would rather see it as a playground within set theory rather than geometry. Zermelo's article was 
only three and a half pages long and needed various clarifications \cite{SchwalbeWalker}.
Restricting to graph situations covers most puzzles and games. The geometric language 
of graphs allows to use intuitive geometric notions.

\paragraph{}
The geometric setup models the usual assumptions:

\begin{itemize}
\item The term {\bf zero sum} means that if one players wins the other playes loses.
It is not possible that both win. If both lose, the game is a draw. 
\index{zero sum}
\item {\bf Deterministic} means that both players at any time have the same full access to 
information, that no external randomness nor third party input can happen. 
\index{deterministic}
\item {\bf Sequential} means that players always take alternate turns. This is encoded
in the bipartite structure of the graph. It also is an ingredient to deterministic. 
\index{sequential}
\item {\bf Finite} means that the entire game has only finitely man states. 
The game is so encoded as a {\bf finite simple graph}, where the nodes are 
the possible configurations and where the oriented edges are the possible moves. 
\index{finite}
\item {\bf Finite events} are forced by the assumption that all {\bf game events} describing 
game events must be maximal simple paths starting at $v$. The case for having no self intersections
is done such that there are only finitely many game events. 
\index{game event}
\item {\bf No cycles} is encoded in the simplicity of the game events. 
It is an important assumptions we want a finite set of game events. The convention in chess,
limiting the number of repetitions can easily be bent to this. If a position appears three times,
just cut away the loops which have led to it and stop the game after the first repeat. 
\end{itemize}
\index{3 position game rule}

\paragraph{}
The standard $3 \times 3 \times 3$ {\bf Rubik cube} $G$ is is a popular solitaire game.
\cite{Singmaster1981, CubedRubik,Joyner}.  Its god number ${\rm God}(G)$ is known to be 20. \cite{Rubik20}.
As for all Cayley group games, it is the diameter of the Cayley graph of a finite 
simple graph with $R=8! 12! (3^7) (2^{10})$ elements. Since this is larger than $2^{65}$,
we can not hope to draw out the entire game graph.

\paragraph{}
In {\bf chess}, the number of nodes of the graph has an upper bound $10^{47}$. 
Claude Shannon estimated the number of game positions as $64!/(32! 8!^2 2!^6) \sim 10^{43}$ without promotion
\cite{Shannon1950}. The best upper bound appears to be $2* 10^{40}$ \cite{Chess2015}.  
It is not the size but the complexity of this game which which often makes it 
impossible to compute the god number of chess yet. To illustrate this, look at the
$5 \times 5 \times 5$ Rubik cube (which is called "professor cube"). This game 
has $24!^3/24^{12}$ positions, more than the number of chess positions. 
\index{professor cube}
\index{Claude Shannon}

\paragraph{}
The size of the neighborhood graphs $B_r(v)$ of the initial position $v$ in chess grows fast and 
is known as the OEIS number A048987. After 1 play of white we have 20 positions, after a play of
white and black, there are 400. The current status is after 15 half moves.
On smaller chess boards or with initial conditions
with very few figures (chess puzzles), we can draw out all the positions. 
\index{Shannon estimate}
\index{professor cube}
\index{chess puzzles}

\paragraph{}
One reason to restrict to 2 players is that for more than two-players, like a 
{\bf 3 player chess}, the possibility of {\bf collusion} is possible already.  
\index{collusion}
With more players, one also has to clarify the order in which players play like taking turns in a
circular way. A geometric description is not as simple as for 2-player games because
Zermelo's theorem does not generalize in an obvious way. It is not true that either
one of the player wins or that we have a draw. We can have a perfectly symmetric case
with 3 players where one of the players needs to pick a strategy favoring one of its
opponents and so decide the winner. The game itself does not determine that as in the
case of  2-player setup. 

\paragraph{}
Also the sequential assumption is important. If the players could move simultaneously, 
like in a {\bf rock-paper-scissors} game or {\bf prisoner dilemma} game, the two 
players would have both to move at the same time.
Since in that case only partial information is available, {\bf Nash equilibria} conditions
start to matter. If randomness was allowed, it would force to include a probability 
space that to generate the input. This would force us to look at very large graphs
or then look at conditional expectations. 
\index{prisoner dilemma}
\index{Nash equilibrium}
\index{rock paper scissors}

\paragraph{}
We can exploit that many popular puzzles and games naturally are given 
as finite directed graphs (even so the graphs are too large to be drawn out
like in chess). This allows for a rather intuitive axiomatic setup
which is simpler, if one compares it with the existing alternatives.
The geometric frame work only appears to be a limitation if one would insist
that one can construct the entire game graph. But such limitations appear everywhere:
we can not explore the graph of all integers $\mathbb{N}$, where two are connected
if one is a factor of the other for example. This graph exists but we can only see
a small pert of it.  Similarly, 
if games are small enough, we have a reasonable chance to 
oversee the entire graph and measure the god number. 

\paragraph{} For the foundations of game theory, see \cite{NeumannMorgenstern}. 
The geometric setup restricts to games that have entered {\bf popular culture}. 
Remarkably, the number of board 
games or solitaire games that can be described using 
graphs is quite rich. 

\paragraph{}
This paper is a collaborative author effort as part of a reading class
Math/CS 91r, which met weekly in the spring 2026 semester at Harvard. 
Initially, we split things up as follows
among the different group members but there was some overlap. This is
reflected also in the current document as we did not optimize the document
for brevity. 

\begin{itemize}
\item "transposition games" (15-game, color shift, graph, peg) 
\item "group type games", (rubik, clock, hungarian rings) 
\item "pile games" (nim, tower of hanoi, mankala )
\item "chess type games" (chess, checkers, go, morris) 
\item "card type games", (solitaire, card ordering) 
\item "geometric games" (tic-tac-toe, connect 4, Hex)
\end{itemize}

\paragraph{}
As for originality, one should note that graph theoretical approaches to 
games are hardly original, this probably goes back to Cayley already. But it should be 
pointed out again that traditionally one looks at ``game trees".
The closest in literature is the recent book \cite{Fijalkow}, where  a strategy (or {\bf policy}) 
is formally seen a function from 
all paths to edges (see page 16 in the open preprint version of \cite{Fijalkow}).
The language uses {\bf arenas} or ``game structure" which could be interpreted as a game graph. 
There are other major differences in that {\bf quantitative} or {\bf qualitative} 
winning conditions are considered. There are other major differences in that all vertices
have an outgoing edge in \cite{Fijalkow}. We feel that this is very restrictive. Already simple games
like peg solitaires or Tic-Tac-Toe two player games have positions without outgoing edge. One could
mend it by assuming that such positions have self-loops. 
\index{arena}
\index{policy}

\paragraph{}
Zermelo's work of 1913 can be seen as the
beginning of formal {\bf game-tree search}. But Zermelo did not use the language of graphs. 
This should probably be attributed to Kalmar from his 1928/29 work on abstract games \cite{Kalmar1928}. 
Kalmar's article was written in German would use terminology like {\bf ``taktik"} instead of strategy.
Kalmar's work considerably clarifies Zermelo's theorem. 
See \cite{DimandDimand}.
\index{taktic}
\index{tactics}

\paragraph{}
The elementary graph-theoretic formalism presented here is close to finite games 
with perfect information and turn-based multiplayer games on graphs which 
are also called ``arenas" with reachability objectives. We think that
a novelty of the present formulation is the {\bf elementary subgraph interpretation}
of strategies and seeing the associated ``god numbers" as measuring the worst-case 
number of moves by the initial player needed under a winning strategy. The 
simplicity of the set-up is maybe best illustrated in that the theorem of 
Zermelo is immediate and that many commonly known puzzles and games can be 
covered with the language. 

\begin{figure}[!htpb]
\scalebox{0.29}{\includegraphics{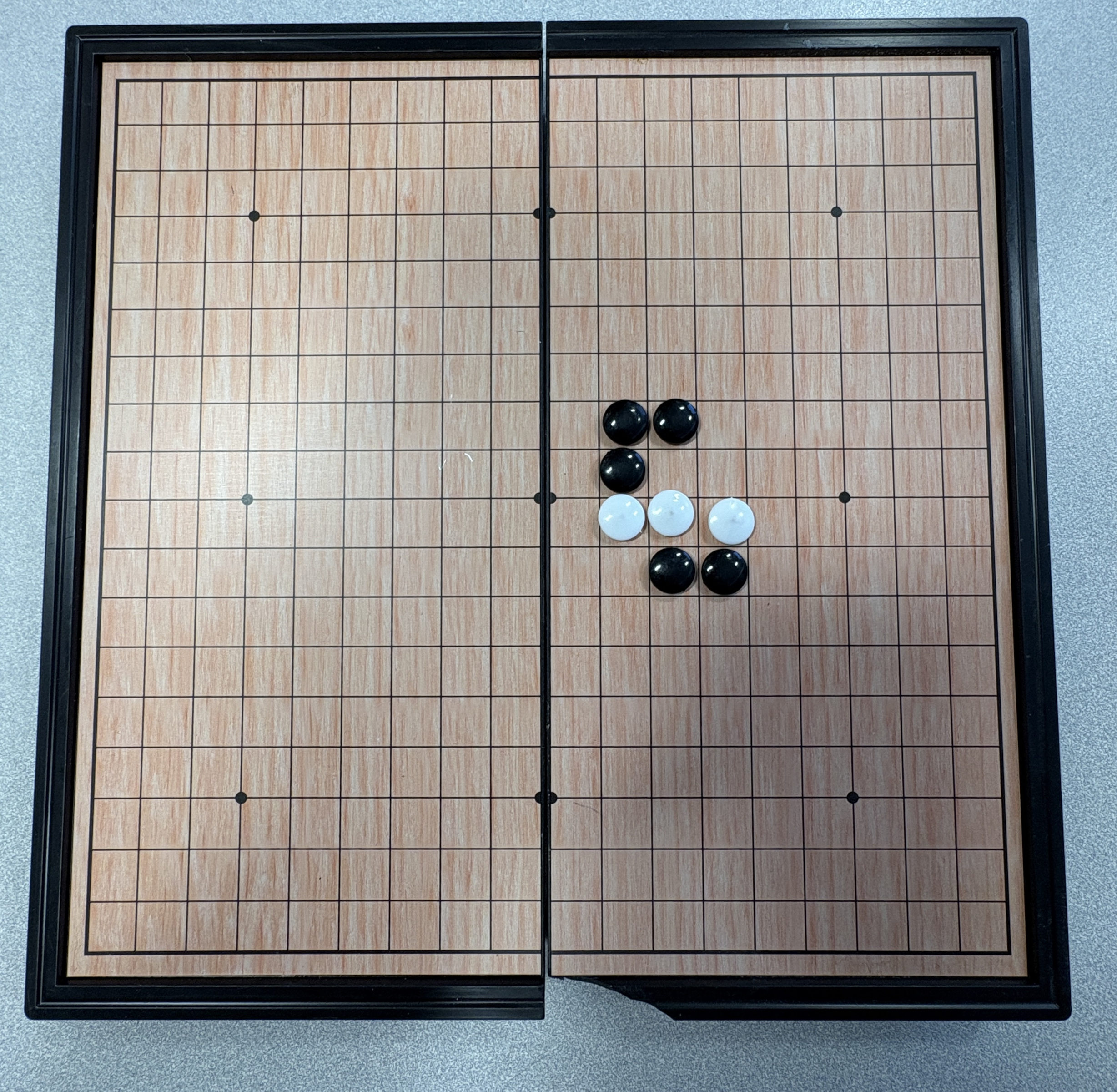}}
\scalebox{0.59}{\includegraphics{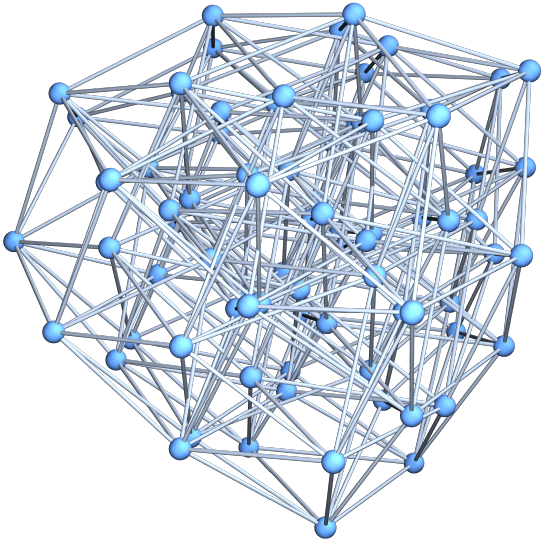}}
\label{To type}
\caption{
{\bf Infinite connect} 5 (Gomoko) can also be played on finite boards like
on a $19 \times 19$ {\bf go game board}. ``Go" itself has been played for more
than 4000 years. Go type games or hex type games are examples of 
{\bf topological games}, where the winning positions are geometrically
defined. $5 \times 5$ go was determined to be a win for the first player
(Black) \cite{Werf2003}, like for {\bf ``microgo"} on $2 \times 2$ board which
requires only to look at $3^4=81$ possible positions. The game graph of 
microgo is seen to the right.
}
\end{figure}
\index{Go game}
\index{Gomoko}
\index{microgo}
\index{infinite connect}

\paragraph{}
In the case of group puzzles, the concept of ``god
number" seems first have emerged when Rubik type puzzles came up in the late 1970ies. 
It definitely appears in the write-up of \cite{Singmaster1981} (page 34) in the form of 
{\bf ``God's algorithm"}.
Mathematicians certainly have contemplated about the minimal number of moves 
to solve a puzzle before.  For 2-player games like chess, the word "god number" is not
used, but most "solve a position in n-moves" belong to god number problems.
We say more about the history in a separate section.
\index{God's algorithm}

\begin{figure}[!htpb]
\scalebox{0.39}{\includegraphics{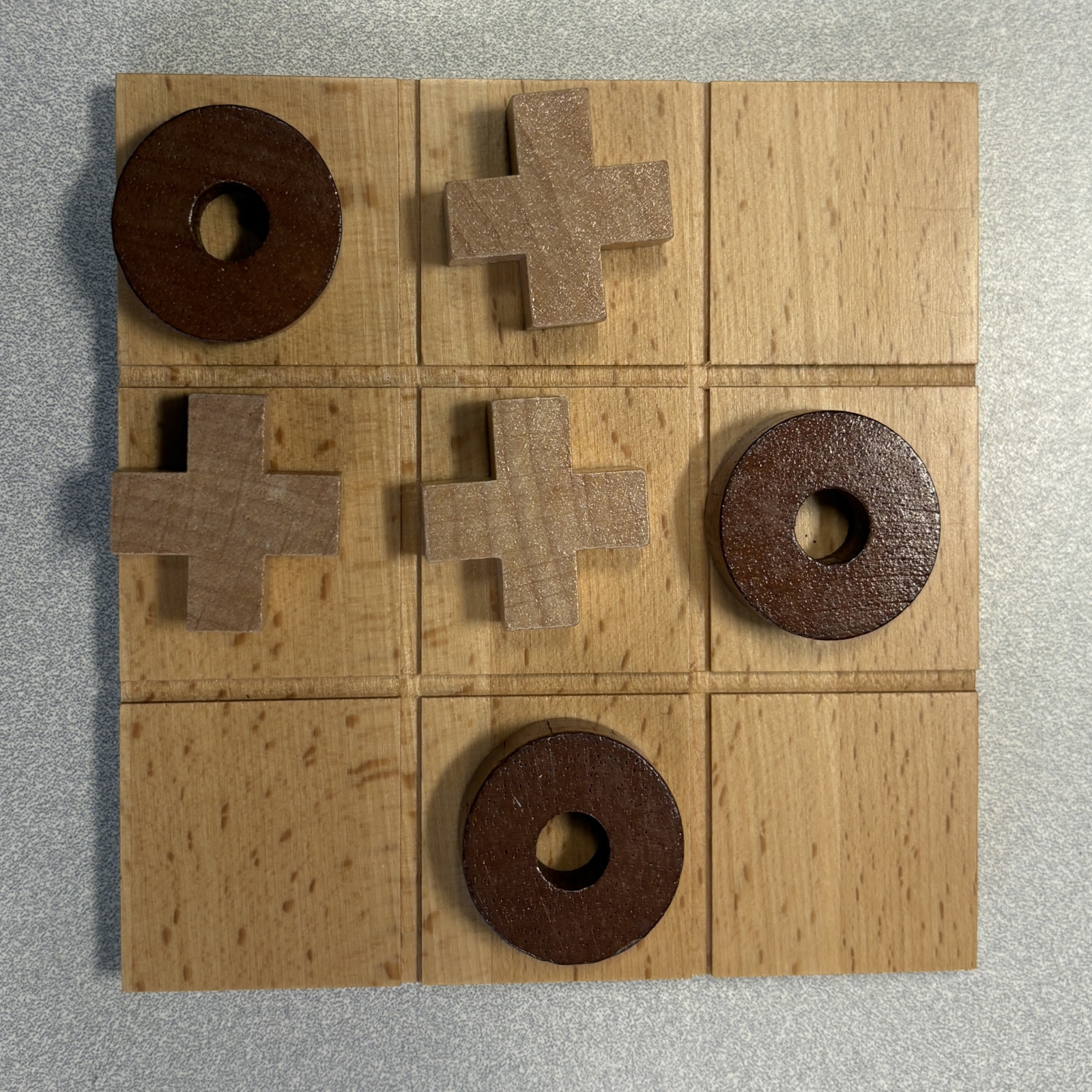}}
\label{Tic Tac Toe}
\caption{
Tic Tac Toe is infamous due to its simplicity and its multi-cultural value, appearing
in blockbuster Hollywood movies like ``WarGames". It is a 2-player
game that is well known to lead to a draw. We can easily map out the entire game graph.
}
\end{figure}
\index{WarGames}

\begin{figure}[!htpb]
\scalebox{0.15}{\includegraphics{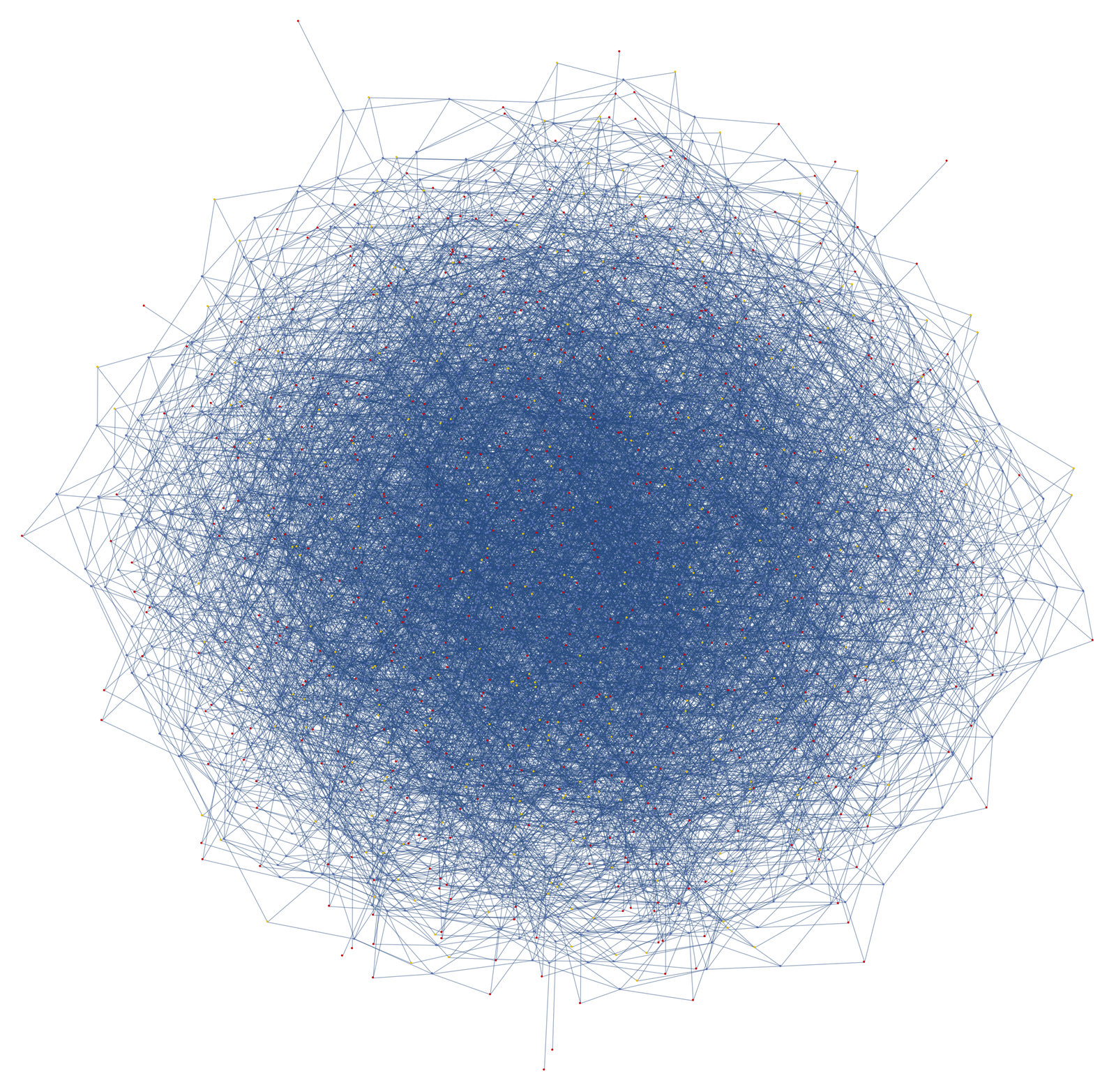}}
\scalebox{0.59}{\includegraphics{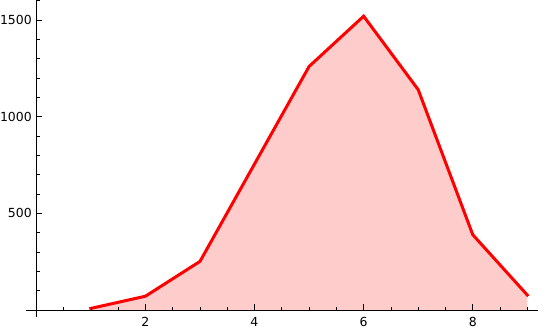}}
\label{tic tac toe graph}
\caption{
The game graph of the standard Tic Tac Toe game has 5478 vertices.
This the number of eligible the $3^9$ possible $3 \times 3$ matrices with 
entries in $\{0,1,2\}$ (we need to look a subset because the number of white stones
is always one larger or equal than the number of black stones). 
The graph ball sizes $B_r(0)$ starting from the initial $v=0$ position
are $1, 10, 82, 334, 1090, 2350, 3870, 5010, 5400, 5478$. The difference sequence
is plotted to the right. It  is known as the {\bf BFS layer profile}. 
The {\bf ball} (=neighborhood graph) $B_9(0)$ covers the game graph. 
The maximal game length is $5$ (if measured in the number of
moves of White). 
}
\end{figure}
\index{BFS layer profile}

\begin{figure}[!htpb]
\scalebox{0.65}{\includegraphics{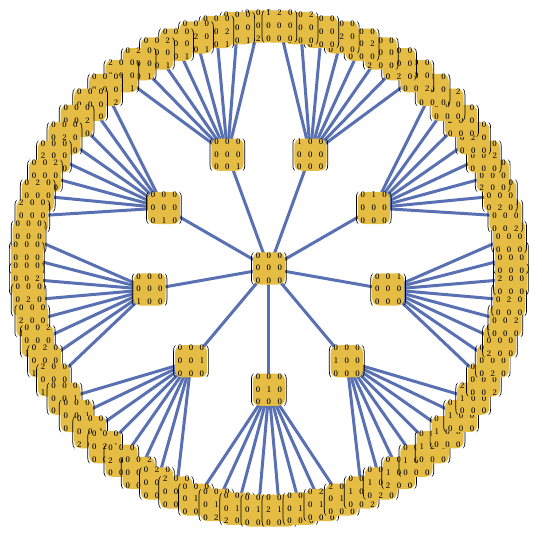}}
\label{tic tac toe Ball}
\caption{
We see the ball $B_1(0)$ of play radius $1$ centered at $0$. 
It consists of 82 possible Tic Tac Toe games in which both players
have played one stone. We count the number of plays of the first
player, as in chess. 
}
\end{figure}

\begin{figure}[!htpb]
\scalebox{0.65}{\includegraphics{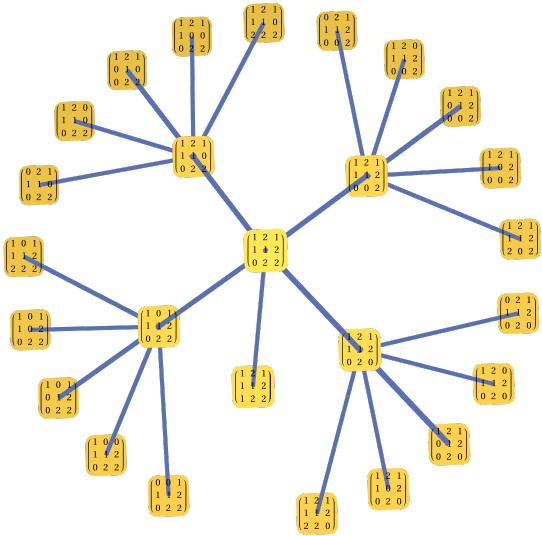}}
\label{tic tac toe Ball last}
\caption{
We see the ball $B_1(0)$ of play radius $1$ centered at 
a winning position for the first player.
}
\end{figure}

\begin{figure}[!htpb]
\scalebox{0.65}{\includegraphics{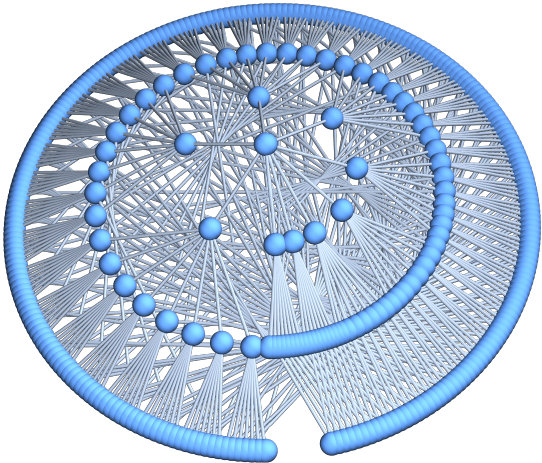}}
\label{tic tac toe Ballof radius 3}
\caption{
We see the graph after an other move by the first player.
}
\end{figure}

\paragraph{}
For a $4 \times 4$ Tic-Tac-Toe game, we already have to weed through 
$3^{16}=43'046'721$ positions. Exploring the full graph is already 
tough with standard computer algebra.

\paragraph{}
Here is an example of a Cayley graph for which the BFS profile is {\bf not unimodal}.
\footnote{This example was found AI assisted.}
The graph $G$ is defined by the generators $48,88,138,-48,-88,-138$ in 
the Abelian group $(\mathbb{Z}_{713},+)$. Its  BFS layer profile 
$(1,6,18,38,66,102,140,138,140,60,4)$ has a small dip $140,138,140$. 
\index{BFS profile}
\index{unimodal}
\index{Abelian group}

\begin{figure}[!htpb]
\scalebox{0.65}{\includegraphics{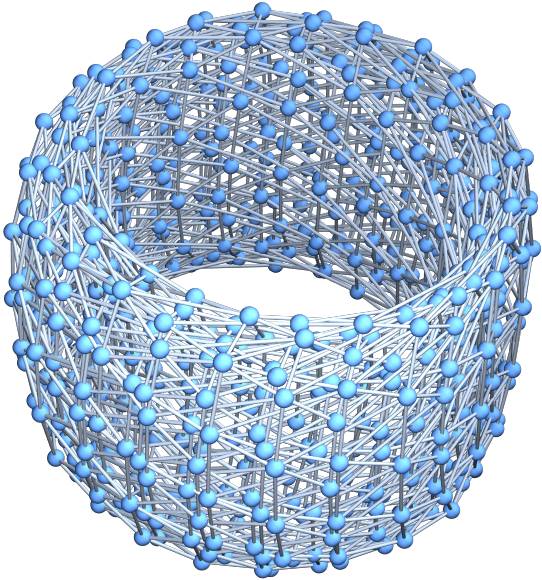}}
\label{bfs}
\caption{
We see the Cayley graph of an Abelian group in which 6 generators 
$a,b,c,a^{-1},b^{-1},c^{-1}$ are given.
Its diameter is 10. The BFS layer profile has a concave up segment.
Most Cayley graphs show a unimodal BFS layer profile.
}
\end{figure}
\index{BFS layer profile}

\vfill
\newpage

\section{Solitaire games}

\paragraph{}
A {\bf solitaire game} or {\bf puzzle} is a finite simple directed graph $G=(V,E)$, 
a vertex $v \in V$ as
{\bf starting position} and a subset $A \subset V$ representing the {\bf solved positions}. 
\index{Solitaire game}
\index{starting position}
\index{solved positions} 
Solved positions also also called {\bf win positions}.  
\index{win positions}
The player starts at a position $v \in V$. The goal is to find the shortest game event
from $v$ to $A$ if it exists. The puzzle moves do not need to be reversible. 
A ``peg solitaire" is an example, where one can not turn back. That game can be played 
on many geometries, the most common being sub-graphs of lattices.
\index{peg solitaire}

\paragraph{}
The {\bf god number} ${\bf God}(G,v,A)$ of a winnable puzzle is the {\bf geodesic graph distance}
${\rm min}_{w \in A} d(v,w)$.  If $v$ and $A$ are not described, the god number $\Gamma(G)$ of the graph 
can still be defined as the {\bf diameter of the directed graph}.  When talking about the god
number of a solitaire game, we usually mean the {\bf game graph diameter}.
This is the solution of a {\bf maximin variational problem}
$$  {\rm God}(G,A) = \max_{v \in V} \min_{\gamma} |\gamma| \; . $$
\index{maximin variational problem}
If $A$ is not specified, the god number of the game graph $(V,E)$ can be given as 
$$  {\rm God}(G) = \max_{v \in V, w \in V} \min_{\gamma: v \to w} |\gamma| \; , $$
where $\gamma$ is a simple path from $v \in V$ to $w \in W$. This is the obviously 
the same than the {\bf graph diameter}, the maximal geodesic 
distance between any two vertices in $V$. Also the diameter is a maximin problem
$\max_{v \in V, w \in V} \min_{\gamma:a \to b} |\gamma|$. 
\index{geodesic distance}
\index{graph diameter}

\begin{figure}[!htpb]
\scalebox{1.0}{\includegraphics{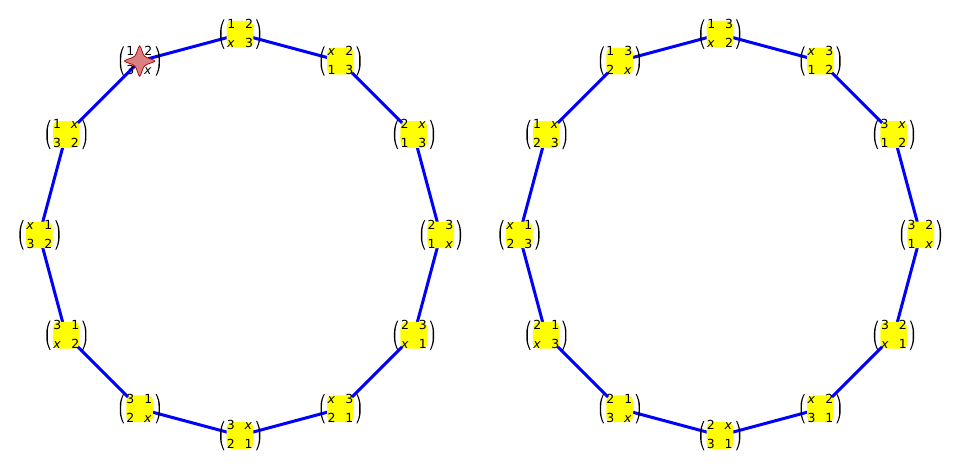}}
\label{3 Puzzle}
\caption{
The {\bf 3-puzzle} game $G$ is an example of a sliding game. It is 
the $2 \times 2$ analog of the {\bf 15 game} which is played on a 4x4 board.
The graph has two connected components; each is a cyclic graph of 
order 12. The god number in the winnable component is ${\rm God}(G)=6$. 
For the 8-puzzle, the god number is 31. For the 15 game, the god number is $80$. 
For the $n \times n$ generalization, its computation is NP complete. 
}
\end{figure} 
\index{15 puzzle}
\index{8 puzzle}
\index{nxn-1 puzzle} 

\begin{figure}[!htpb]
\scalebox{0.85}{\includegraphics{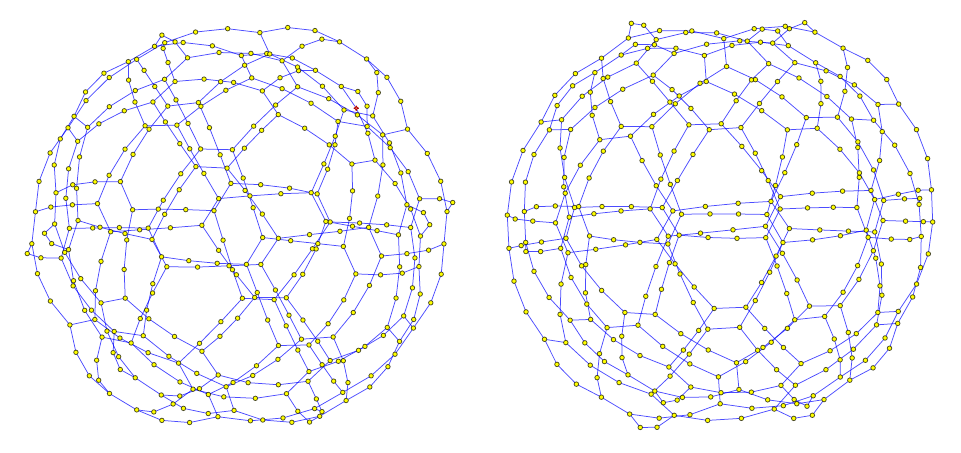}}
\label{5 Puzzle}
\caption{
The 5 puzzle game is the $2 \times 3$ analog of the 15 game.
The graph has two connected components. The god number is 21. 
}
\end{figure}
\index{5 puzzle}
\index{nxn puzzle}

\paragraph{}
In the code section we have included the code which generated the following picture.
We show here only the first connected component, which represent the solvable cases:
the position
$\left[ \begin{array}{ccc} 4 & 5 & x \\ 1 & 2 & 3  \end{array} \right]$ has distance 21 from 
$\left[ \begin{array}{ccc} 1 & 2 & 3 \\ 4 & 5 & x  \end{array} \right]$.

\begin{figure}[!htpb]
\scalebox{0.85}{\includegraphics{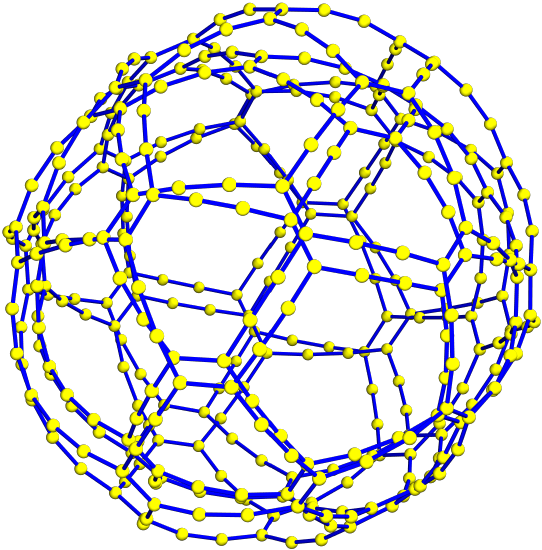}}
\label{5 Puzzle}
\caption{
The solvable connected component of the previous picture is here seen embedded in 3D. 
The game graph has $360=6!/2$ vertices and $420$ edges. All vertex degrees are equal to 
$2$ or $3$. 
}
\end{figure}

\paragraph{}
Whether a solitaire has a win strategy depends on whether 
$v$ is in the same connected component than one of the points $A$. 
If $v$ is disconnected from $A$, there is no solution. There are various mechanisms
which can render $G$ disconnected. The most obvious one is an invariant 
non-constant quantity. 
\index{invariant} 

\paragraph{}
{\bf Examples:}	\\
{\bf a)} For the {\bf 15-puzzle}, the initial position is given by a permutation in $S_{16}$.
The graph has two connected components. The parity of the permutation 
and the signature of the $l_1$-distance of the hole to $(4,4)$ is an invariant. \\
\index{invariant}
{\bf b)} For a finitely presented permutation group with generators that all have even 
signature, the signature of the initial condition is an invariant, the reason being 
that the composition of two even permutations is even. \\
\index{signature}
{\bf c)} In the Rubik cube case, the total rotation of all cube corners has to be a multiple of $2\pi$. 
A situation, where one cube is turned by 120 degrees has been called a {\bf "quark"} configuration in \cite{Golomb}.
Three different quarks model a {\bf Baryon}, two different quarks where one is turned by 120, 
the other by 240 is a {\bf Meson}. We see so caricatures of {\bf Hadron}'s in finite groups. While quarks
can not be realized, Hadron combinations can. 
\index{Rubik cube}
\index{quark}
\index{meson}
\index{boson}
\index{hadron}
\index{signature of permutation}

\begin{figure}[!htpb]
\scalebox{0.8}{\includegraphics{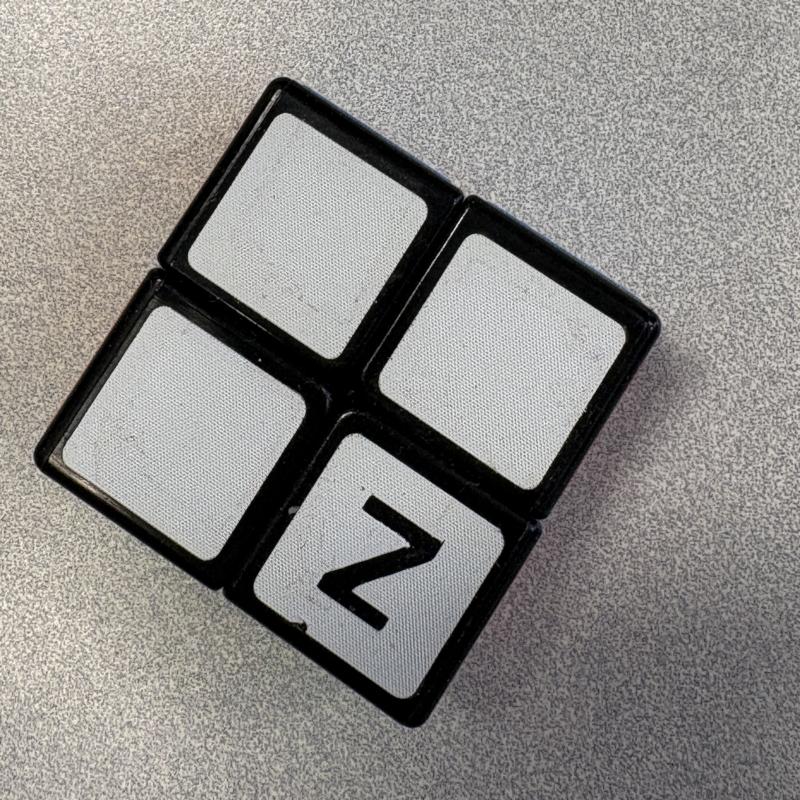}}
\scalebox{0.8}{\includegraphics{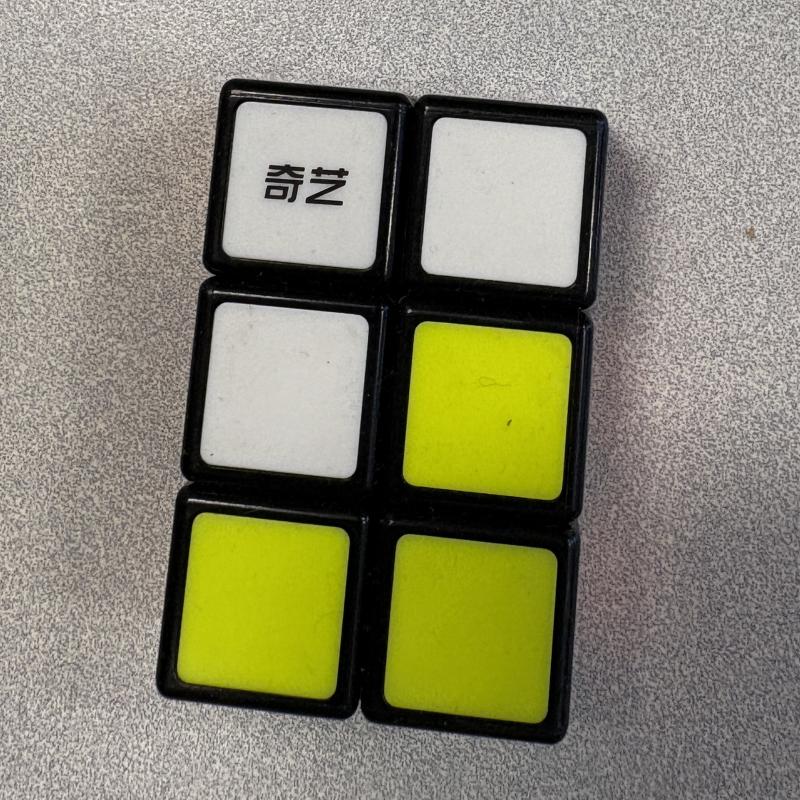}}
\scalebox{0.8}{\includegraphics{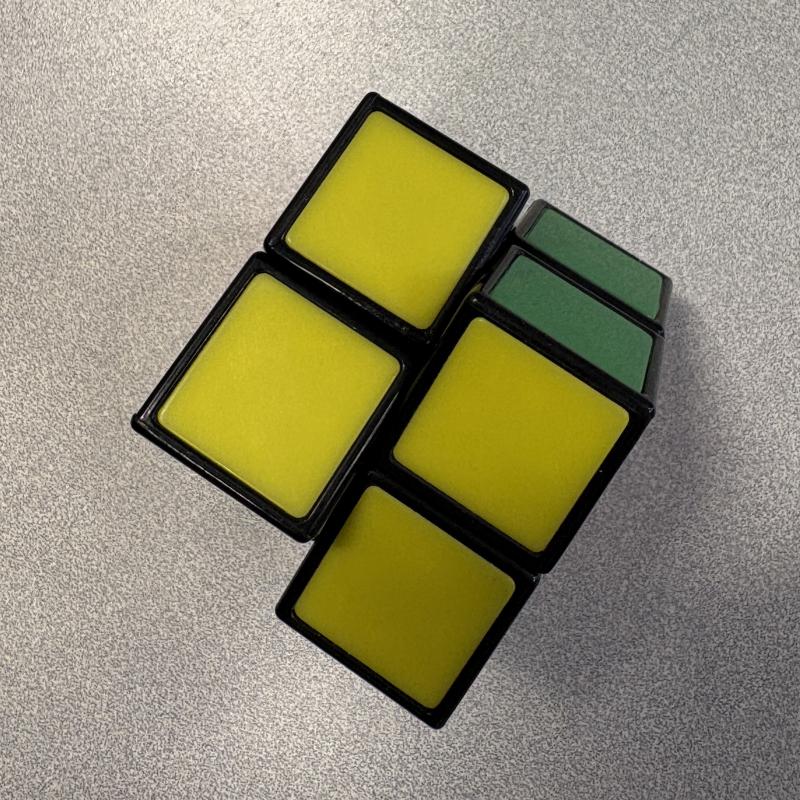}}
\scalebox{0.8}{\includegraphics{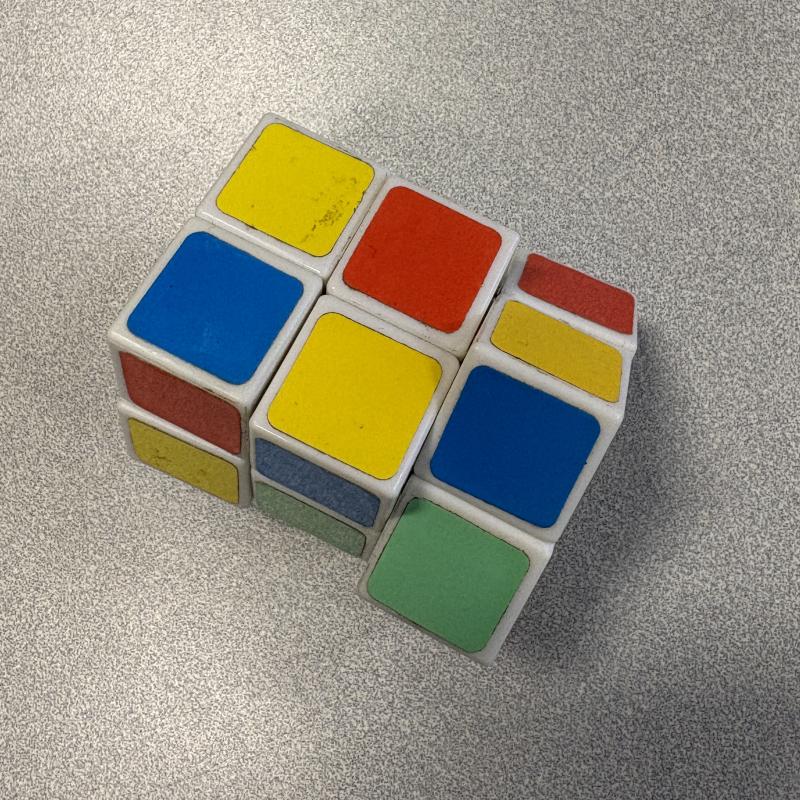}}
\scalebox{0.8}{\includegraphics{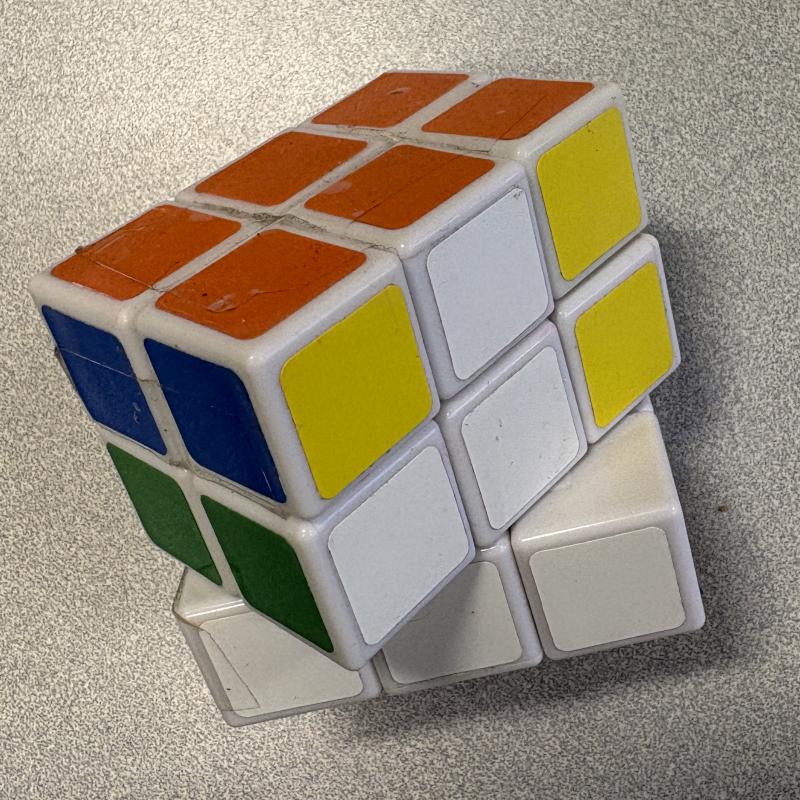}}
\scalebox{0.8}{\includegraphics{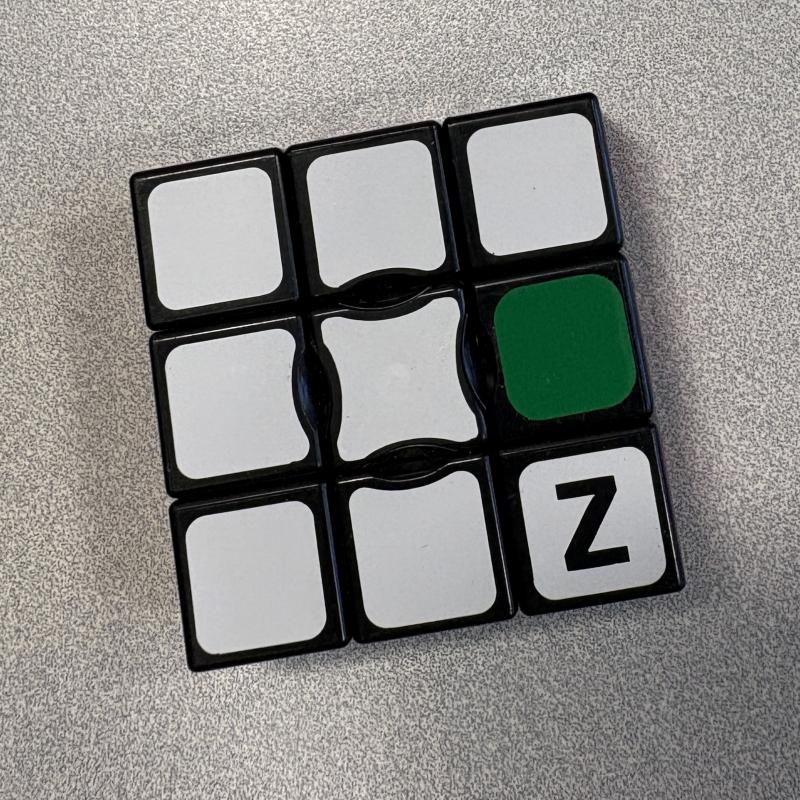}}
\scalebox{0.8}{\includegraphics{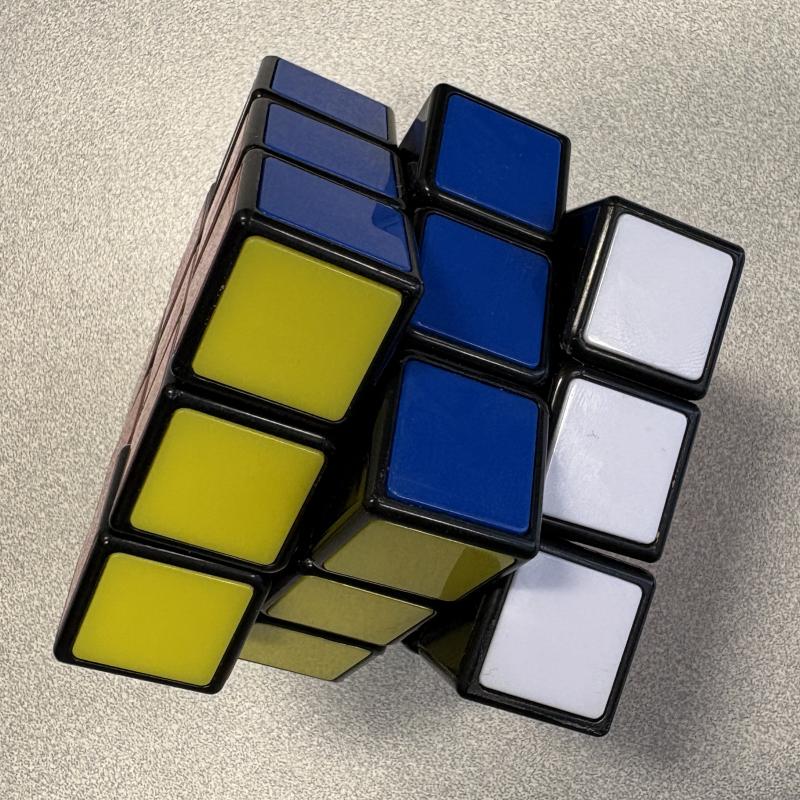}}
\scalebox{0.8}{\includegraphics{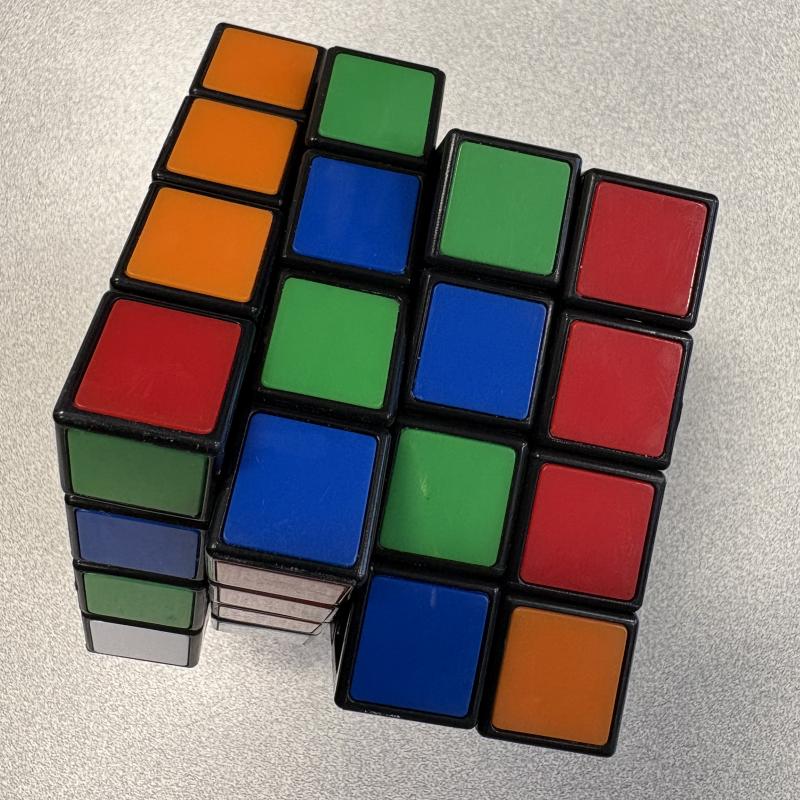}}
\scalebox{0.8}{\includegraphics{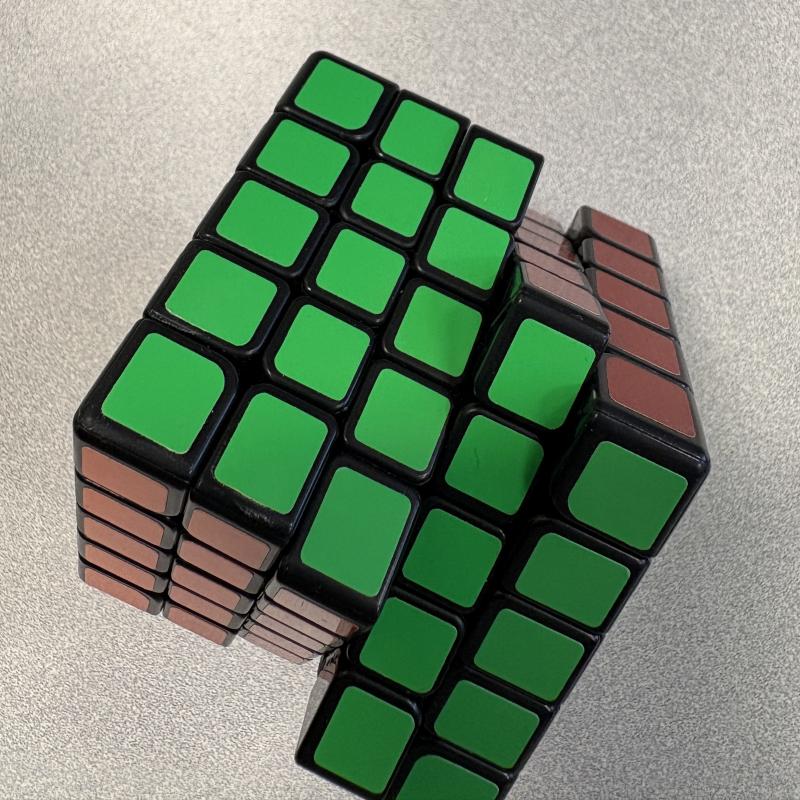}}
\label{Rubik}
\caption{
Rubik puzzles of size $2 \times 2 \times 1$, $2 \times 3 \times 1$, 
$2 \times 2 \times 2$ (pocket cube),
$2 \times 2 \times 3$, $2 \times 3 \times 3$, $3 \times 3 \times 1$, 
$3 \times 3 \times 3$, $4 \times 4 \times 4$ and $5 \times 5 \times 5$ 
(professor's cube). 
}
\end{figure}
\index{Professor's cube}
\index{pocket cube}

\vfill
\pagebreak

\section{Two player games}

\paragraph{}
A {\bf two-player game} is a finite simple directed bipartite graph $G=(V \cup W,E)$ 
together with a vertex $v$ giving the staring position and a subset $A$ of {\bf win positions} 
in $S$, where $S$ is the set {\bf end positions}, the set $S$ of vertices without an
emanating edge. Points in $S$ are game positions, where no move is possible any more. 
The set $S$ of end positions is part of the rules of the game and sometimes artificially
added. For example, it can be a position where naturally no move is possible any more.
It can be a position which is defined to be a defeat like if the {\bf king is mate}, or then a 
position which is declared to be hopeless, like a position where both players do not have 
enough resources to win and where it would be pointless to continue. An example in 
chess would be if both players only have a king and a horse, in which case only a ``help mate"
could render a win. 
\index{king is mate}
\index{hopeless position}
\index{help mate}

\paragraph{}
An {\bf game event} is a maximal simple path $\gamma$ in $G$ that starts from $v$.
The assumption is that on each directed edge, the path needs to follow the direction 
of the digraph. Simple means that no vertex is covered twice.
Pawn moves in chess are examples of directed moves. 
Capturing an opponents piece is an other example of a non-reversible move in chess. 
The {\bf length $|\gamma|$ of $\gamma$} is the number of $V$-vertices in $\gamma$. 
We choose this and not the length of $\gamma$ in order to have ``mate in n moves" be a 
problem with god number $n$.
We also designate a subset $A \subset S$ as the set of
{\bf win positions}. The complement of $A$ in $S$ are declared to be {\bf draw positions}. 
\index{game event}
\index{win positions}

\paragraph{}
A {\bf strategy} is a sub-graph $H$ of $G$. 
A {\bf strategy for $V$} is a strategy that does not remove any edges starting from $W$.
A {\bf strategy for $W$} is a strategy that does not remove any edges starting from $V$. 
A {\bf win strategy for V} is a strategy for $V$ such that every game event ends in $A \cap W$. 
A {\bf win strategy for W} is a strategy for $W$ such that every game event ends in $A \cap V$. 
A {\bf draw strategy} is a strategy that is neither a win strategy for $V$ nor for $W$. 
\index{strategy}
\index{win strategy}
\index{draw strategy}

\paragraph{}
The {\bf god number of a strategy $H$} is the maximal possible length of a game event in $H$.
For a game that can be won by $V$, the {\bf god number} is 
${\rm God}(G)=\min_{H} \max_{\gamma=\gamma(H)} |\gamma|$,
where $H$ runs over all $V$-strategies leading to a win of $V$. The other cases are similar. 
While minimax principles for two-player zero-sum games go back John von Neumann,
the min-max situation in our case is much simpler. The length of a game event is 
the potential loss. To get the god number, we minimize the maximal loss over 
all possible game strategies. 
\index{von Neumann}
\index{minimax principle}
\index{god number of a strategy}
\index{god number of a game} 

\begin{figure}[!htpb]
\scalebox{0.65}{\includegraphics{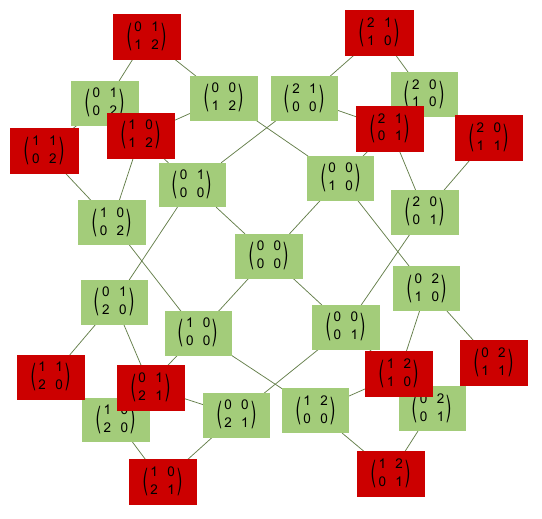}}
\label{2 TicTacToe}
\caption{
we see the simple $2 \times 2$ tic-tac-toe game has no draw. The initial player $V$ always
wins. The god number is 2. The initial player always wins in 2 moves. The empty board is 
the initial position $v$. The vertex set $V$ is the set of boards with an even number of stones. 
The vertex set $W$ is the set of boards with an odd number of stones. $S$ is the set of 
boards where all squares are occupied or where a row, a column or a diagonal have the same
color.}
\end{figure}

\paragraph{}
{\bf Examples:} \\
{\bf a)} Let us look at a chess game on a $2 \times 4$ board, where both parties only have 
king or queen and no other figures are present. This is obviously a draw case. 
The set of configurations with only kings present is declared to be part of $S$ as the
game does not continue there. The initial player $V$ = (White) has a winning strategy. 
It removes all initial plays except capturing the other queen. 
$W$ can now use the empty $W$-strategy because it is forced to recapture the queen. But then 
we are in $S$.  \\
{\bf b)} Lets look at a chess game on a  $4 \times 3$ board again with only king and 
queen on both sides at the corners. Now the initial player $V$ = (White) 
has an obvious win strategy: capture the queen. It is mate in 1. The god number is $1$. 

\begin{figure}[!htpb]
\scalebox{0.5}{\includegraphics{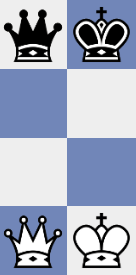}}
\hspace{15mm}
\scalebox{0.5}{\includegraphics{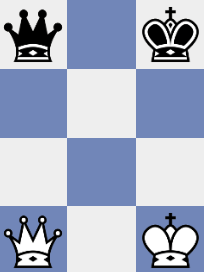}}
\hspace{15mm}
\scalebox{0.5}{\includegraphics{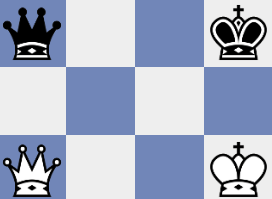}}
\label{King Queen Chess}
\caption{
In the left case, we have a draw. The god number is $1$ because after both parties have
moved, we only have kings left, which is part of $S$. 
In the middle case, there 8 possible legal moves initially, one of them (catch the queen)
leads to a white win, two to black wins and the rest to draws. 
A game engine tells that the god number is $11$ to get to a mate.
In the right case, there is an obvious mate in 1. The god number is $1$. 
}
\end{figure}

\vfill
\pagebreak

\section{Transposition games}

\paragraph{}
In this section we describe a class of solitaire games that come from 
{\bf finitely presented groups} that are generated by {\bf involutions},
and so are {\bf Coxeter group}. In this case it is in general the symmetric group. 
We are not interested so much in the group but in the 
{\bf Cayley graph} of the group generated by a set of transpositions. 
Finding its diameter and so the god number gives an answer to a 
{\bf topological sorting problem}. 

\paragraph
The set of transpositions that generate the subgroup $G$ are defined by the
{\bf transposition graph }
$\Gamma$ which is a finite graph on the vertex set $V=\{ 1, \dots, n\}$, 
the edges being the generators. he problem is to find the god number of the game
graph defined by a connected 
graph $\Gamma$ with vertex set $V$. 
The reason to restrict to connected graphs is that the god number of of the disjoint union 
$\Gamma=\Gamma_1 \cup \Gamma_2$ of two graphs is the sum of the god numbers of the components. 
\index{presentations}
\index{finitely presented groups}
\index{transposition graph}
\index{Coxeter group}

\index{Cayley graph}
Rather than using the language of {\bf presentations}, where the group elements 
are words, we use the {\bf permutation representation} and work in $S_n$, 
the permutation group of $n$ elements. 

\paragraph{}
It is easy to see by induction that the {\bf transposition graph} $\Gamma$ is connected 
if and only if the generated group is $S_n$ and that if $\Gamma$ is the disjoint union
of $m$ graphs $\Gamma_k$, $k=1,\dots,m$  with vertex cardinality $n_k$, then the group $G$
is the Cartesian product $\Gamma_{n_1} \times \cdots \times \Gamma_{n_m}$. 
This is known as the {\bf "connected transposition graph criterion"}.
From the introduction to \cite{Ganesan2013} p 2483: {\it the folklore fact that
$\Gamma$ is connected is equivalent to $G=S_n$.} See also \cite{Kornhauser}. Since we did not find
the statement explicitly proven anywhere (it might be too simple to deserve this), here is the proof. 
\index{Cayley graph}
\index{transposition graph}
\index{connected transposition graph criterion}
\index{Folklore theorem on transposition games}

\begin{thm}[Folklore result on transposition game graphs]
Let $\Gamma=(V,E)$ be a finite simple connected graph with $n=|V|$ vertices.
The edges $(a,b) \in E$ define transpositions in the permutation group $S_n$.
The subgroup of $S_n$ generated by these transpositions is $S_n$. 
\end{thm}
\begin{proof} 
If $\Delta$ is a subgraph of $\Gamma$, it generates a subgroup as less generators
are used. Let $\Delta$ be a spanning tree in $\Gamma$. It still has $n$ vertices, but
now it has $(n-1)$ edges. We made the problem seemingly harder by restricting the permutations
to a smaller set of generators and show that $\Delta$ already generates $S_n$. \\
We use induction with respect to $n$. For $n=1$, there is no
edge and $S_1$ is the trivial group. This is the induction foundation. As for
the induction step, assume that the statement is true for all trees with $n-1$
vertices. Take a tree $\Delta$ with $n$ vertices and pick a leaf $b$, 
(a vertex with vertex degree $1$). It defines a unique edge $(a,b)$ and so a flip $(a,b)$.
Since the tree $\Delta \setminus b$ has $n-1$ vertices, its edges
generate the group $S_{n-1}$ by the induction assumption. 
Take an arbitrary permutation $\pi \in S_n$. (Case 1): If $\pi_n=n$, we can use permutations 
in $\Delta \setminus b$ to order it to $(1,2,\dots, n)$.  
(Case 2): If $\pi_n<n$, assume that $n$ has label $j<n$ so that it is
$\pi=(\pi_1, \dots, \pi_{n-1},j)$, where one of the labels $\pi_j$ is $n$. 
Use permutations from edges in $\Delta \setminus b$ to rearrange this to 
$(1, \dots, n, j)$. Now apply the new flip $(a,b)$ allowed in the full tree $\Delta$ to get 
to $(1, \dots, j, n)$. Again use permutations from edges in $\Delta \setminus b$ to 
do the final adjustment reaching the solved state $(1,2, \dots, n-1,n)$. 
\end{proof} 
\index{leaf}
\index{flip}

\paragraph{}
The Cayley graph $G$ has a natural {\bf geodesic metric} $d(x,y)$ on the graph $G$. 
The god number is the diameter of $G$ with respect to this metric. Note that in 
general, the Cayley graph is a directed graph. It is often assumed however that
if $a$ is a generator in a group, then also $a^{-1}$ is a generator. It is important
to specify this when looking at god numbers. The {\bf half turn metric} on the Rubik
cube for example, allows to turn any face by any angle $90,180,270$. The 
{\bf quarter turn metric} only allows to use turns by $90$ or $270$ degrees.
There is also the {\bf slice turn metric} in which on can turn any layer by any
angle. The god number in the slice turn metric is either 18,19 or 20. The God
number in the quarter turn metric is $26$. The God number in the most commonly 
used half turn metric is $20$. 
\index{geodesic metric}

\paragraph{}
{\bf Breadth-First Search} (BFS) generates the graph and then determine its
god's number. The graph can be built up by building {\bf neighborhood 
graphs} or ball $B_r(v)$ of radius $r$. In a Cayley graph, this can be used to 
find the diameter of the Cayley graph, the distance between a two points
$(a,b)$ is the same than the distance to the points $ca,cb$, where $c$ is an
other group element. To talk in the language of {\bf differential geometry}, Cayley 
graphs are {\bf objects of constant diameter}, similarly as curves of constant
width like the filled in {\bf Reuleux triangle}, which is an object for which 
the diameter can be computed from any point. 
\index{breadth first search} 
\index{half turn metric}
\index{ball of radius $r$}
\index{BFS}
\index{differential geometry}
\index{Reuleux triangle}

\paragraph{}
We can define an other metric on $G$
which is in general different but allows to estimate the god number without
generating the entire graph. This distance depends on the graph $\Gamma$:
If  $M$ be the {\bf distance matrix} of the connected $\Gamma$, 
define the {\bf Gamma distance} on the symmetric group $G$ as 
$d_{\Gamma}(x,y) = [\sum_{k=1}^n M_{x(k),y(k)}/2]$.
\index{distance matrix}

\paragraph{}
$d_{\Gamma}$ is a metric: it is symmetric, commutative and satisfies the triangle inequality. 
It is also a lower bound for the geodesic metric
$d_{\Gamma}(x,y) \leq d(x,y)$. But it is in general smaller. If $\Gamma$ is the 
{\bf octahedron} graph $\Gamma= K_{2,2,2}$ for example, then the god number is $7$, while the 
Hausdorff Gamma distance is $6$. Even simpler is $\Gamma=K_3$, where the graph diameter of $G$
is $2$, while the $d_{\Gamma}$ the Hausdorff distance is $1$. 

\paragraph{}
One of the simplest cases is the triangle $K_3$ (complete graph of 3 elements), where we have $6$
possibilities. $\langle a=(1,2), b= (2,3), c=(1,3)  | a^2=b^2=c^2=(ab)^3=(ac)^3=(bc)^2 \rangle$.
This puzzle can be solved in $3$ or less cases. The hardest case is to translate 
$v=(3,2,1)$ to $a=(1,2,3) \in A$. 
The shortest path is $(3,2,1) \to (2,3,1) \to (2,1,3)  \to (1,2,3)$. 

\paragraph{}
It is an interesting question to relate the topology of $\Gamma$ with the god number.
The diameter plays a role. We can for example ask, in which cases we can use 
{\bf spanning tree} in $\Gamma$ and get the same god number. An example we studied a bit
more detailed is whether the star graph produces the same god number than the wheel graph. 
A bold question to ask is whether for large enough graphs $\Gamma$, it is always possible to find
a spanning tree in $\Gamma$ on which the god number is the same.  This is of interest in 
logistic sorting problems as it is easier to organize a sorting mechanism along a spanning 
tree.
\index{spanning tree}

\paragraph{}
Here are some small examples. Denote by $\lfloor x \rfloor$ the floor of $x$. 
It is the largest integer smaller or equal than $x$. 

\begin{center}
\begin{tabular}{|l|l|l|}   \hline
Graph  $G$                 &   {\rm God}(G)   \\ \hline \hline
a) $K_n$  complete graph   &   $n-1$.  \\
b) $P_n$  path graph       &   $n(n+1)/2$.  \\
c) $C_n$  cycle graph      &   $\lfloor (n^2+1)/4 \rfloor$ \cite{AdinAlonRoichman} \\
d) $S_n$  star graph       &   $\lfloor 3n/2 \rfloor$ sketched in \cite{AkersKrishnamurthy}\\
e) $W_n$  wheel graph      &   between $n$ and $\lfloor 3n/2 \rfloor$.  \\ \hline
\end{tabular}
\end{center} 

\paragraph{}
Our proof of the circular c) appears to be slightly simpler than what can be seen in \cite{AdinAlonRoichman}.
Together with \cite{ZuylenBieronSchielekampYu2014} it produces the result in \cite{AdinAlonRoichman}. 
The paper \cite{AkersKrishnamurthy} for the star graph gives an upper bound.

\begin{proof}
a) Induction. If an $(n+1)$'th point is added, we just need one more transposition. \\
b) This is bubble sort. The worst case is to sort the reversed string, as it needs
   $n-1 + n-2 + n-3  + \cdots + 1 = n(n+1)/2$ swaps. \\
c) This is the best studied case. The graph distance is $m^2$ if $n=2m$
and $m(m-1)$ if $n=2m+1$. The reason is that we have to move distance $m$
maximally for $m$ vertices. A full proof is in \cite{AdinAlonRoichman}.
See also \cite{ZuylenBieronSchielekampYu2014} for the upper bound. 
As for the lower bound a concrete example is
$\pi(k) = k + \lfloor \frac{n}{2} \rfloor \; {\rm mod} \; n$, which has
distance $d(\pi,\iota)= \lfloor\frac{n^2+1}{4} \rfloor$ from the identity $\iota(k)=k$.
Proof: If an edge $e$ in $\Gamma=C_n$ is removed, it gives a path graph $P_n$. Label its vertices $\{1, \dots, n\}$.
For any permutation $\sigma$, define $P_{e,\sigma}(x)$ to be the position of $x$ in $\{1,\dots,n\}$
that $x$ occupies in the order determined by our cut $e$. Note that $P_{e,\iota}(x)=x$. Define
$$ F_e(\sigma):= \sum_{x=1}^n |P_{e,\sigma}(x)-x| \; . $$
We have $F_e(\iota)=0$. If d($\sigma',\sigma)=1$ then $|F_e(\sigma')-F_e(\sigma)| \leq 2$ because
a transposition exchanges the positions of exactly two labels, say $x$ and $y$.
Therefore, if $(d(\sigma,\iota)=t$, then $F_e(\sigma)=|F_e(\sigma)-F_e(\mathrm{id})| \le 2t$.
It follows that $d(\sigma,\iota)  \geq \frac{F_e(\sigma)}{2}$.  \\
Case (i) In the case $n=2m$, chose the edge $e$ between $m$ and $m+1$ so that $P_{e,\pi}(x)=x-m$.   
 This gives $F_e(\pi)=\sum_{x=1}^{2m} |P_{e,\pi}(x)-x|=\sum_{x=1}^{2m} m= 2m^2$ and so
$d(\pi_{2m},\iota) \geq \frac{F_e(\pi)}{2}=m^2 = \lfloor \frac{(2n)^2+1}{4}$.  \\
Case (ii) in the case $n=2m+1$, let $e$ be the edge between $m+1$ and $m+2$. The labels $1,\dots,m$ are 
shifted to the right by $m+1$ steps, and the labels $m+1,\dots,2m+1$ are shifted to the left by $m$ steps.
We have $F_e(\pi_{2m+1})=\sum_{x=1}^{m} (m+1) +\sum_{x=m+1}^{2m+1} m= m(m+1)+(m+1)m=2m(m+1)$ so that
$d(\pi_{2m+1},\iota) \geq \frac{F_e(\pi_{2m+1})}{2}=m(m+1)=\lfloor \frac{(2m+1)^2+1}{4} \rfloor$. \\
In the circle, where we deal with the integer sequence A002620. \\
d) For the star graph, an upper bound is covered by the wheel case. The star graph 
appears in \cite{AkersKrishnamurthy} (Theorem 7). \\
e) This is covered in two separate lemmas. It follows from d) as a wheel graph has
more edges than a star graph. The lower bound follows because it is contained in the
complete graph. 
\end{proof} 

\paragraph{}
The next two lemmas establish also part e) in the table. The star graph is a subgraph so that the god number
known since \cite{AkersKrishnamurthy} is larger or equal than in the wheel case. 
We see in experiments that the difference between star and wheel case is maximally $1$ and
that for larger $n$, the wheel and star graph have the same god number. (There is also a more detailed numerical 
table in the code section.)
It is an interesting question, whether $\mathrm{God}(W_{n+1}) = \mathrm{God}(S_{n+1})$ for large enough $n$. 
Of course $\mathrm{God}(W_{n+1}) \leq \mathrm{God}(S_{n+1})$ because we can on the wheel graph also use swaps on 
the rim. It seems however that these {\bf rim swaps} can be avoided for large $n$ (this is an open question 
equaivalent to the above question whether the god numbers of the star and wheel graph agree for larger $n$.
\index{rim swap}
\index{Wheel-Star problem}

\paragraph{}
We do not know an exact formula for the god number of $W_{n}$. A natural conjecture is that it is the same than 
for the star graph for larger $n$. 
A trivial estimate is ${\rm God}(W_{n+1}) \leq {\rm God}(S_{n+1} \leq 3(n-1)$ as three moves
allow to switch any two points in the boundary keeping the center invariant. The additional "parking position" at 
the center renders $O(n^2)$ complexity sorting on the circular graph (or path graph) to a linear $O(n)$ 
complexity. 

\begin{lemma}[Wheel upper bound]
$\mathrm{God}(W_{n+1}) \leq \lfloor\frac{3n}{2} \rfloor$
\end{lemma}

\begin{proof}
Let $c$ be the center of the wheel. Denote by $r_1,\dots,r_n$ the {\bf rim vertices} of $W_{n+1}$. 
Because $W_{n+1}$ is connected, any permutation $\pi \in S_{n+1}$ is a product of transpositions 
$\pi_1 \cdots \pi_\ell$. Then either $\pi_i(c) \neq c$ (which means that we use transpositions on the star) 
or then $\pi_i(c)=c$ (which means that we use transpositions on the rim). In the first case, $\pi_i$ 
is of the form $\pi_i=(c,a_1,a_2,\dots,a_k)$, where each $a_j$ is some $r_i$. We have the factorization
$$ (c,a_1,a_2,\dots,a_k)= \underbrace{(c\,a_k)(c\,a_{k-1})\cdots(c\,a_1)}_{k \ \text{transpositions}} \; . $$
A cycle of length $k+1$ that contains $c$ can be reached in $k$ moves on the star. 
In the second case, $\pi_i$ is of the form $\pi_i=(a_1\,a_2\,\dots,a_k)$,
where all $a_j$ are some $r_i$. Then we have the factorization
$$ (a_1\,a_2\,\dots,a_k)=\underbrace{(c\,a_1)(c\,a_k)(c\,a_{k-1})\cdots(c\,a_2)(c\,a_1)}_{k+1 \ \text{transpositions}} $$
So a cycle of length $k$ that does not contain $c$ can be reached in $k+1$ moves on the star. \\
Define $\Pi:=\{\pi_i \mid 1\le i \le \ell\}$ to be the set of our disjoint cycles. Then 
$$\Pi=\underbrace{\{\pi_j\mid \pi_j(c)=c\}}_{:=A}\sqcup \underbrace{\{\pi_j\mid \pi_j(c)\neq c\}}_{:=B}$$ 
By disjointedness of the $\pi_i$, $|B|=1$ or $|B|=0$. 
Let $|A|=\beta$ and let the unique $\tau\in B$ (if at all) be $\tau=(c\,a_1\,\dots,a_m).$
Let the cycles in $A$ have lengths $k_1,\dots,k_\beta.$
Since all the cycles are disjoint and $\pi$ permutes at most $n$ vertices, we have 
$m+\sum_{i=1}^\beta k_i \leq n$. By the above, we can achieve $\pi$ in at most
$$m+\sum_{i=1}^\beta (k_i+1)=m+\sum_{i=1}^\beta k_i + \beta\le n+\beta$$
moves. Now observe that each $\sigma \in A$ has length at least $2$, so
$2\beta\le \sum_{i=1}^\beta k_i \le n\implies \beta\le \lfloor\frac n2 \rfloor$.
Thus we see that 
$d(\pi,\mathrm{id})\le n+\beta\le n+\lfloor\frac n2 \rfloor=\lfloor\frac{3n}{2} \rfloor$ 
if and only if  $\mathrm{God}(W_{n+1}) \le n+\lfloor\frac{n}{2} \rfloor$.
\end{proof}

\paragraph{}
Now lets turn to the lower bound, where need to give a specific configuration.

\begin{lemma}[Wheel lower bound]
$\mathrm{God}(W_{n+1}) \geq n$
\end{lemma}

\begin{proof}
Similarly as in the upper bound of the circular case, we give a
permutation $\pi \in S_{n+1}$ such that 
$d(\pi,\iota) \geq n$. For any $\sigma \in S_{n+1}$, 
we define $P_{\sigma}(x)$ to be the vertex of $W_{n+1}$ that $x$ occupies. Define
$$ D(\sigma)=\sum_{x\in\{c,r_1,\dots,r_n\}} |P_{\sigma}(x)-x| \; . $$
We observe that a single legal move swaps the labels on the endpoints of some edge, 
and each of them moves by 1. Thus their distance from home can decrease by at most $1$, so $D$ 
can decrease by at most $2$ in a single move. Therefore, if we require a minimum of $t$ 
moves to take $\sigma$ to $\iota$, then $0=D(\iota)\ge D(\sigma)-2t \implies t \ge \frac{D(\sigma)}{2}$.
For $m=\lfloor \frac{n}{2} \rfloor$ define $\pi\in S_{n+1}$ by
$\pi(c)=c,\ \pi(r_i)=r_{i+m}$, where we work $\mod \; n$. For $n \geq 4$, 
$r_{i+m}$ is not adjacent to $r_i$ on the rim, so $d(r_i,r_{i+m})=2$,
(where we traverse the center for shortest distance). Since $\pi(c)=c$, we have
$$  D(\pi)=\sum_{i=1}^n d(r_{i+m},r_i)=\sum_{i=1}^n 2=2n \; . $$
\end{proof}

\paragraph{}
The question which game graphs of transposition games 
are Hamiltonian is related to the 
{\bf Lovasz vertex-transitive graph conjecture}. We noticed experimentally that all 
game graphs from transposition games
have the Hamiltonian graph property.
For the octahedron $\Gamma=K_{2,2,2}$ for example we have $|V|=720, |E|=4320$. 
It is a Hamiltonian graph.  This is indeed a known result from 1975: 
\cite{KompelmakherLiskovets1975}. We call it the {\bf transposition game theorem:}
\index{Lovaz conjecture}
\index{octahedron}
\index{transposition game theorem}

\begin{thm}
Every connected finite graph $\Gamma$ defines a transposition game for which 
the game graph is Hamiltonian.
\end{thm} 

\index{Lovasz conjecture}
\index{theorem of Kompelmakher-Liskovets}

\paragraph{}
We also looked at the size of the {\bf last layer} of complete graph puzzles.
These are the points furthest away from a point.
For $K_5$, the last layer has exactly $3$ configurations all of 
cycle type $(2, 2, 1)$ and all with the hole at vertex $0$.
For $K_6$, the last layer has $35$ configurations, split between 
cycle types $(2, 2, 2)$ (with $15$ representatives) and $(3, 2, 1)$ 
(with $20$ representatives), with hole positions distributed across all vertices.
We have code producing the BFS layer profile for any graph. 
\index{last layer problem}

\paragraph{}
Finally, lets look at a topological example. We take a discrete {\bf M\"obius strip}
$\Gamma_1$ and a discrete {\bf Cylinder strip} $\Gamma_1$. They both have 
$8$ vertices and $16$ edges and $8$ triangular faces; but the cylinder is 
oriented, while the M\"obius strip is not. We can now look at the 
game graphs $G_i$ of both. They are both finite graphs with 
$8!=40320$ vertices and $322560$ edges. But the God numbers are
different. For the M\"obius strip, we have ${\rm God}(G_1) = 10$. 
For the cylinder, we have ${\rm God}(G_2)=11$. In the code section we have both
mathematica and GAP procedures what establish these numbers independently. 

\begin{figure}[!htpb]
\scalebox{0.8}{\includegraphics{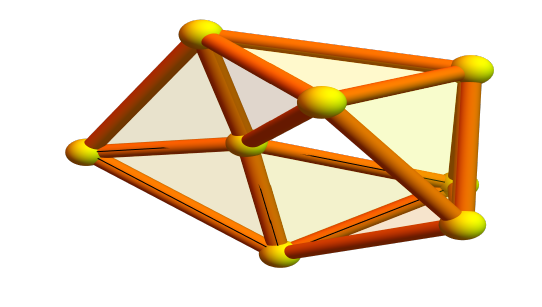}}
\scalebox{0.8}{\includegraphics{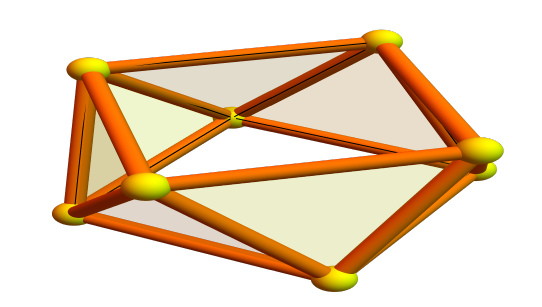}}
\label{moebius}
\caption{
The Moebius strip and Cylinder both generate game graphs $G_1,G_2$.
These two graphs are both Cayley graphs of $S_8$ and both have the same
constant vertex degree $16$, and so the same number of edges. 
But they have different diameter $10$ rsp. $11$. 
}
\end{figure}
\index{von Neumann Morgenstern axioms}

\vfill
\pagebreak

\section{Sliding puzzles}

\paragraph{}
{\bf Sliding puzzles} generalize the classic 15-puzzle $(\Gamma=C_4 \times C_4)$ 
or the {\bf magic ball} $\Gamma$ is the {\bf icosahedron graph}) to arbitrary connected graphs. 
Sliding puzzles for connected $\Gamma$ produce game graphs $G$ that are in general 
more interesting than the game graphs of transposition games, as the latter are always $S_n$. 
Wilson's theorem~\cite{wilson1974} classifies the graphs. 
\index{Wilson's theorem}
\index{sliding puzzles}
\index{icosahedron graph}

\paragraph{}
Given a graph $\Gamma=(V,E)$ on $n$ vertices, one of the tokens is called a
{\bf hole}. A legal move slides the hole into an adjacent position, opening the 
hole where it had been. From a node $v \in G$
only transpositions involving the hole can be used. The sorting is now localized near 
the hole. The set of all reachable permutations again produces a game graph $G$ whose
diameter is then the God number.
\index{hole}

\paragraph{}
The history of the 15-puzzle (where $\Gamma$ is the grid graph $C_4 \times C_4$) is iconic 
and well documented: after a precursor was shown around by {\bf Noyes Chapman} in 1874, he applied 
for a patent in 1880. A craze was driven by Sam Loyd's offer of $\$1000$ for swapping tiles 14 and 15. 
Johnson and Story~\cite{johnsonstory1879} demonstrated that a permutation parity is an invariant. 
implying that only half of the $16!$ positions are reachable.
Finding the god number of the general $n \times n$ version of the puzzle is an NP complete problem
\cite{RatnerWarmuth}. 
\index{history of 15-puzzle}
\index{15 puzzle craze}

\paragraph{}
Computational results for the $(n^2{-}1)$-puzzle proceeded with the development of
{\bf iterative deepening} in the context of {\bf pattern databases}. 
\cite{korf1985} introduced {\bf iterative-deepening}. \index{iterative deepening}
\cite{reinefeld1993} showed that the 8-puzzle's god's number is $31$. 
\cite{bmfn1999} established that less than 81 moves are necessary for the 15-puzzle in 1999 via parallel BFS on ZRAM, 
extending Gasser's earlier search that had found 9 states at depth 80. 
\cite{Korf2008} confirmed and extended such exhaustive implicit-search computations using disk-based breadth-first search
and stated 17 states at depth 80. 
\cite{korftaylor1993} built the first optimal solver for individual 24-puzzle instances.
The best bounds are $152 \le G \le 205$, the upper bound improved from 208 in 2016. 
\cite{ratnerwarmuth1990} proved that finding a shortest solution to the $(n^2{-}1)$-puzzle is NP-hard. 
\cite{parberry1995} gave a $\Theta(n^3)$ real-time algorithm and showed $\Theta(n^3)$ is also the growth rate of the 
God's number itself. The graph-theoretic generalisation is due to \cite{wilson1974}.
A modern textbook approach is \cite{archer1999}.
For a discussion of higher dimensional puzzles, see \cite{BMRV}.

\paragraph{}
Here is an overview of the God's number table for the $(n^2{-}1)$-puzzle with a single-tile metric:

\begin{center}
\small
\begin{tabular}{@{}l l r r r l@{}}
\toprule
Puzzle & Grid & \# states $(n^2)!/2$ & God's number & \# hardest states & Reference \\
\midrule
3-puzzle  & $2\times 2$ & $12$ & $6$ & $1$ (antipode on 12-cycle) & trivial \\
8-puzzle  & $3\times 3$ & $181{,}440$ & $31$ (24 multi-tile) & --- & \cite{reinefeld1993} \\
15-puzzle & $4\times 4$ & $1.046\times 10^{13}$ & $80$ (43 multi-tile) & $17$ & \cite{bmfn1999} \\
24-puzzle & $5\times 5$ & $7.76\times 10^{24}$ & $\in[152,\,205]$ & open & \cite{korftaylor1993} \\
35-puzzle & $6\times 6$ & $\approx 1.9\times 10^{41}$ & unknown & --- & --- \\
$(n^2{-}1)$-puzzle & $n\times n$ & $(n^2)!/2$ & $\Theta(n^3)$ & --- & \cite{parberry1995} \\
\bottomrule
\end{tabular}
\end{center}

\paragraph{}
There is also a {\bf multi-tile metric}, where consecutive blank moves in the 
same direction count as one. For the 15-puzzle, an upper bound is 43, for the  
8-puzzle it is 24  (OEIS A087725) gives the records $0,6,31,80,\dots$ for $n=1,2,3,4$.
Finding some solution is polynomial in $n$: {\bf Parberry's algorithm} uses 
$\Theta(n^3)$ moves~\cite{parberry1995}. Finding a {\bf shortest} solution is 
NP-hard~\cite{ratnerwarmuth1990}. It remains NP-hard to approximate within any additive constant, 
though a polynomial-time constant-factor approximation exists.
\index{Parberry algorithm}
\index{multi-tile metric}
\index{NP hard}

\paragraph{}
The game has a graph theoretical version which we discuss a bit more in an other section.
Let $\Gamma$ be a connected simple graph with $n$ vertices in which one of them is called ``blank" 
or ``hole". A legal move swaps the blank with an adjacent label. 
The resulting configuration graph $G(\Gamma)$ has one vertex per labeling and an edge per move.
The {\bf  eccentricity} of the solved vertex $w$ in its connected component is defined to be the maximal shortest
path distance from $v$ to any other vertex. 
\index{eccentricity}
\index{blank}
\index{hole}

\paragraph{}
The game graph $G$ generated by these moves depends only on $\Gamma$. 
Wilson's theorem~\cite{wilson1974} characterizes it completely. It is the analog
of the fact that for a connected graph the transposition game defined by a connected
graph $\Gamma$ on $n$ vertices is the symmetric group $G=S_n$. 
Here is the statement of {\bf Wilson's theorem on sliding graphs}:
\index{Wilson's theorem}

\begin{thm}[Wilson]
Let $G$ be a finite simple 2-connected graph with $n \ge 4$ vertices. 
The sliding-puzzle group $G$ for $\Gamma=C_n$ is $\mathbb{Z}_n$.
For $\Gamma = \theta_0$, the group is $\mathrm{PGL}_2(\mathbb{F}_5)$ of order 120.
If $\Gamma$ is bipartite but not $C_n$ nor $\theta_0$, it is the 
{\bf alternating group} $A_n$.
Otherwise, it is the full symmetric group $S_n$.
\end{thm}
\index{alternating group}

\paragraph{}
For the 15 puzzle, because the $4\times 4$ grid $\Gamma$ is bipartite. 
Wilson's result gives $A_{15}$, a group of order $15!/2$.

\paragraph{}
In a sliding game on a graph $\Gamma$, one of the vertices is marked and transpositions
are only allowed if this special point is involved.
The prototype is the 15 puzzle, where the underlying graph is the $4 \times 4$ grid graph. 
The graph $G$ of the game is in general not the Cayley graph of a group as the possible
generators can depend on the position. 
In the 15 game, if the hole is at the boundary there are less legal moves than if the hole
is inside. 

\paragraph{}
For the {\bf magic rainbow puzzle ball}, the graph $\Gamma$ is the {\bf icosahedral graph}.
We have not been able to get the god number there with our general program as the 
game graph is too large
For the magic ball analog where $\Gamma$ is the octahedron graph however,
\index{icosahedron graph}
\index{magic rainbow puzzle}

\begin{figure}[!htpb]
\scalebox{0.25}{\includegraphics{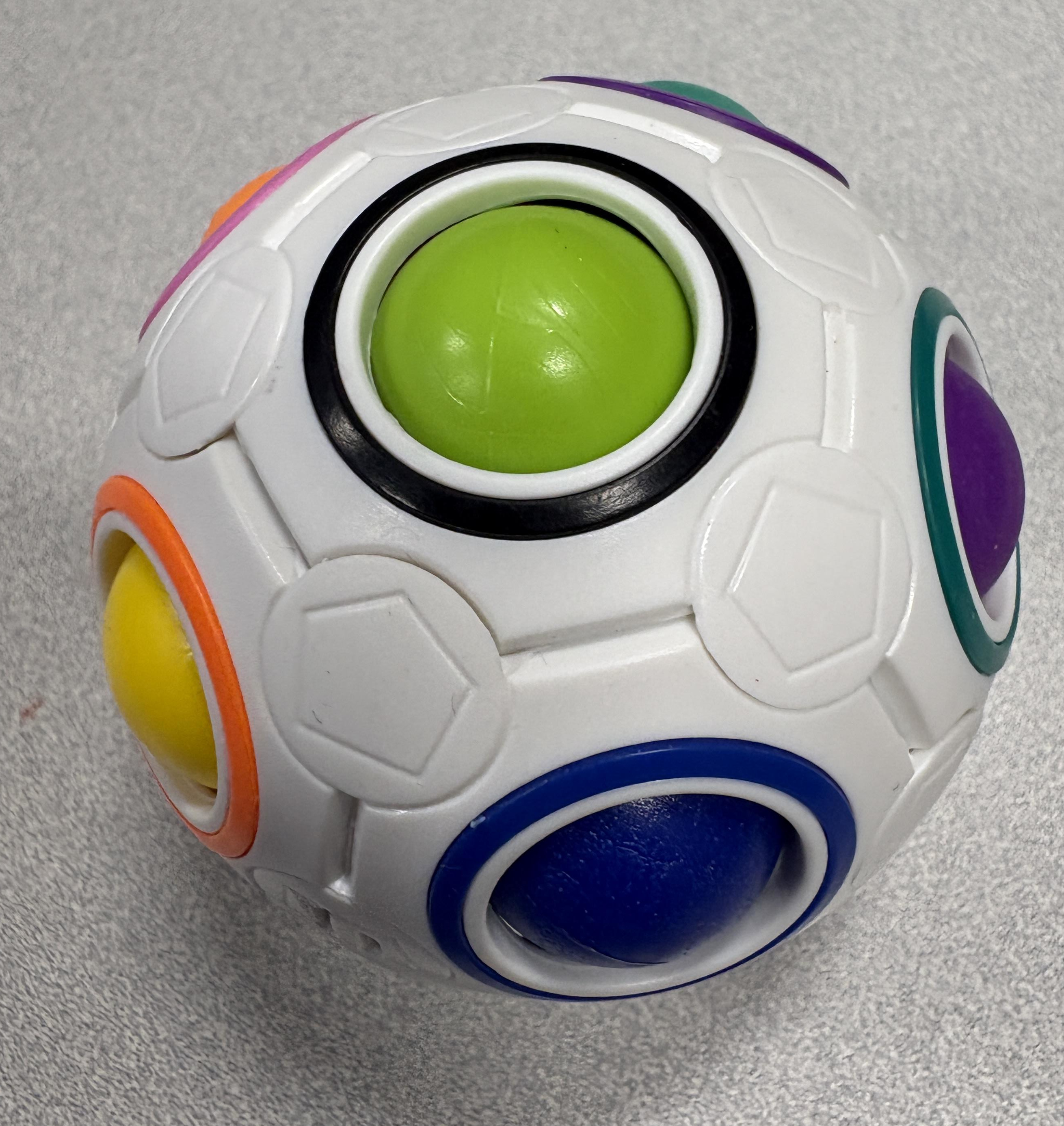}}
\label{Rainbow puzzle}
\caption{
Magic Rainbow puzzle ball.
}
\end{figure}
\index{magic rainbow puzzle}

\paragraph{}
For the 8 and 11 puzzles, we also ran C programs.
\footnote{The C programs were obtained from asking ChatGPT to 
translate the Python implementation of the
standard {\bf Lehmer-code ranking} of permutations and breadth-first
search on the implicitly generated sliding-puzzle graph.}
The 8 puzzle game is the $3 \times 3$ has god number 31.
The 11 puzzle game is the $3 \times 4$ analog. It has god number 53. 
This diameter $53$ is known and OEIS A151944 \cite{OEIS_A151944} which 
credits \cite{Korf2008}. Our Mathematica procedures to generate the small
graphs are not feasible for the 15 puzzle any more. Storing the graph with 
$16! = 20922789888000$ vertices would require dozens of terabytes of RAM.
\index{Lehmer code ranking}

\paragraph{}
Given a graph $G$ and a vertex $v$, one can look at the {\bf neighborhood graph} 
$B_r(v)$ which consists of all vertices in distance $r$ or less from $v$. This defines 
a {\bf layer partition} or {\bf spheres} $S_r = B_{r} \setminus B_{r-1}$ of $G$ of positions which can be reached
with $r$ moves but not with $r-1$ moves. It is in some sense a ``wave front" in the graph game. 
It defines the {\bf BFS layer profile} $\{|S_r(x)|\}_{r=0}^{\mathrm{{\rm God}(G,v)}}$, the sizes of these
fronts.  Natural questions pop up: what is the 
structure of this sequence of integers? What is the value of the largest non-zero entry?
Note that in the case of Cayley graphs ${\rm God}(G,v)$ is independent of $v$ but for
general game graphs, this is not necessarily the case.
The BFS layer profile often appears unimodal (concave down): it increases then decreases. 
\index{BFS layer profile}
\index{wave front}
\index{BFS layer partition}
\index{layer profile}
\index{layer partition}

\paragraph{}
{\bf Examples. } \\
{\bf 1} For the sliding puzzle with $\Gamma=K_n$, the reachable configuration
space from the identity is all of $S_n$ and ${\rm God}(G) = n-1$. 
For $K_5$, the layer sizes are $1, 4, 12, 30, 44, 26, 3$, totaling to $120=5!$.
In this case, the BFS layer profile is independent of $v$. \\
The last layers of complete graph puzzles are harder to describe cleanly but show
their own structure. For $K_5$, the last layer has exactly $3$ configurations,
all of cycle type $(2, 2, 1)$, and all with the hole at vertex $0$. For $K_6$,
the last layer has $35$ configurations, split between cycle types $(2, 2, 2)$
(with $15$ representatives) and $(3, 2, 1)$ (with $20$ representatives), with
hole positions distributed across all vertices.
For $K_5$, the last layer is a single incomplete conjugacy class: it contains
$3$ of the $15$ permutations in $S_5$ with cycle type $(2, 2, 1)$, and all $3$
have the hole at vertex $0$. \\
{\bf 2)} For the sliding puzzle on the path graph $\Gamma=P_n$,
each connected component of the game graph $G$ is again a path $P_n$, 
so that ${\rm God}(G)=n-1$. In this case however, the maximal ${\rm God}(G,v)$ 
depends on $v$. If the initial condition is near the middle, then 
The BFS profile from the case when the hole is at the boundary is 
$1, 1, 1, \ldots, 1$. \\
{\bf 3} For $\Gamma=C_n$, the game graph is $G=C_{n(n-1)}$. The
last layer $L_{\mathrm{God}}$ contains exactly one configuration, 
the antipodal point. \\
{\bf 4} We can build {\bf game graphs} coming from transposition games on trees,
for which the BFS layer profile
is not concave. Just take a tree with small vertex degree in distance $\leq 3$
then have very large vertex degree in distance $\leq 6$ and finally again 
small vertex degree until maximal radius $9$. 
\index{tree}

\vfill
\pagebreak

\section{Chess games}

\paragraph{}
{\bf Chess} is without doubt one of the most iconic multi-player games. 
It motivated the first steps in mathematical game theory \cite{Zermelo1913,Shannon1950}.
Many chess variants have been developed \cite{PritchardChess}. In the graph 
theoretical setup used here, we assumed that a game event is modeled by a finite
simple maximal graph that starts at $v$. This obviously is not the standard
assumption in chess, as one does not want to have the same position appear
3 times. And it does not matter for the God number, nor for the {\bf Zermelo function}
${\bf Zermelo}: \{ Games \} \to \{V wins, W wins, draw \}$. 
\index{Zermelo function}
\index{chess}

\paragraph{}
Let us discuss the repetition case issue.
Assume $G(n)$ is a game variant, where the game rule assumes that a position can not
occur more than $(n-1)$ times and where otherwise, all rules are the same. Game events in 
a sub-graph $H$ of $G(n)$ are now the maximal finite simple paths in $H$ which can not have
more than $n-1$ self intersections. We formulated the result for chess but obviously it
takes place for all games covered by the axioms. The following lemma is therefore a justification
why the assumption that game events are finite simple maximal paths starting at $v$ is a reasonable
one and would not change essential things like "who wins" if one would allow a number of 
repetitions (like in chess). 
\index{repetition assumption}
\index{repetition issue}

\begin{lemma}[Repetition]
${\rm God}(G(n))={\rm God}(G(1))$ and ${\rm Zermelo}(G(n)) = {\rm Zermelo}(G(1))$. 
\end{lemma}
\begin{proof}
Assume $G(n)$ is win for the initial player $V$=White. This means that $V$ has 
a strategy $H$ such that no game event for $W$ with $n-1$ possible self intersections
ends up in the winning set $A$. But then also no game events $W$ that has no self intersections
ends up in the winning set $A$ so that also $G(1)$ is win. \\
The same argument holds if $G(n)$ is a win for $W$.  \\
If $G(n)$ is a draw, then there is a game strategy $H$ to which both players have contributed
such that there is no game event with possibly up to $(n-1)$ self intersections that ends up in $A$.
So, also no game event with no self-intersections ends up in $A$ which means that the game 
strategy $H$ is also a game strategy for draw in $G(1)$. \\
The opposite direction follows from Zermelo's theorem and logic: 
assume that $G(n)$ is not a win for white, then 
this means that $G(n)$ is a draw or win for black for $G(n)$. But in both cases this leads to a 
draw or win for black for $G(1)$ and so not to a win for white for $G(1)$. The other cases are 
similar. 
\end{proof} 
\index{Chess repetition lemma}

\paragraph{}
We see that the theoretical simplification to assume finite simple paths as game events is 
also in very concrete situations like chess irrelevant. It does not change the basic 
question whether the game is a win or loss. Not imposing any limit on the number of times
that a position can be revisited, would change things however. Such objections had been 
voiced to Zermelo's original approach. Having not only a finite set of game positions but also 
a finite set of game events is essential. 

\paragraph{}
If we play on smaller boards, we can hope to explore the entire game graph.
Here is the state graph of level 4 of the $4\times2$ chess board.

\begin{center}
\begin{adjustbox}{max width=12cm}
\begin{forest}
  for tree={
    draw, rounded corners=2pt,
    align=center,
    inner sep=1pt,
    fill=white,
    font=\tiny, 
    edge={-latex, thin},
    parent anchor=south,
    child anchor=north, 
    grow=south,
    l sep=50mm, 
    s sep=1mm, 
  }
  [\pos{Qa1,Kb1,Pa2,Pb2,qa4,kb4,pa3,pb3}\\\textbf{Start}
    [\pos{Qa1,Kb1,Pb2,Pb3,qa4,kb4,pa3}\\1.axb3
      [\pos{Qa1,Kb1,Pb2,Pb3,qa4,kb4,pa2}\\1\textellipsis a2+
        [\pos{Qa2,Kb1,Pb2,Pb3,qa4,kb4}\\2.Qxa2
          [\pos{Qa2,Kb1,Pb2,Pb3,qa3,kb4}\\2\textellipsis Qa3]
          [\pos{qa2,Kb1,Pb2,Pb3,kb4}\\2\textellipsis Qxa2+]
          [\pos{Qa2,Kb1,Pb2,qb3,kb4}\\2\textellipsis Qxb3]
        ]
      ]
      [\pos{Qa1,Kb1,Pb3,qa4,kb4,pb2}\\1\textellipsis axb2
        [\pos{Qa1,Kb1,Ba4,kb4,pb2}\\2.bxa4=B
          [\pos{ba1,Kb1,Ba4,kb4}\\2\textellipsis bxa1=B]
          [\pos{na1,Kb1,Ba4,kb4}\\2\textellipsis bxa1=N]
          [\pos{qa1,Kb1,Ba4,kb4}\\2\textellipsis bxa1=Q+]
          [\pos{ra1,Kb1,Ba4,kb4}\\2\textellipsis bxa1=R+]
        ]
        [\pos{Qa1,Kb1,Na4,kb4,pb2}\\2.bxa4=N
          [\pos{ba1,Kb1,Na4,kb4}\\2\textellipsis bxa1=B]
          [\pos{na1,Kb1,Na4,kb4}\\2\textellipsis bxa1=N]
          [\pos{qa1,Kb1,Na4,kb4}\\2\textellipsis bxa1=Q+]
          [\pos{ra1,Kb1,Na4,kb4}\\2\textellipsis bxa1=R+]
          [\pos{Qa1,Kb1,Na4,kb3,pb2}\\2\textellipsis Kb3]
        ]
        [\pos{Qa1,Kb1,Qa4,kb4,pb2}\\2.bxa4=Q\#]
        [\pos{Qa1,Kb1,Ra4,kb4,pb2}\\2.bxa4=R+
          [\pos{Qa1,Kb1,Ra4,kb3,pb2}\\2\textellipsis Kb3]
        ]
        [\pos{Qa1,Kb2,Pb3,qa4,kb4}\\2.Kxb2
          [\pos{Qa1,Kb2,Pb3,qa2,kb4}\\2\textellipsis Qa2+]
          [\pos{Qa1,Kb2,Pb3,qa3,kb4}\\2\textellipsis Qa3+]
          [\pos{qa1,Kb2,Pb3,kb4}\\2\textellipsis Qxa1+]
          [\pos{Qa1,Kb2,qb3,kb4}\\2\textellipsis Qxb3\#]
        ]
        [\pos{Qa2,Kb1,Pb3,qa4,kb4,pb2}\\2.Qa2
          [\pos{Qa2,Kb1,Pb3,qa3,kb4,pb2}\\2\textellipsis Qa3]
          [\pos{qa2,Kb1,Pb3,kb4,pb2}\\2\textellipsis Qxa2+]
          [\pos{Qa2,Kb1,qb3,kb4,pb2}\\2\textellipsis Qxb3]
        ]
        [\pos{Qa3,Kb1,Pb3,qa4,kb4,pb2}\\2.Qa3+
          [\pos{ka3,Kb1,Pb3,qa4,pb2}\\2\textellipsis Kxa3]
          [\pos{qa3,Kb1,Pb3,kb4,pb2}\\2\textellipsis Qxa3]
        ]
        [\pos{Qa4,Kb1,Pb3,kb4,pb2}\\2.Qxa4\#]
        [\pos{Qb2,Kb1,Pb3,qa4,kb4}\\2.Qxb2
          [\pos{Qb2,Kb1,Pb3,qa1,kb4}\\2\textellipsis Qa1+]
          [\pos{Qb2,Kb1,Pb3,qa2,kb4}\\2\textellipsis Qa2+]
          [\pos{Qb2,Kb1,Pb3,qa3,kb4}\\2\textellipsis Qa3]
          [\pos{Qb2,Kb1,qb3,kb4}\\2\textellipsis Qxb3]
        ]
      ]
      [\pos{Qa1,Kb1,Pb2,qa4,kb3}\\1\textellipsis Kxb3
        [\pos{Qa1,Kb1,Pa3,qa4,kb3}\\2.bxa3
          [\pos{Qa1,Kb1,Pa3,qb4,kb3}\\2\textellipsis Qb4]
          [\pos{Qa1,Kb1,qa3,kb3}\\2\textellipsis Qxa3]
        ]
        [\pos{Qa2,Kb1,Pb2,qa4,kb3}\\2.Qa2+
          [\pos{Qa2,Kb1,Pb2,qa4,kb4}\\2\textellipsis Kb4]
        ]
        [\pos{Qa3,Kb1,Pb2,qa4,kb3}\\2.Qxa3+
          [\pos{qa3,Kb1,Pb2,kb3}\\2\textellipsis Qxa3]
        ]
      ]
      [\pos{Qa1,Kb1,Pb2,qb3,kb4,pa3}\\1\textellipsis Qxb3
        [\pos{Qa2,Kb1,Pb2,qb3,kb4,pa3}\\2.Qa2
          [\pos{Qa2,Kb1,qb3,kb4,pb2}\\2\textellipsis axb2]
          [\pos{Qa2,Kb1,Pb2,qb3,ka4,pa3}\\2\textellipsis Ka4]
          [\pos{Qa2,Kb1,Pb2,qa4,kb4,pa3}\\2\textellipsis Qa4]
          [\pos{qa2,Kb1,Pb2,kb4,pa3}\\2\textellipsis Qxa2+]
          [\pos{Qa2,Kb1,qb2,kb4,pa3}\\2\textellipsis Qxb2+]
        ]
        [\pos{Qa3,Kb1,Pb2,qb3,kb4}\\2.Qxa3+
          [\pos{qa3,Kb1,Pb2,kb4}\\2\textellipsis Qxa3]
        ]
      ]
    ]
    [\pos{Qa1,Kb1,Pa2,Pa3,qa4,kb4,pb3}\\1.bxa3+
      [\pos{Qa1,Kb1,Pa2,qa4,ka3,pb3}\\1\textellipsis Kxa3
        [\pos{Qa1,Kb1,Pb3,qa4,ka3}\\2.axb3+
          [\pos{Qa1,Kb1,Pb3,qa4,kb4}\\2\textellipsis Kb4]
          [\pos{Qa1,Kb1,qa4,kb3}\\2\textellipsis Kxb3]
        ]
        [\pos{Qb2,Kb1,Pa2,qa4,ka3,pb3}\\2.Qb2+
          [\pos{Qb2,Kb1,Pa2,qa4,kb4,pb3}\\2\textellipsis Kb4]
        ]
      ]
      [\pos{Qa1,Kb1,Pa2,qa3,kb4,pb3}\\1\textellipsis Qxa3
        [\pos{Qa1,Kb1,Pb3,qa3,kb4}\\2.axb3
          [\pos{Qa1,Kb1,qa3,kb3}\\2\textellipsis Kxb3]
          [\pos{Qa1,Kb1,Pb3,qa2,kb4}\\2\textellipsis Qa2+]
          [\pos{Qa1,Kb1,Pb3,qa4,kb4}\\2\textellipsis Qa4]
          [\pos{Qa1,Kb1,Pb3,qb2,kb4}\\2\textellipsis Qb2+]
          [\pos{qa1,Kb1,Pb3,kb4}\\2\textellipsis Qxa1+]
          [\pos{Qa1,Kb1,qb3,kb4}\\2\textellipsis Qxb3+]
        ]
        [\pos{Qb2,Kb1,Pa2,qa3,kb4,pb3}\\2.Qb2
          [\pos{Qb2,Kb1,Pa2,qa3,ka4,pb3}\\2\textellipsis Ka4]
          [\pos{Qb2,Kb1,Pa2,qa4,kb4,pb3}\\2\textellipsis Qa4]
          [\pos{Qb2,Kb1,qa2,kb4,pb3}\\2\textellipsis Qxa2+]
          [\pos{qb2,Kb1,Pa2,kb4,pb3}\\2\textellipsis Qxb2+]
        ]
      ]
    ]
  ]
\end{forest}
\end{adjustbox}
\end{center}

\paragraph{}
One path leads down to a white forced win, but the other paths all lead to draws or 
stalemates assuming all players play rationally.

\begin{center}
\begin{adjustbox}{max width=12cm}
\begin{forest}
  for tree={
    draw, rounded corners=3pt,
    align=center,
    inner sep=2pt,
    fill=white,
    font=\small,
    edge={-latex, thick},
    parent anchor=south,
    child anchor=north, 
    grow=south,
    l sep=4mm,
    s sep=6mm,
  }
  [\pos{\posINIT}\\\textbf{Start}
    [\pos{\posMa}\\1.axb3
      [\pos{\posMb}\\1\textellipsis a2+
        [\pos{\posMc}\\2.Qxa2
          [\pos{\posMd}\\2\textellipsis Qa3
            [\pos{\posMATE}\\\textcolor{red}{\textbf{3.Qxa3\#}}]
          ]
        ]
      ]
      [\pos{\posAxbTwo}\\1\textellipsis axb2
        [{$\cdots$}\\\textcolor{blue}{Draw}]
      ]
      [\pos{\posQxbThree}\\1\textellipsis Qxb3
        [{$\cdots$}\\\textcolor{blue}{Draw}]
      ]
      [\pos{\posKxbThree}\\1\textellipsis Kxb3
        [{$\cdots$}\\\textcolor{blue}{Draw}]
      ]
    ]
    [\pos{\posDraw}\\1.bxa3
      [{$\cdots$}\\\textcolor{blue}{Draw (stalemate)}]
      [{$\cdots$}\\\textcolor{blue}{Draw (repetition)}]
    ]
  ]
\end{forest}
\end{adjustbox}
\end{center}

\paragraph{}
The $4 \times 2$ chess board has pieces $\{K,Q,P\}$ for both white and black players. 
It defines a directed digraph $G = (V \cup W, E)$. The graph starts from the starting position 
$v \in V$. The graph is generated using a breadth-first search (BFS) starting from $v$. For each state 
dequeued, we compute the legal moves using pseudo-legal move generation followed by check validation 
to ensure legality. States are added to the vertex set $V \cup W$ and the queue only upon first 
discovery so that each unique position is processed exactly once. The BFS traversal guarantees that 
$G$ contains all reachable states and represents the complete game tree. 
The height of a node is the distance to $v$. 

\paragraph{}
The set of end positions $S$ are the nodes for which the out-degree is zero. These are positions,
where no legal continuation is possible like with {\bf chess mate} or where the game is pointless to continue, like of
both parties only have kings. Within $S$ is the set, $A \subset S$ is the set of win positions. 
Game events are modeled as maximal simple paths starting at $v$. A path ending in $A \cap W$ constitutes 
a victory for V=White, whereas a path ending in $A \cap V$ constitutes a victory for W=Black. 
Terminal nodes that fall into the complement $S \setminus A$, which are draw positions.
This can be {\bf stalemates}, where the active player is not in check but possesses no legal moves.
It is custom to include in $S$ also situations, where a continuation is pointless as no win is possible.
It is also assumed that $A$ contains situations, where a continuation is hopeless. 
Nobody would battle out a game where a rook and a king battle a lonely king, even so the game path 
would allow to reach a stalemate still in principle.
\index{stalemate}
\index{chess mate}
\index{lonely king}

\paragraph{}
Let us look at the following initial configurations:
(I) king-side pawn only, (II) queen-side pawn only, and (III) where both pawns present. 
The corresponding graphs exhibit the following properties:

\begin{itemize}
    \item Configuration I: $|V \cup W| = 517$, $|E| = 1034$, $|A \cap W| = 10$, $|A \cap V| = 11$, $|S \setminus A| = 52$, $\max d(x) = 16$
    \item Configuration II: $|V \cup W| = 519$, $|E| = 1037$, $|A \cap W| = 10$, $|A \cap V| = 11$, $|S \setminus A| = 52$, $\max d(x) = 16$
    \item Configuration III: $|V \cup W| = 1966$, $|E| = 4534$, $|A \cap W| = 54$, $|A \cap V| = 44$, $|S \setminus A| = 113$, $\max d(x) = 20$
\end{itemize}

\begin{figure}[!htpb]
\scalebox{0.07}{\includegraphics{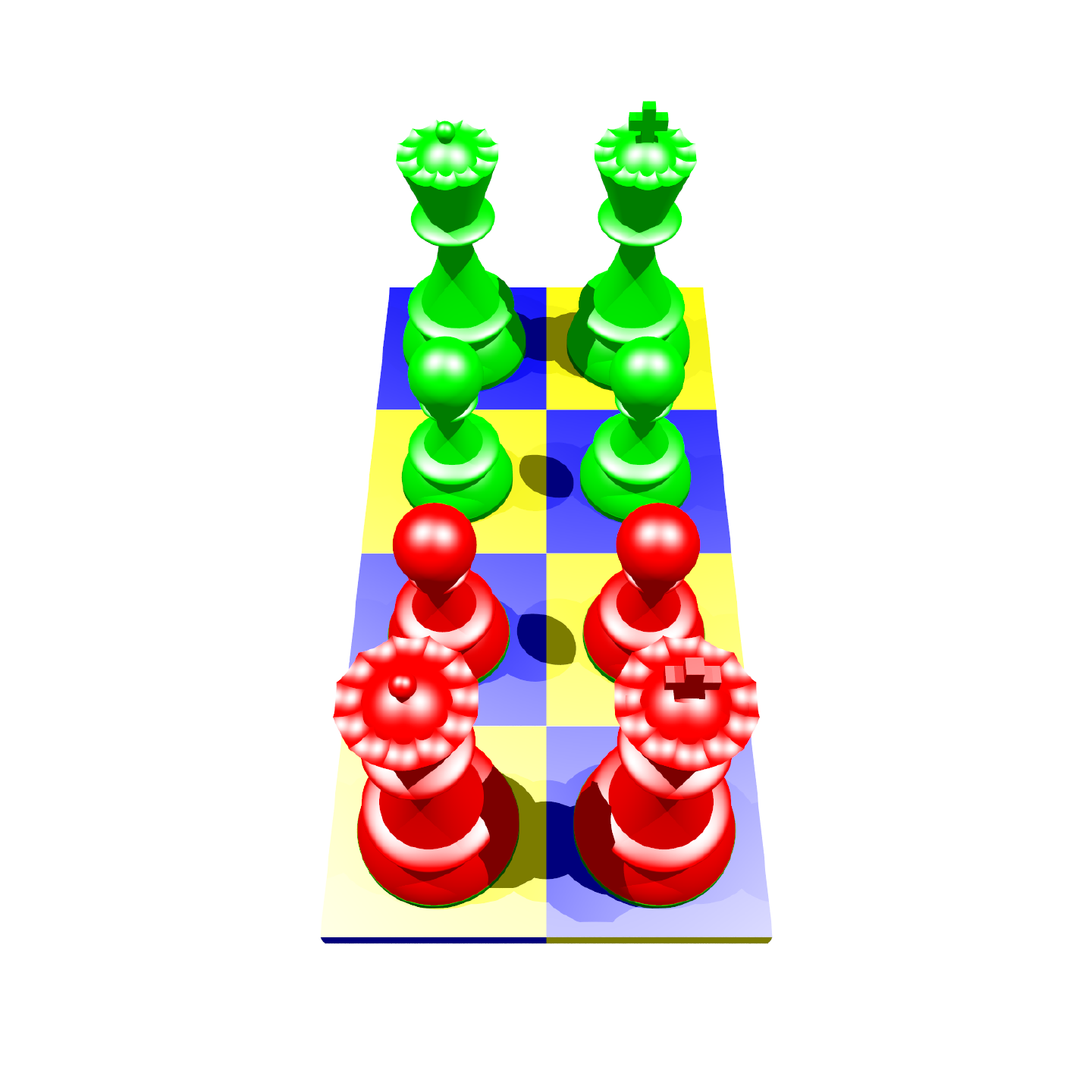}}
\scalebox{0.07}{\includegraphics{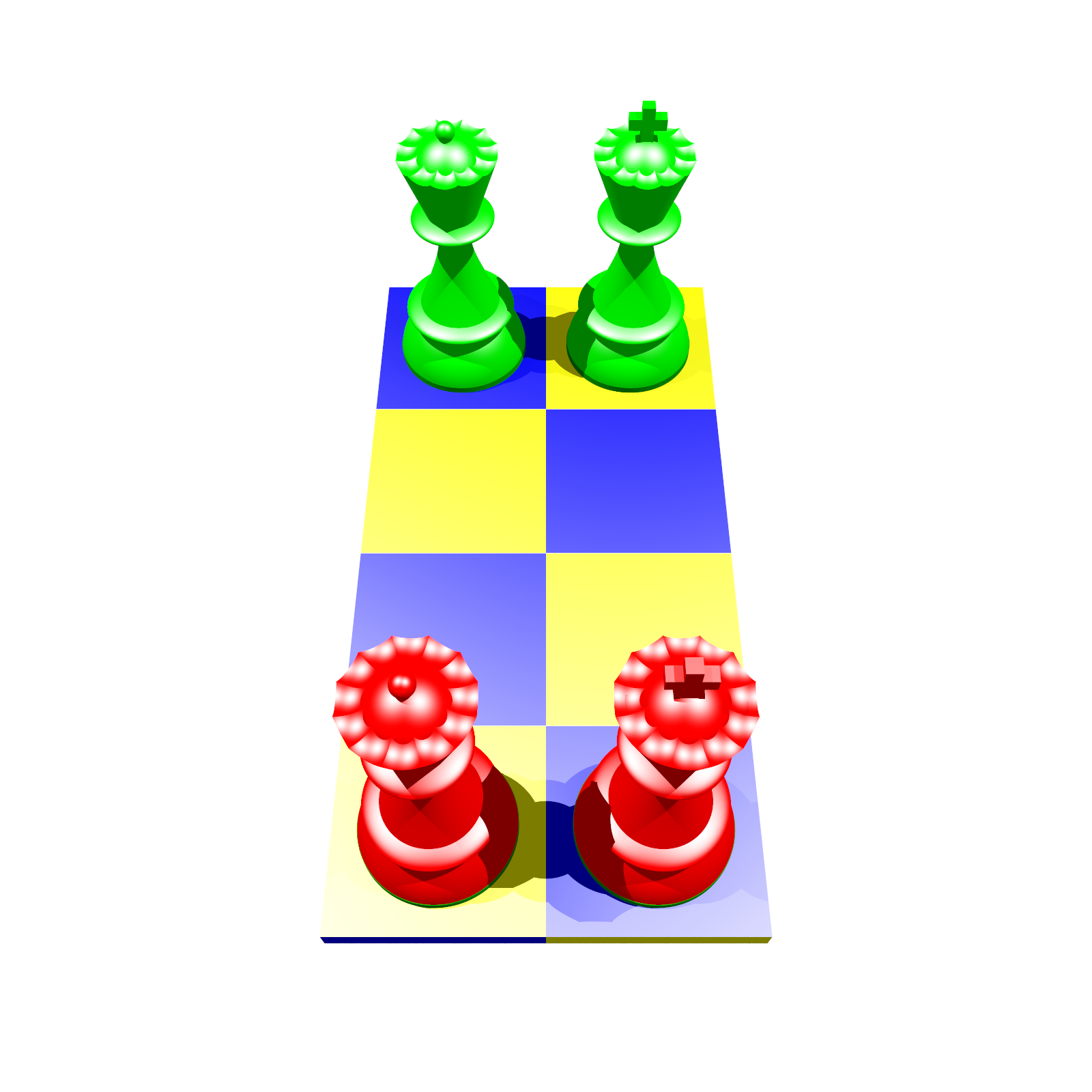}}
\scalebox{0.07}{\includegraphics{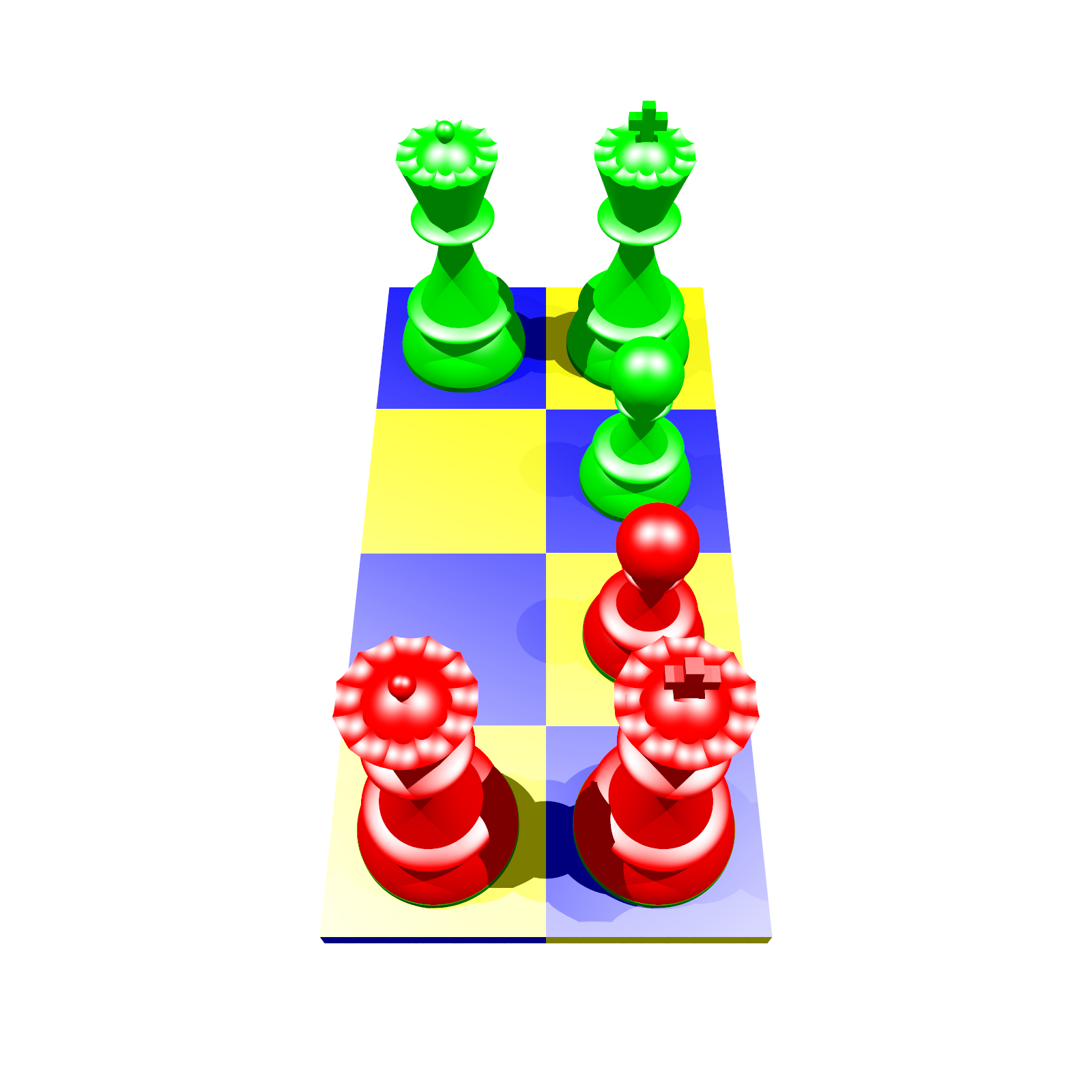}}
\scalebox{0.07}{\includegraphics{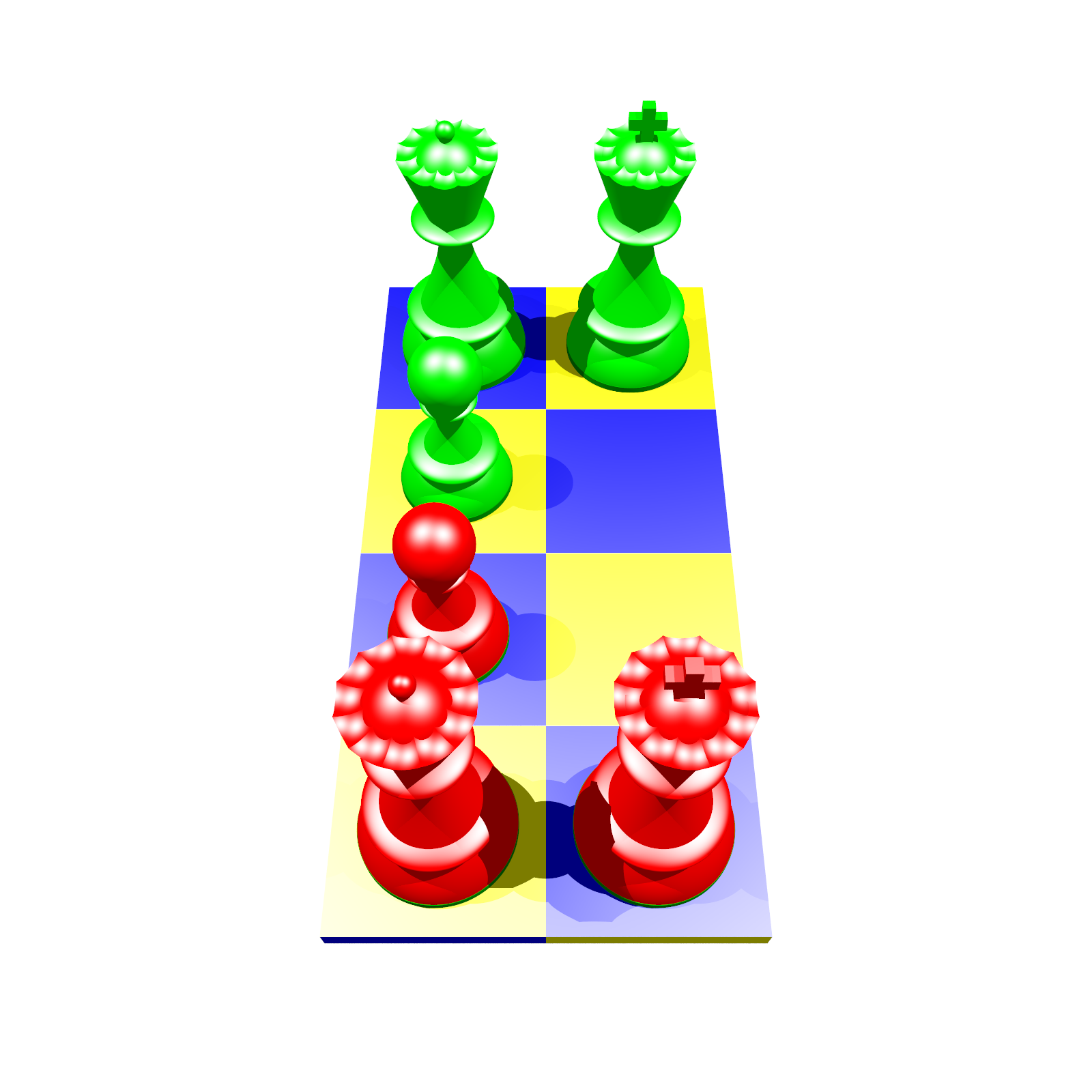}}
\label{Small Chess puzzles}
\caption{ 
The middle two cases cases are ``draw in one". The strategy for $V$ is
to eliminate all but the capturing move. In the right, $W$ wins.
}
\end{figure}

\paragraph{}
The exponential growth in graph complexity from configurations I/II to III demonstrates the combinatorial explosion 
inherent in chess variants as piece density increases.
The maximum shortest-path distance $\max d(x)$ reflects the diameter of the reachable graph and indicates the minimum 
number of plies required to reach the most strategically distant positions. 
While the both-pawn configuration exhibits greater graph depth due to increased tactical complexity, 
determining the formal God number of these configurations - defined as the minimal graph path length evaluating over all 
sub-graph strategies will require further minimax optimization across the sub-graphs $H \subset G$ to minimize the 
maximal potential loss. In theoretical analysis of chess, one usually does not consider {\bf promotion} as this can 
render the game graph even larger. 
\index{promotion}

\paragraph{}
Lets look at two small $4 \times 4$ and $2 \times 4$ chess games.
In the $4 \times 4$ situation there is mate in 3. In the $2 \times 4$ game, there is
a draw. Can we estimating the longest possible game,
meaning the longest possible game event in the graph? In small situations as such
it might be possible. 
A trivial upper bound for the game events is $|V|^{|V|}$, where V are the 
set of all possible positions.

\paragraph{}
Reducing the number of positions produces to {\bf chess problem},
A typical problem gives the god number and asks to verify that this is the case. 
\index{chess problems}

\begin{figure}[!htpb]
\scalebox{0.1}{\includegraphics{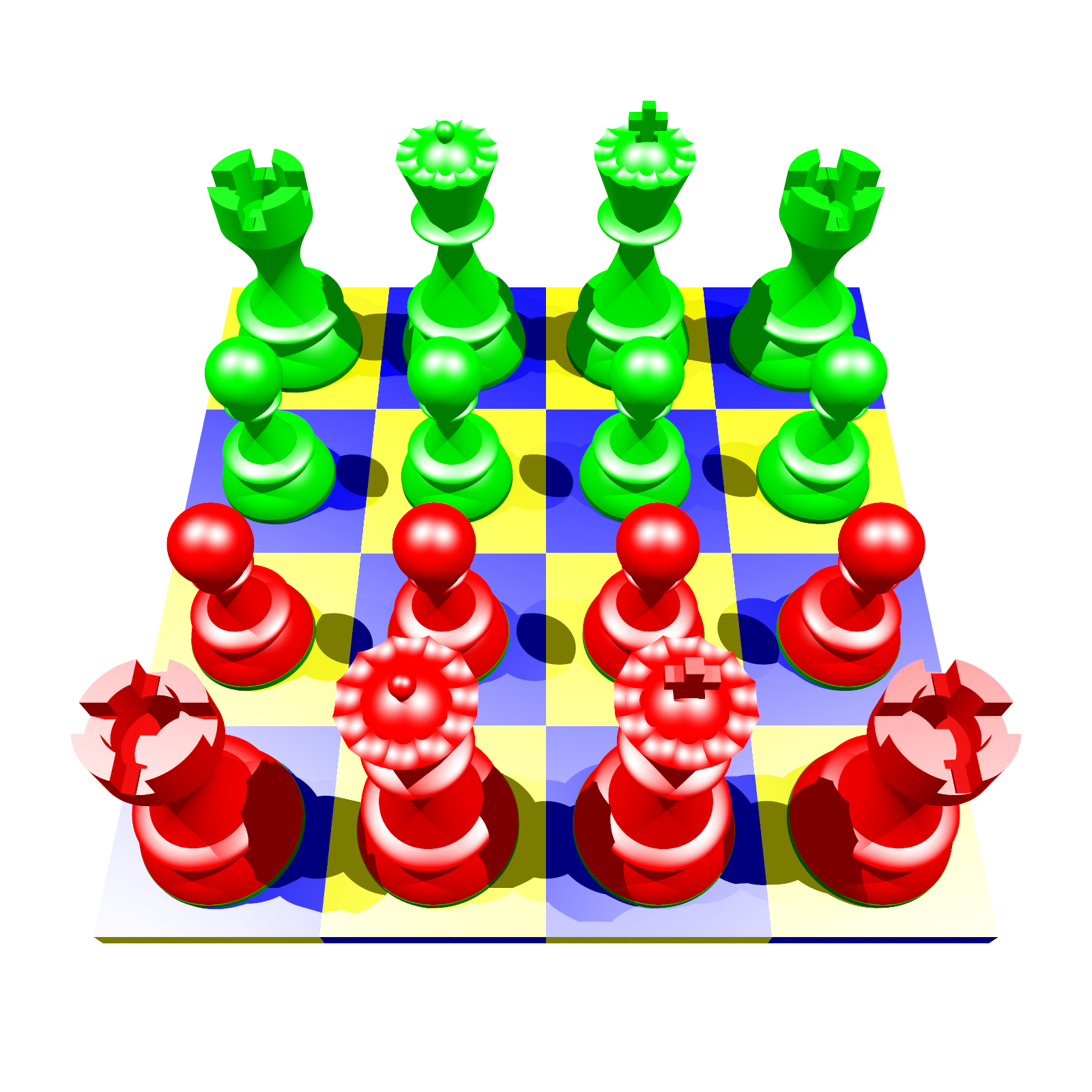}}
\label{Small Chess puzzles}
\caption{ 
A $4 \times 4$ board.
}
\end{figure}

\begin{figure}[!htpb]
  \includegraphics[width=12cm]{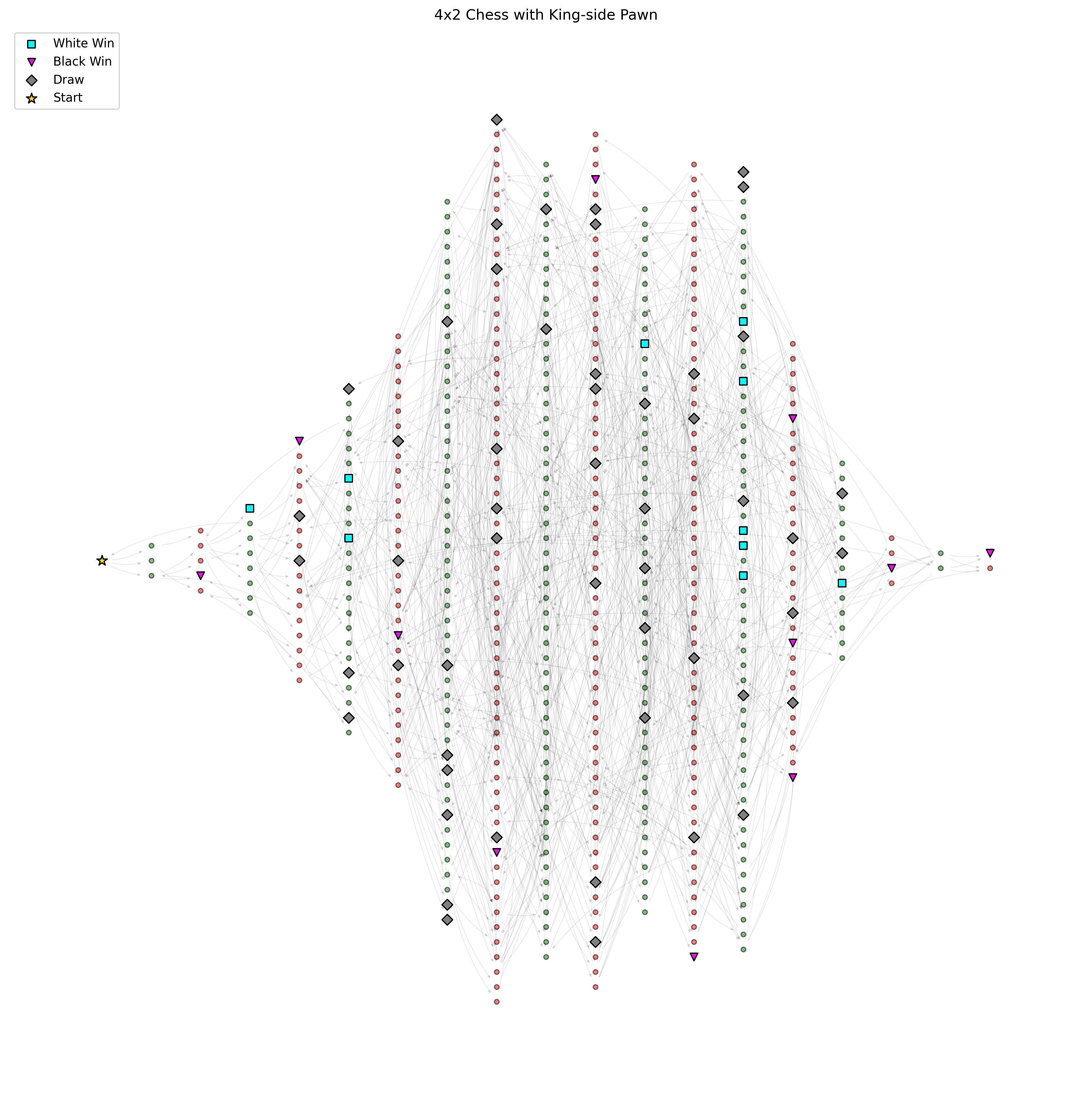}
  \label{kingside}
\caption{ 
The game graph for the king side pawn.
}
\end{figure}

\begin{figure}[!htpb]
   \includegraphics[width=12cm]{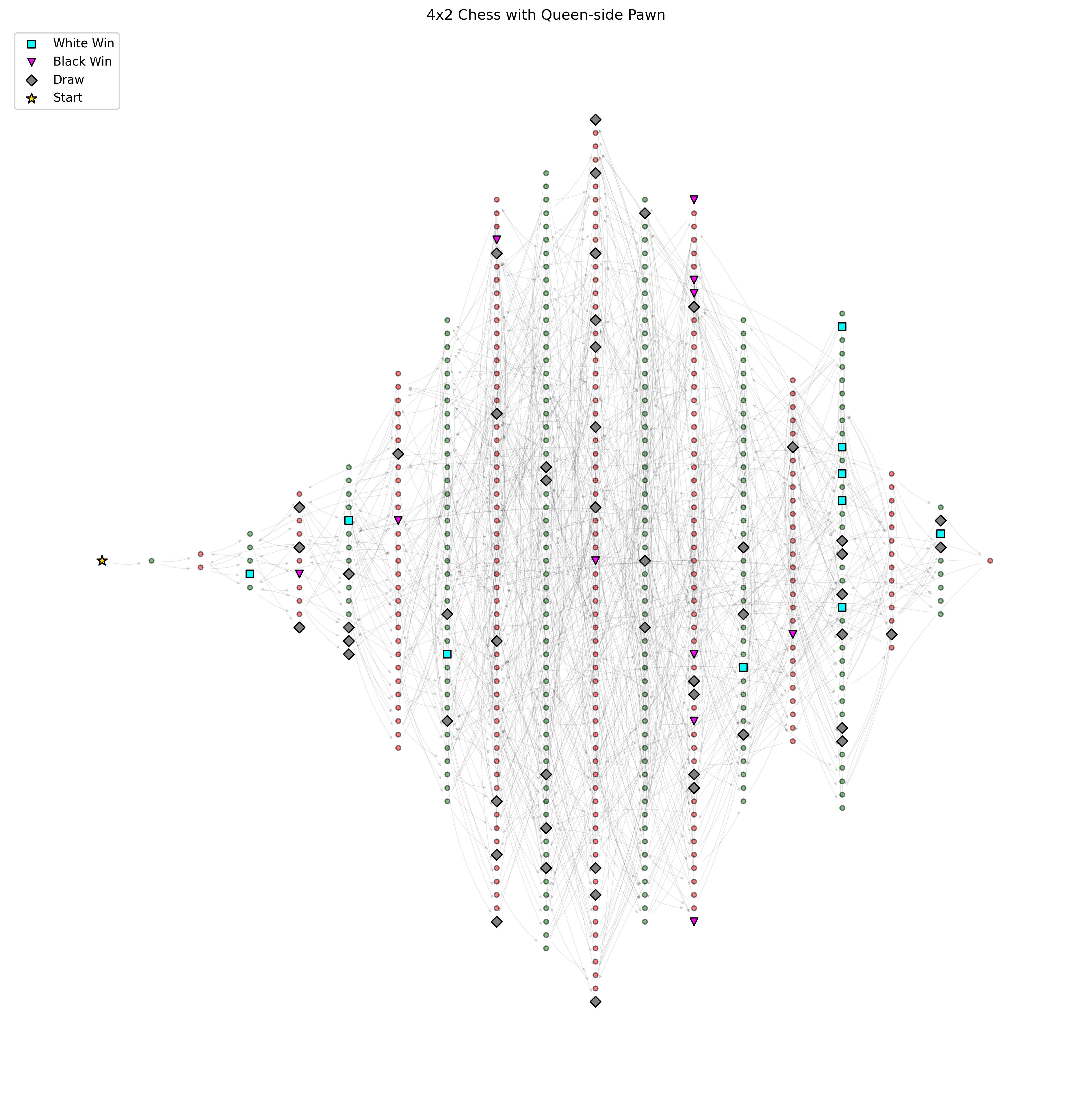} 
   \label{queenside}
\caption{
The game graph for the queen side pawn.
}
\end{figure}

\begin{figure}[!htpb]
   \includegraphics[width=12cm]{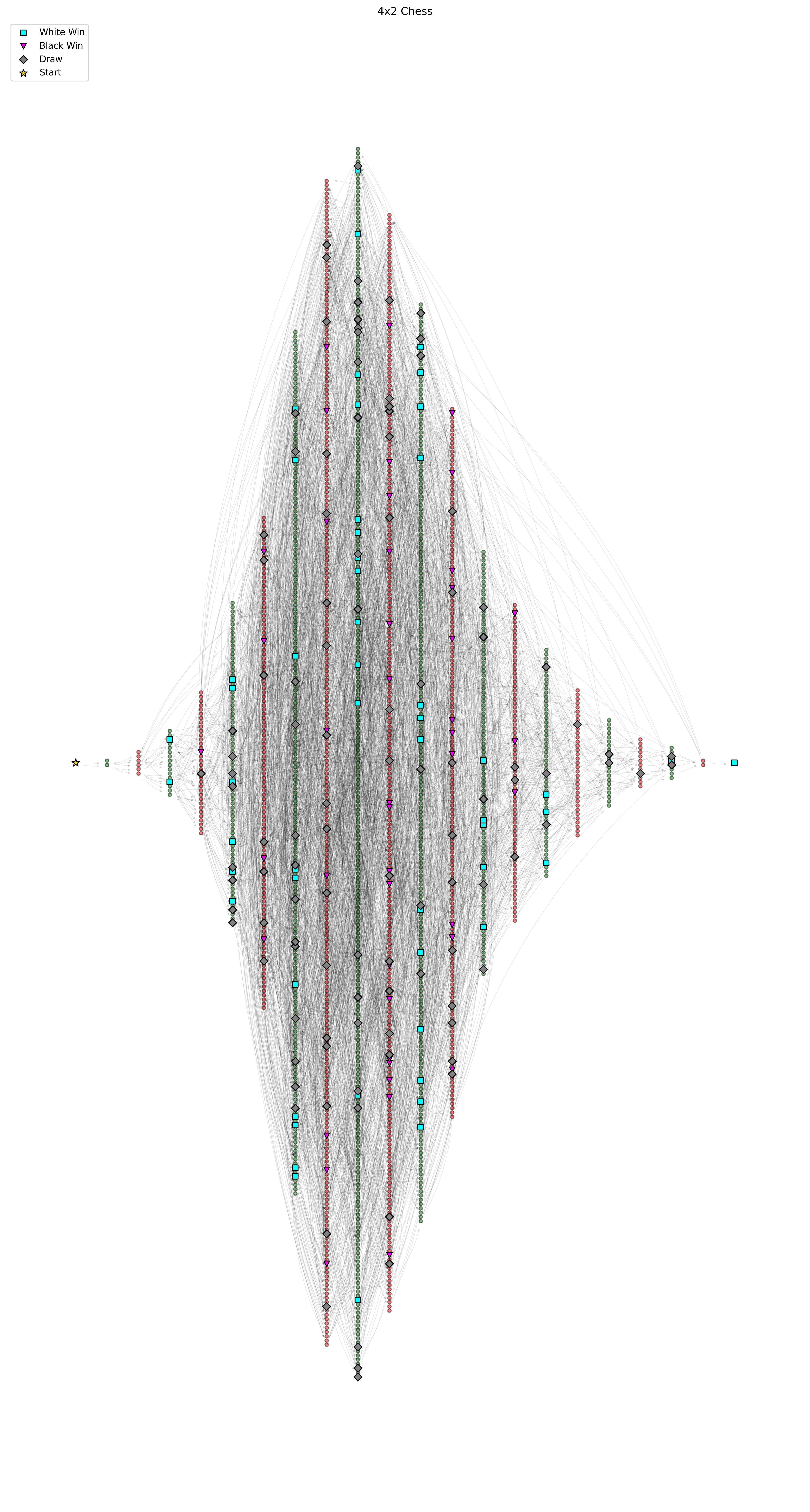} 
   \label{both}
\caption{
The game graph for both pawns
}
\end{figure}

\begin{figure}[!htpb]
   \includegraphics[width=12cm]{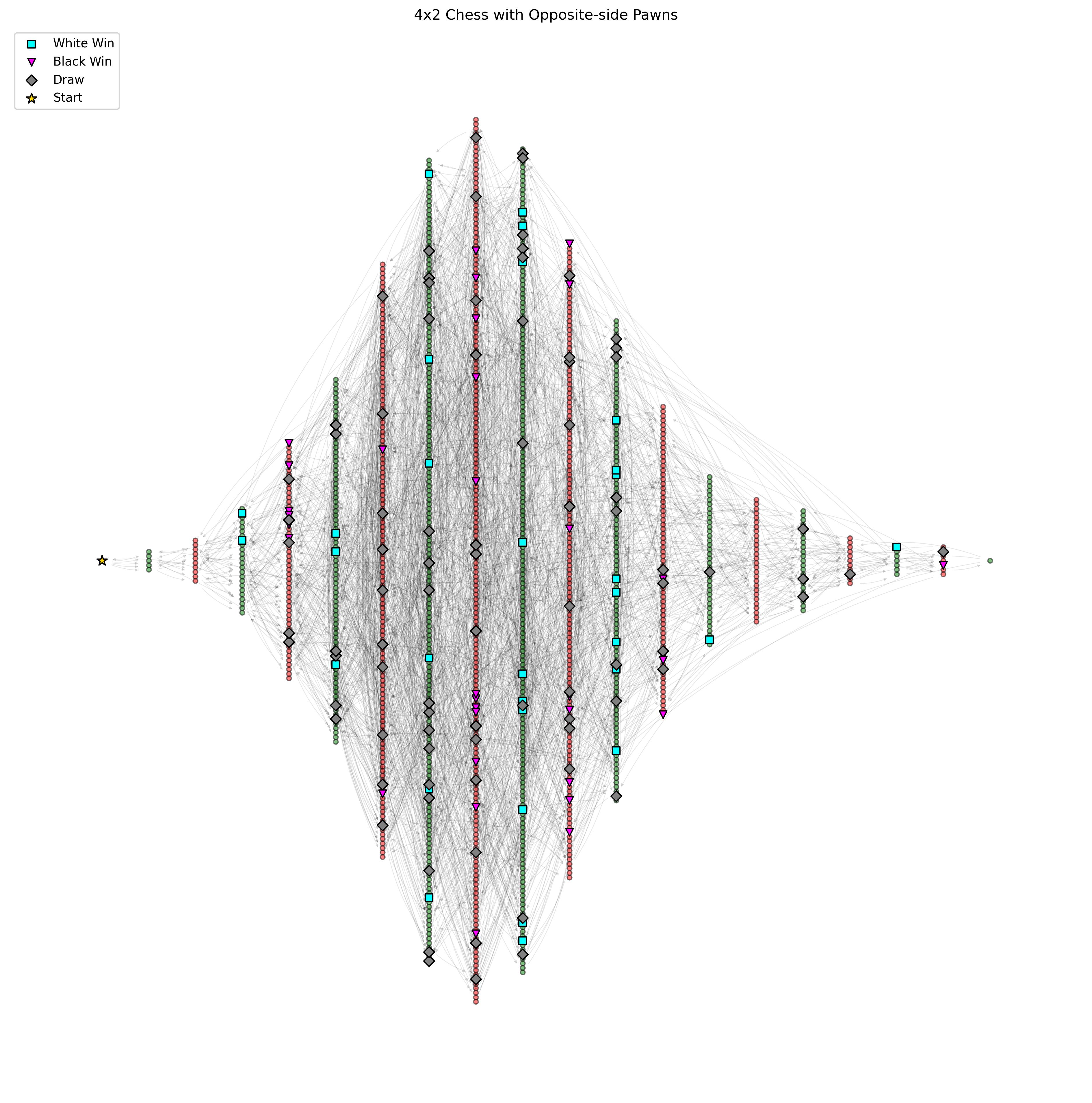}      
   \label{both}
\caption{
The game graph for opposite pawns
}
\end{figure}

\paragraph{}
Here is the $2 \times 4$ game without paws after 2 moves. 
This was generated in Mathematica:

\begin{figure}[!htpb]
   \includegraphics[width=15cm]{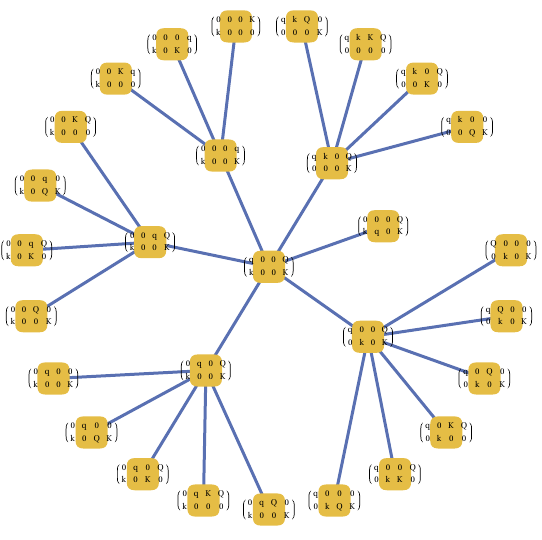}
   \label{no paws}
\caption{
The ball $B_1(v)$ of chess on a $2 \times 4$ board without paws. 
This game is "mate in 1". The queen can move in front of the king 
and force a chess mate. Lets call this position $w$. The winning strategy 
$H$ of $V$ is the subgraph $H$ generated by the edge $v \to w$. All other edges
are cut away. Player $W$ would like to maximize the length of a 
game play in $H$ but the edge degree of $w$ is zero because it is a mate. 
}
\end{figure}

\vfill
\pagebreak

\section{Card games}

\paragraph{}
Consider a deck with two cards: ace and king. Each player gets a card. 
The one who has the ace wins. The game here is represented
by a graph with two disconnected single points
$(V,E) = (\{A\},\{K\}, \emptyset)$. The {\bf god number} is zero. The initial condition
already is either win or lose. 

\paragraph{}
Take a deck with 4 cards $J,Q,K,A$.
Player V takes the first card, player B the second etc. The first one who has the ace
wins. The $V$-winning situations are the permutations, where $V$ is in an odd position.
The $W$-winning situations are the permutations, where V is in an even position. The god
number is $2$.

\begin{figure}[!htpb]
\scalebox{0.15}{\includegraphics{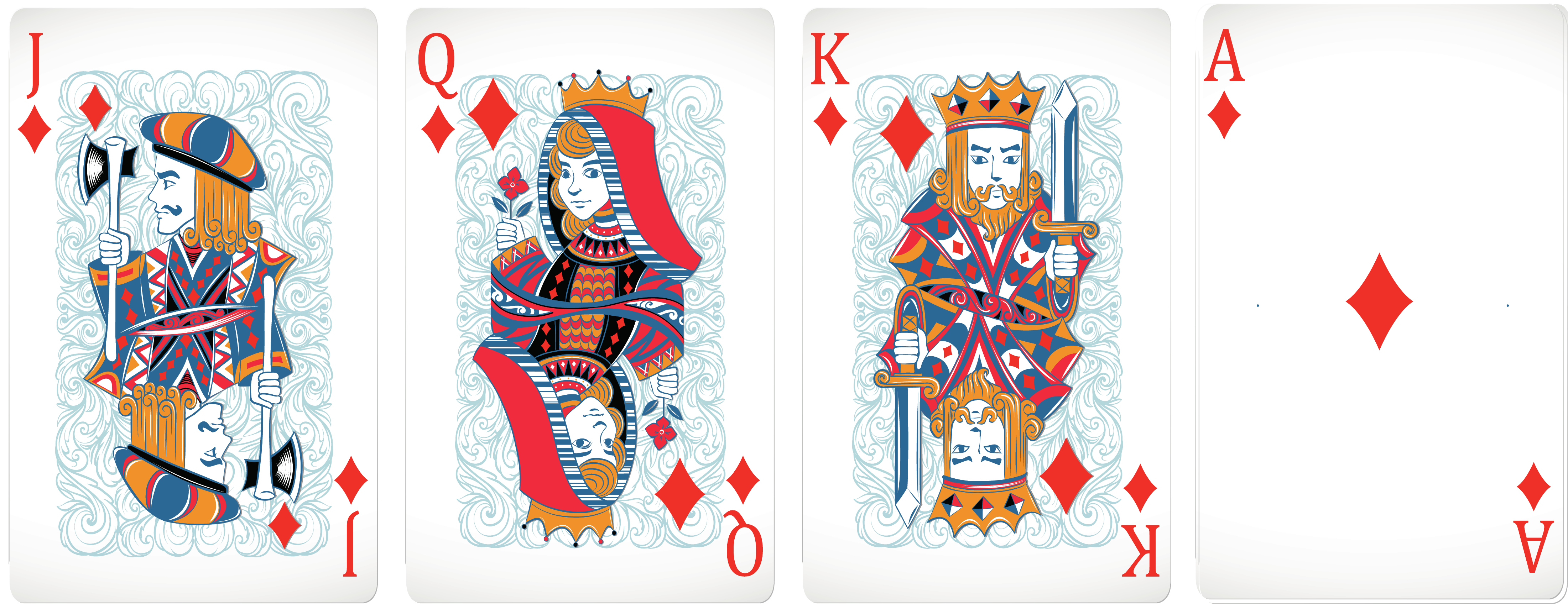}}
\label{A 4 card game}
\caption{
A 4 card game.
}
\end{figure}

\paragraph{}
Is there a simple {\bf Black Jack} 21 type game for which one has a chance
to write down the full game graph? We experimented with
an 11 version, where the highest card is 6 and only
two suits. So there are 5 cards A2,A3,A4,A5,A6 and B2,B3,B4,B5,B6.
What happens with any game with random shuffle is that we need to 
look at all the possible permutation cases in which the card deck 
can be dealt. That is already for such small games heavy, like 10!
\index{Black Jack}
\index{21}
\index{card games}

\paragraph{}
A simple 2-player card game is obtained by placing 3 pair of cards
on the table. A player can flip two adjacent cards if the left is larger
than the right. The cards are organized cyclic. 
The first player who can go from the initial condition 332211 to 112233 wins.
There are $6!/2^3=90$ different positions. Each player is allowed
to switch identical cards during a game so that we
break the parity. The graph diameter is 12, because 12 steps are needed to
order the deck. We probably need to limit the number of identical switches.
Variant: If we allow to switch any two cards (not only adjacent), the graph diameter
is 4. We still have to see whether there is a winning strategy here, especially
due to the passing. It could be that the parity decides.

\begin{figure}[!htpb]
\scalebox{0.73}{\includegraphics{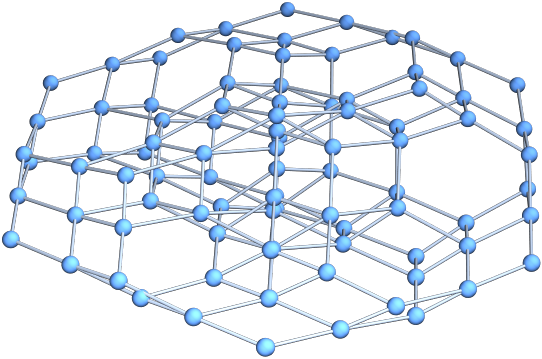}}
\label{Hex}
\caption{
The graph of the card game with 90 positions.
}
\end{figure}

\paragraph{}
To illustrate a 2-player card game, we use a variant of the transposition solitaire.
Choose $2$ queens, $2$ kings and $2$ aces from a standard $52$ card deck. The initial position $v$ is a random shuffle of this.
A player can choose to transpose any two adjacent cards, noting that we utilize a cyclic order. The game ends if the state $QQKKAA$ is reached. 
For example, with the stating arrangement $QAKAQK$, The legal moves are $\{(Q\,A),(A\,K),(K\,A),(A\,Q), (Q\,K), (K\,Q)\}$. 

\paragraph{}
The graph $G$ of the game has $\frac{6!}{2!\,2!\,2!}=90$ nodes and
edges $E$ as the set of pairs $(x,y)$ for which there is a $\tau$ with $\tau(x)=y$ or $\tau(y)=x$. 
Let us look at the neighborhood graphs with radius $d$ and center  $QQKKAA$.

\begin{center}
\begin{tabular}{|c|c|}
\hline
Distance $d$ & Number of nodes \\
\hline
0 & 1 \\
1 & 3 \\
2 & 9 \\
3 & 19 \\
4 & 29 \\
5 & 20 \\
6 & 9 \\
\hline
Total & 90 \\
\hline
\end{tabular}
\end{center}

\paragraph{}
Because every node has been reached after 6 steps, the god number of the game is $6$. Only very close to the 
solution, there is a win. There are three possible initial conditions for which the first player gets to a win. 
The winning initial shuffles for player $1$ to win are $AQKKAQ,\qquad QKQKAA,\qquad QQKAKA$
In all other positions, the second player can always stir away from the win. 

\begin{center}
\begin{tabular}{|c|c|}
\hline
Outcome under Optimal Play & Number of Starting Shuffles \\
\hline
Player 1 wins & 3 \\
Player 2 wins & 0 \\
Draw & 86 \\
\hline
Total & 90 \\
\hline
\end{tabular}
\end{center}

\paragraph{}
The condition of being unable to repeat the transposition used on the previous term means the game is not solely 
determined by the arrangement of the cards, each player much also track which 
vertices were swapped. There are six possible remembered edges,
$$ (1\,2),(2\,3),(3\,4),(4\,5),(5\,6),(6\,1) $$
together with one initial state in which no move has yet been played. 
Thus every arrangement contains one of $7$ memory indicators. Therefore the full state space has size
$$ 90 \cdot 7 = 630  \; . $$
Each state contains not just the arrangement of the card, but also information on which player's move it is, 
and which transposition was made on the previous turn. Computing every state gives the following results, 
\begin{center}
\begin{tabular}{|c|c|}
\hline
Type of expanded State & Count \\
\hline
Winning for player to move & 18 \\
Losing for player to move & 7 \\
Draw & 605 \\
\hline
Total & 630 \\
\hline
\end{tabular}
\end{center}

\paragraph{}
The hexagon's dihedral group $(D_6)$)  works on our initial $90$ 
configurations through rotations and reflections the board. 
Under these transformations, our $90$ positions fall into exactly $11$ symmetry classes,
\index{hexagon dihedral group}

\begin{center}
\begin{tabular}{|c|c|}
\hline
Orbit Size & Number of Orbits \\
\hline
6 & 7 \\
12 & 4 \\
\hline
Total Orbits & 11 \\
\hline
\end{tabular}
\end{center}

We see that many positions are geometrically equivalent.

\vfill
\pagebreak

\section{Rubik type games}

\paragraph{}
For group games and in particular Rubik type games, we work in a {\bf finitely presented
group} $\langle S | R \rangle = \langle g_1, \dots g_m | r_1 \dots r_m \rangle$,
where $S$ are the {\bf generators} and $R$ are the {\bf relations}.
A group can also be defined by giving a finite set of generators in a permutation group.
This is the {\bf permutation representation}.
One usually assumes that the {\bf generator set} $S$ is invariant under involutions. 
This gives god number $20$ for the Rubik cube.
\index{finitely presented group}
\index{generator set}
\index{relation set}
\index{Rubik cube}

\paragraph{}
Many geometric shapes have been made into Rubik type puzzles in particular various $n \times m \times k$
generalizations, or tetrahedra, octahedra or dodecahedra. There are version where the coloring has
changed like {\bf bi-colored cubes} \cite{MillerPhaovibulShen}.
One of the famous examples is the $2 \times 2 \times 2$ Rubik cube already described in \cite{Singmaster1981}.
(it is also  known as the``pocket cube") and where we only allow moves
of 180 turns of two faces L and U. This group has 6 elements and is isomorphic to
$\langle a=(1,2), b= (2,3)  | a^2=b^2=(ab)^3 \rangle = S_3$.
The god number is $3$. For the full pocket graph, the Cayley graph can be constructed
by GAP and the god number determined. See the code section. 
\index{pocket cube}

\paragraph{}
The $2 \times 2 \times 1$ Rubik is called the "Floppy". Its group has 2304 elements. 
The god number is 13. If we work with the floppy, where one of the cubes is fixed, the
number of positions drops to 192 and the God number is 6. 
There are $k \times l \times n$ versions of Rubik cubes as seen 
in Figure~\ref{Rubik}.

\begin{figure}[!htpb]
\scalebox{0.45}{\includegraphics{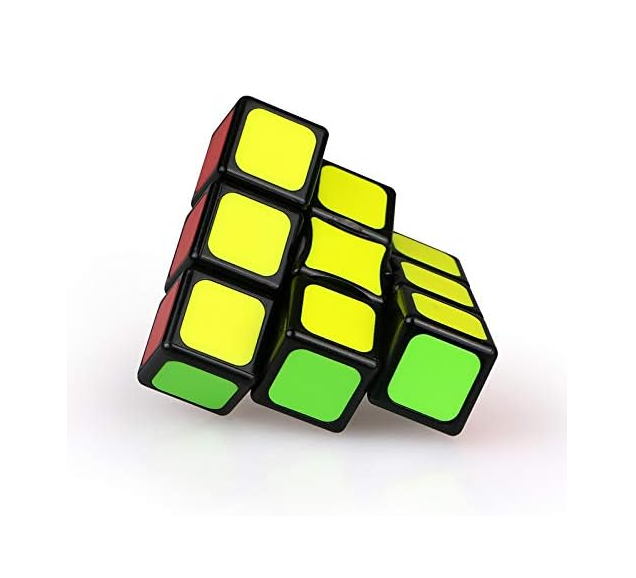}}
\scalebox{0.55}{\includegraphics{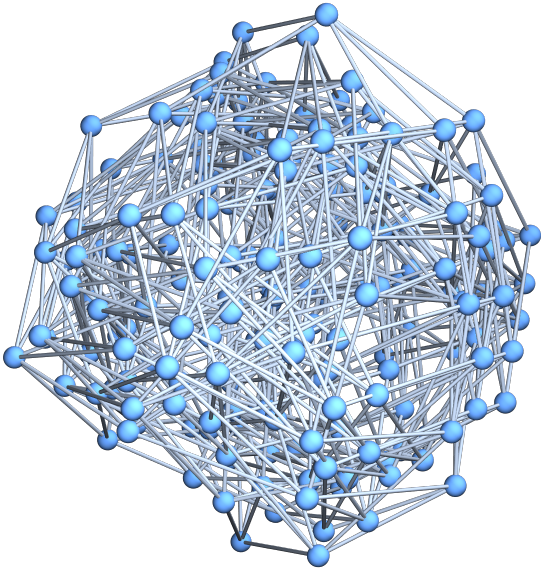}}
\label{The floppy}
\caption{
The {\bf floppy} was designed by Katsuhiko Okamoto. 
}
\end{figure}
\index{floppy}

\paragraph{}
For the $2 \times 2 \times 2$ cube case, the god number is $11$ respectively $14$ is known. 
There are $3'674'160$ positions for the $2 \times 2 \times 2$ (pocket) cube
in the half turn metric the god number is $11$ (analog to $20$ in the 
standard Rubik $3 \times 3 \times 3$ cube).
If one of the cubes is fixed, the number of positions reduces to $8! 3^7$
which is $1/24$ of $3'674'160$.
In the quarter turn metric, the god number of the pocket cube is 
$14$ (this is the analog to the god number $26$ in the $3x3x3$ case).
The GAP code in the code section was able to get this number within a few hours of
run time. 
The subgroup of the full $3 \times 3 \times 3$ Rubik generated by 
180 degree points is a group with 663552 elements and has 
god number 15.
\index{quarter turn metric}

\paragraph{}
Simpler and older than Rubik are interlocked rings, 
called the {\bf Hungarian rings}. There are two cyclic permutations
generating a subgroup of the permutation group such that the 
two cycles intersect in 4 points. 
There are has $38!$ possible positions. Gap was unable to compute 
the diameter in this case. Smaller Hungarian rings can be done.
\index{Hungarian rings}

\begin{tiny}
\lstset{language=Mathematica} \lstset{frameround=fttt}
\begin{lstlisting}[frame=single]
LoadPackage("grape");
a:=(1,38,37,36,35,6,34,33,32,31,30,29,28,27,26,25,24,23,22,21);
b:=(1,20,19,18,17,16,15,14,13,12,11,10,9,8,7,6,5,4,3,2);
G:=Group(a,b); Size(G);
S:=GeneratorsOfGroup(G); C:=CayleyGraph(G,S); Diameter(C);
\end{lstlisting}
\end{tiny}

\begin{figure}[!htpb]
\scalebox{0.25}{\includegraphics{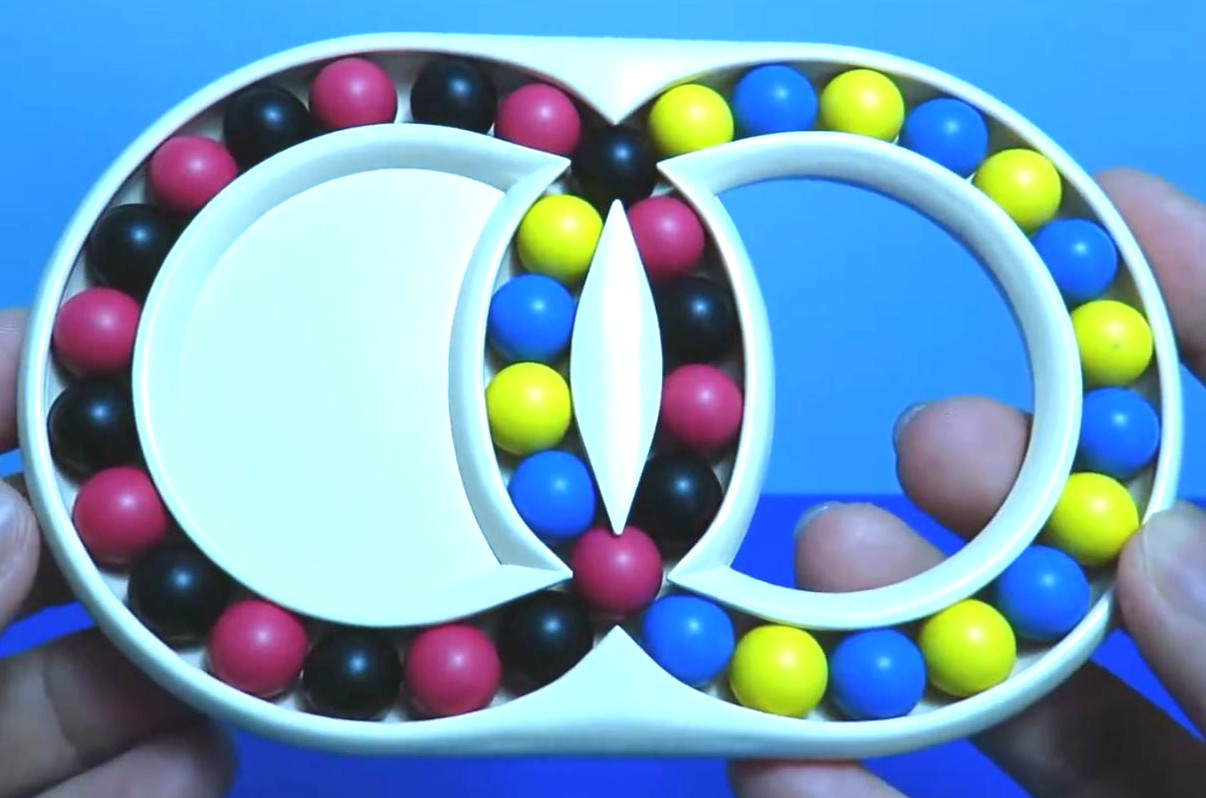}}
\label{Hungarian Rings}
\caption{
Hungarian rings
}
\end{figure}

\paragraph{} 
For the {\bf Rubik's clock}, the god number is known to be $12$.
There are $2^{14}$ positions. The Rubik's clock has lots of generators: 
there are $2^4$ flag positions and in each case $4$ generators. So that
there are $2^6=64$ generators. 
\index{Rubik's clock}

\begin{figure}[!htpb]
\scalebox{0.6}{\includegraphics{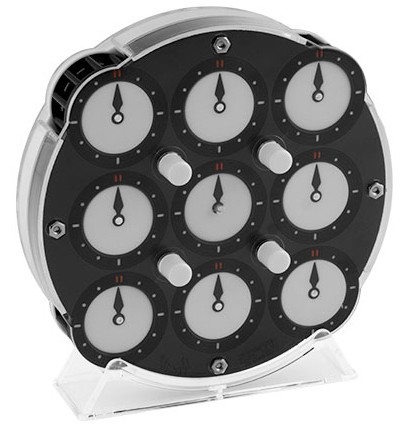}}
\label{clock}
\caption{
Rubik's clock has many generators. 
}
\end{figure}
\index{Rubik's clock}

\vfill
\pagebreak

\section{Tic Tac Toe games}

\paragraph{}
The standard $3^2$ Tic-Tac-Toe is famously known to be a draw \cite{Gardner1959}.
The $n^d$-game looks for a line of length $n$ in a $d$-dimensional hypercube of side $n$. 
The two-dimensional {\bf $m,n,k$-game} ($k$ in a row on an $m\times n$ board) formalise positional 
tic-tac-toe \cite{BerlekampConwayGuy}. 
Hales and Jewett \cite{HalesJewett} used strategy-stealing to show no second player wins an $n^d$-game, and 
that $n^d$ is a first-player win, once $d$ is large relative to $n$ (like $d>3^n$).  
A particular example mentioned in \cite{HalesJewett} is {\bf Gomoku}, which is a {\bf 5 in a row game}. 
It is usually played on a $15 \times 15$ or earlier as $19 \times 19$ board.
\footnote{Gomoku means five pieces in Japanese. It is also called
{\bf Wuziqi} in China or {\bf Omok} in Korea of {\bf Gobang} in Britain}
Patashnik \cite solved the $4^3$ "Qubic". The first-player wins
Allis \cite{Allis1988,Allis1993} showed free-style Gomoku $(15,15,5)$ is a first-player win.
Already the $5 \times 5 \times 5$ Tic-Tac-Toe with $3^{125}$ positions is unsolved 
\cite{BeckCombinatorialGames}.
\cite{HsuKoWu2018} computationally established $(7,7,5)$ and $(8,8,5)$ as draws. 
\index{connect four}
\index{tic tac toe}
\index{Gomoku} 
\index{Omok}
\index{Qubic}
\index{Wuziqi}

\begin{figure}[!htpb]
\scalebox{0.35}{\includegraphics{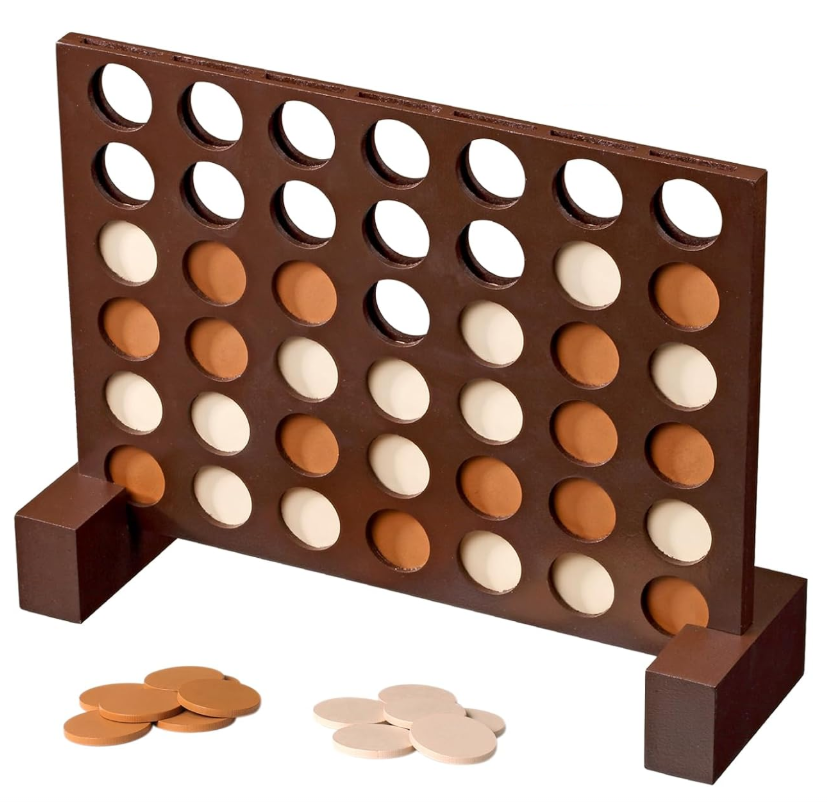}}
\label{Connect 4}
\caption{
Connect 4 game.
}
\end{figure}

\paragraph{}
For $n^d$ games, the following is known:

\begin{center}
\begin{tabular}{@{}llll@{}}
\toprule
Size & God's number & First-player win? & Reference \\
\midrule
$1^d$ & 1 & Yes & trivial \\
$n^1$, $n\geq 2$ & --- & No (draw) & trivial \\
$2^d$, $d\geq 2$ & 3 & Yes & folklore (strategy-stealing) \\
$3^2 = (3,3,3)_{2D}$ & --- (draw at move 9) & No (draw) & [5] \\
$3^3$ (3D, $3\times 3\times 3$) & not precisely established & Yes & [8, 3] \\
$4^3$ (Qubic) & not precisely established & Yes & [7, 8] \\
$n^d$, $d \gg n$ & --- & Yes & [6] \\
$n^d$, $n$ odd, $d \leq f(n)$ & --- & No (draw, pairing) & [6] \\
\bottomrule
\end{tabular}
\end{center}

\paragraph{}
For two dimensional $m,n,k$-games, the God's number is given in brackets:

\begin{center}
\footnotesize
\begin{tabular}{@{}lccccc@{}}
\toprule
Board & $k=2$ & $k=3$ & $k=4$ & $k=5$ & $k=6$ \\
\midrule
$3\times 3$ & Win (2) & Draw & --- & --- & --- \\
$4\times 4$ & Win (2) & Win (3) & Draw & --- & --- \\
$5\times 5$ & Win (2) & Win (3) & Draw & Draw & --- \\
$6\times 5$ & Win (2) & Win (3) & Win & Draw & --- \\
$7\times 7$ & Win (2) & Win (3) & Win & Draw & Draw \\
$9\times 6$ & Win (2) & Win (3) & Win & Draw & Draw \\
$15\times 15$ & Win (2) & Win (3) & Win & Win (Gomoku) & Draw \\
\bottomrule
\end{tabular}
\end{center}

\paragraph{}
From \cite{BeckCombinatorialGames}: 
$k=1,2$ are trivial wins except for $(1,1,2)$ and $(2,1,2)$; $(m,n,3)$ is a win if and only if 
$(m \geq 4, n \geq 3)$ or $(m \geq 3, n \geq 4)$, else draw. 
First general result: $(m,n,k)$ is a draw whenever $m = n = k \geq 3$. 
$(m,n,4)$: draw for $m,n \leq 5$, $(6,5,4)$ win, $(m,4,4)$ win if and only if 
$m \geq 9$; $(m,n,5)$: draw for $m,n \leq 8$; $(15,15,5)$ win, $(9,6,6)$ and $(7,7,6)$ draw via pairings.

\paragraph{}
Tic-Tac-Toe Games are also played {\bf with gravity}:
impose a vertical axis and require pieces to stack from the bottom of each column. It is
the $n^d$ analogue of Connect Four. Strategy-stealing still rules out second-player wins on 
$n^d$, but values change for some 2D $m,n,k$ sizes.
God numbers also can be given in draw situations. In the $3^2$ case, a draw reached at move 
9 as the board is filled.
\index{gravity puzzles}
\index{tic tac toe with gravity}

\begin{center}
\footnotesize
\setlength{\tabcolsep}{4pt}
\renewcommand{\arraystretch}{1.0}
\begin{tabular}{@{}l l l c c l@{}}
\toprule
Size & G & H (gravity) & {\rm God}(G) & {\rm God}(H) & Ref \\
\midrule
$1^d$ & Win & Win (identical) & 1 & 1 & trivial \\
$n^1$, $n\geq 2$ & Draw & Draw (moves forced) & --- & --- & trivial \\
$2^d$, $d\geq 2$ & Win & Win (corner: $2^{d-1}$ threats) & 3 & 3 & folklore \\
$3^2$ & Draw & Connect-3 on $3\times 3$; not in Tromp's tables & 9 & ? & --- \\
$3^3$ & Win & Connect-3 on $3\times 3\times 3$; status open & ? & ? & --- \\
$4^3$ (Qubic) & Win~[8] & Open: Connect-Four 3D on $4^3$ unsolved~[10] & ? & ? & [10] \\
$(5,5,4)$ & Draw & Draw & --- & --- & [9] \\
$(6,5,4)$ & Win & Draw & ? & --- & [9] \\
$(6,6,4)$ & Win & Second-player win & ? & ? & [9] \\
$(7,6,4)$ Connect 4 & Win & Win, forced by move 41 & ? & 41 & [1, 12] \\
\bottomrule
\end{tabular}
\end{center}
\index{Tromp's tables}

\paragraph{}
Gravity preserves the god's number for $1^d$, $n^1$, and $2^d$, but can decrease the number of playable, 
winnable lines for the first player. Qubic without gravity is a first-player win \cite{Patashnik1980}. 
With gravity ($4\times 4\times 4$ {\bf Connect-Four 3D} was solved independently by Allen and Allis in Oct 1988.
For $7\times 6$), it remains open. 
Gravity can even flip $(6,6,4)$ from a first-player win to a second-player win 
\index{qubic}

\paragraph{}
For $q=2$, we have looked at the $2 \times 2$ case, which is win for the starter with
god number 2.  The case $3 \times 3$ which is always a tie with ${\rm ceiling}(n^2/2)$
god number. For $3 \times 3 \times 3$, there is always win.

\vfill
\pagebreak

\section{Hanoi puzzles}

\paragraph{}
The {\bf tower of Hanoi} $T_3(n)$ was posed by Edouard Lucas in 1883 with the legend of 
64 golden disks~\cite{lucas1883}. In this three-peg $p=3$ situation, the god number is $2^n-1$. 
It is a standard example in many programming textbooks like \cite{Roberts2006}.
Dudeney~\cite{dudeney1907} introduced the {\bf Reve's puzzle} $T_4(n)$, which is 
the four-peg version.  Frame and Stewart~\cite{frame1941,stewart1941} 
independently proposed a divide-and-conquer 
recurrence for $p \ge 4$, later shown equivalent to seven other algorithms \cite{hkmp}. 
\cite{chenshen} (building on Szegedy 1999) established the correct order of 
magnitude. Bousch \cite{bousch2014} proved that Frame-Stewart is optimal for $p=4$.
We see below the graphs for $p=4$ and $n=2,3,4,5,6$.
Optimality for $p \geq 5$ is open. See the book \cite{hkmp}.
\index{tower of Hanoi}
\index{golden disks}
\index{Reve's puzzle}

\begin{center}
\small
\begin{tabular}{@{}p{4.3cm} p{6.3cm} p{2.5cm} p{2cm}@{}}
\toprule
Variant & $T_p(n)$ / god's number & Status & Reference \\
\midrule
Lucas 1883     $p=3$ & $2^n-1$                   & Proven          & \cite{lucas1883,hkmp} \\
Reve's puzzle, $p=4$ & Frame--Stewart recurrence & Proven & \cite{bousch2014} \\
$n=1$                & $1,3,5,9,13,17,25,33,41,49$ & Proven & \cite{oeisA007664} \\
Multi-peg, $p \ge 5$  & Frame--Stewart recurrence & Conjectured & \cite{frame1941,stewart1941,chenshen} \\
F-S formula         & $F(p,n)=\min_{1\le i<n}\{2F(p,i)+F(p{-}1,n{-}i)\}$ & --- & \cite{hkmp} \\
asymptotic order     & $2^{(1\pm o(1))\,(n(p-2)!)^{1/(p-2)}}$ & Proven & \cite{chenshen} \\
Cyclic (directional) Hanoi, $p=3$ & exponential, closed recurrence & Proven & \cite{atkinson1981} \\
Arbitrary-to-arbitrary, $p=3$ & diameter of $H^n_3$ still $2^n-1$ & Proven & \cite{berend2006} \\
Arbitrary-to-arbitrary, $p\geq 4$ & diameter unknown in closed form & Open & \cite{berend2006} \\
\bottomrule
\end{tabular}
\end{center}

\paragraph{}
The game graph is the  {\bf Hanoi graph} $H_p^n$ in which the vertices are the $p^n$ 
strings in $\{1,\dots,p\}^n$ encoding the disk $i$ on peg $s_i$. 
The order on each peg is forced to be decreasing in size) and edges connect 
configurations differing by one legal move.
The diameter of $H_3^n$ is $2^n-1$, attained between any two perfect states~\cite{hkmp}.
$|V(H^n_3)|=3^n$ and $|E(H^n_3)|=\frac{3(3^n-1)}{2}$. \cite{sgs1944}.
\index{Hanoi graph}

\paragraph{}
$H^n_3$ is three copies of $H^{n-1}_3$ joined at their extremal vertices by bridging edges~\cite{hkmp}.
It is planar, a penny graph, and its drawing converges to 
the Sierpinski triangle~\cite{hinz1992}. 
The vertex set coincides with the odd entries of the first $2^n$ rows of {\bf Pascal's triangle}.
Average distance, chromatic, and metric properties are given in \cite{hkmp}.
The graph $H_p^n$ has $p^n$ vertices; its diameter is unknown for $p \geq 4$. 
For $p>4$ with $n>0$, or $p=4$ with $n>2$, $H^n_p$ is non-planar~\cite{hinz2002}
\index{Sierpinski triangle}
\index{Pascal triangle}

\begin{figure}[!htpb]
\scalebox{0.15}{\includegraphics{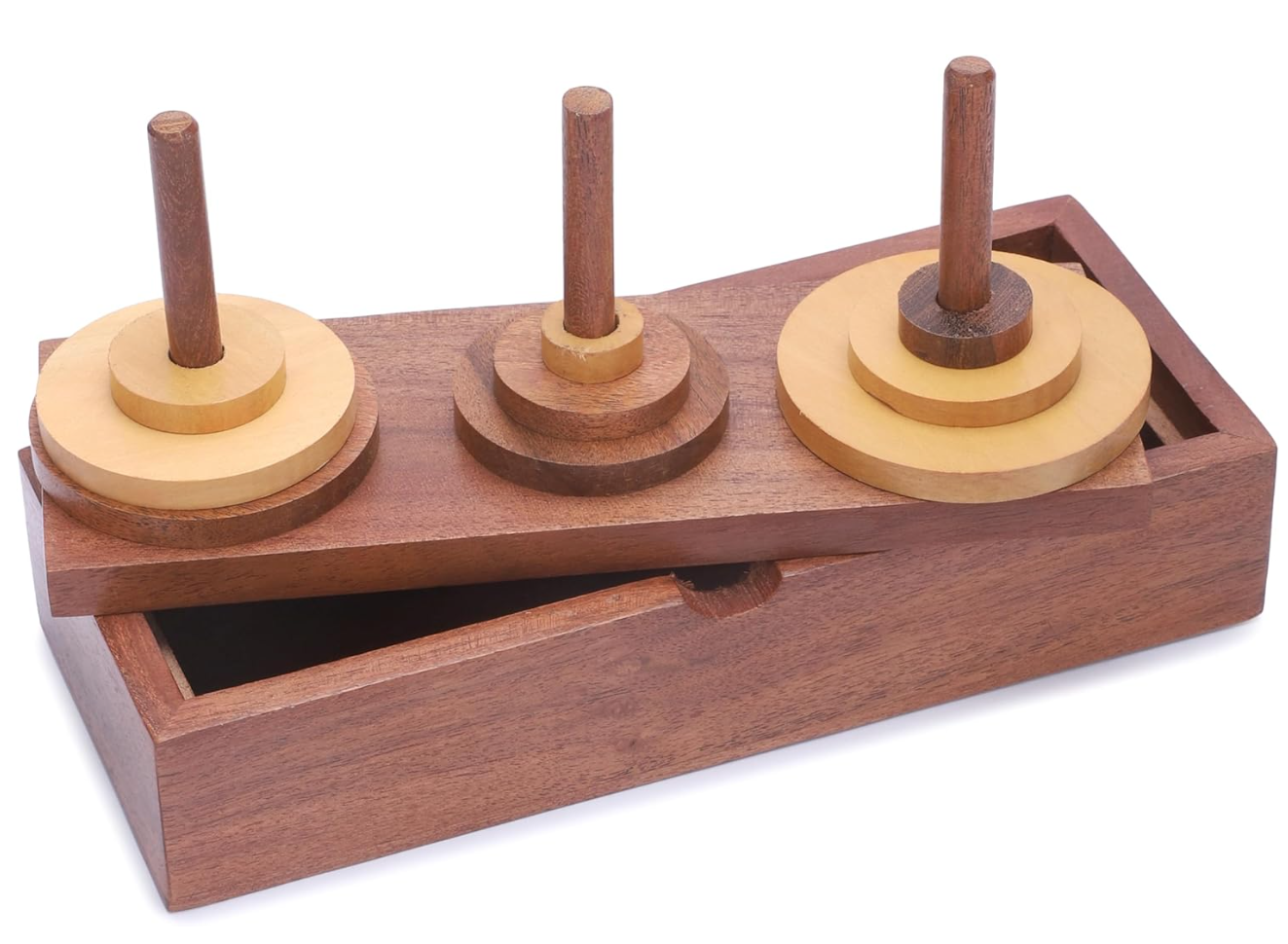}}
\label{Tower of Hanoi}
\caption{
The tower of Hanoi with n pieces has god number $2^n-1$ }
\end{figure}

\begin{figure}[!htpb]
\scalebox{0.69}{\includegraphics{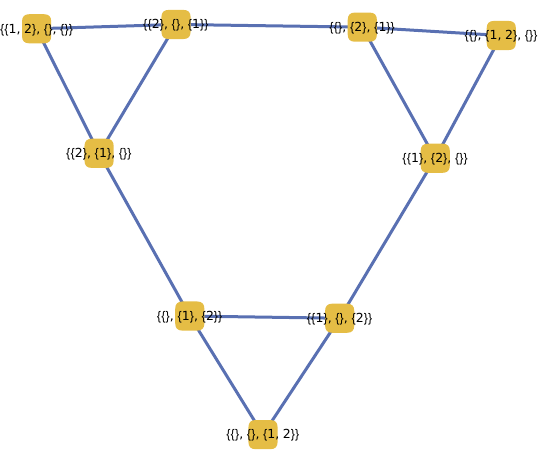}}
\scalebox{0.69}{\includegraphics{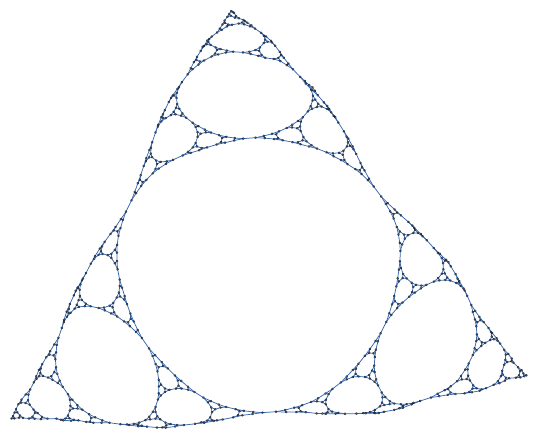}}
\label{Tower of Hanoi}
\caption{
The Hanoi graph in the case $n=2$ has $9$ vertices, $12$ edges
and diameter $3$. We see also the graph for $n=6$, where we have
$|V|=729$ and $|E|=1092$. The positions where all disks are on one peg
are on the boundary. 
}
\end{figure}

\begin{figure}[!htpb]
\scalebox{1.2}{\includegraphics{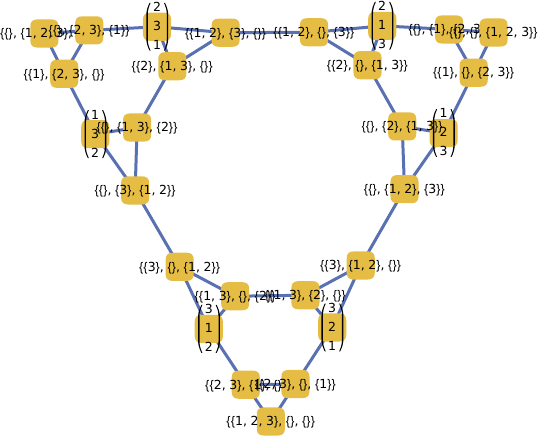}}
\label{Tower of Hanoi for n=3}
\caption{
The Hanoi graph for n=3. Going along the boundary is 
the optimal path. It recursively first uses the 
solution for n=2, then moves the entire pile and then
again uses the solution for n=2. 
}
\end{figure}

\paragraph{}
The Hanoi problem with $p=4$ pegs is not explored for large n. 
The 4 peg problem is a more interesting transport problem 
as one can now work more effectively using two "storage pegs". 
The game graph obviously has a $S_4$ symmetry, which leads to a
tetrahedral symmetry when plotted in $3$- dimensional space $\mathbb{R}^3$. 

\begin{figure}[!htpb]
\scalebox{0.65}{\includegraphics{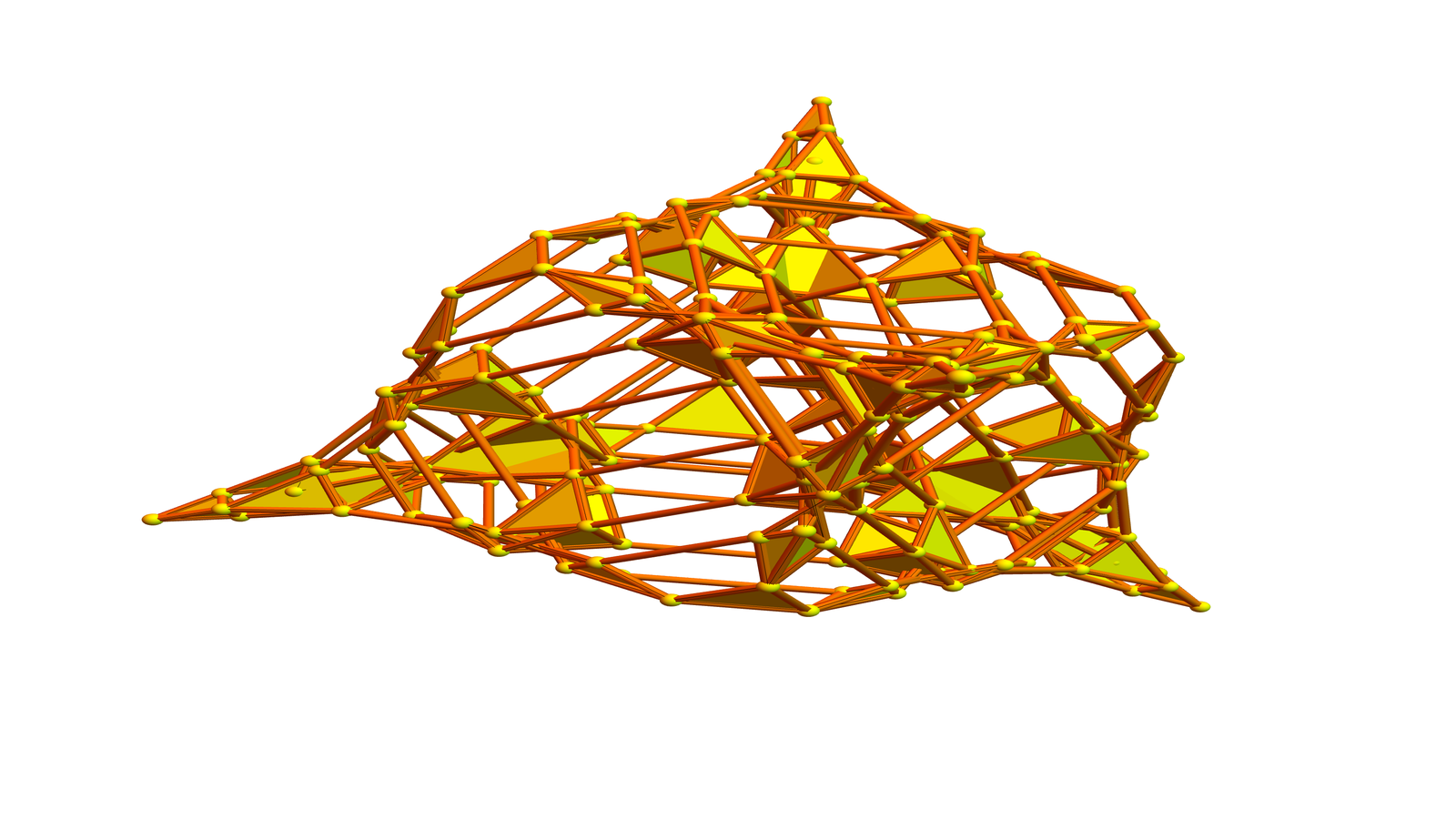}}
\label{Reeves puzzle}
\caption{
The graphs of the Reeves puzzle, ($p=4$-peg Hanoi puzzle) have smaller diameter
as we have more transport possibilities. We measure the graph
diameters ${\rm God}(2)=3,{\rm God}(3)=5,{\rm God}(4)=9,{\rm God}(5)=13, {\rm God}(6)=17$. We see
here the case $n=4$. 
}
\end{figure}

\vfill
\newpage

\section{Other games}

\paragraph{}
{\bf Hex} \cite{Hex2022} is a two-player game that has relations to theorems in topology.
It can be played on any rectangle \cite{AlbertNowakowskiWolfe}. It was invented in 1942 by Piet Hein
and rediscovered by John Nash in 1948. Hex is a paradigm because of the {\bf Hex theorem}  telling that
it never leads to a draw. It has been used to prove the Brouwer fixed point theorem \cite{Gale}, who
also pointed out that it is also related to the Jordan curve theorem. David Gale also covert the 
n-dimensional Hex theorem on a $n$ dimensional cube which belongs to a n-person Hex game. Gale writes
{\it From the game-theoretic point of view, n-person Hex is probably not of much interest. 
Aside from the mechanical problem of designing a practical board there is the usual difficulty of
multi-person games that players may form coalitions, requiring more sophisticated analysis how
how such coalitions can be formed and play against each other.}
\index{David Gale}
\index{Hex game}
\index{Jordan curve theorem}
\index{Brouwer fixed point theorem}

\paragraph{}
Hex type games can be played on any finite simple graph $\Gamma$. Each player selects out
an edge. The first one who has a closed simple loop of the same color wins.
A famous variant is a square grid on a torus. Player A needs to build a path from
left to right not crossing any path of B. Player B needs to build a path from
bottom to top to bottom. The  most famous variant is "Hex" which has been popularized
by Nash (but Nash would play on a square board, not a hexagonal lattice). 
It is an example where the first player has a winning strategy.
\index{Hex game}
\index{Nash}
\index{topological games}

\begin{figure}[!htpb]
\scalebox{0.15}{\includegraphics{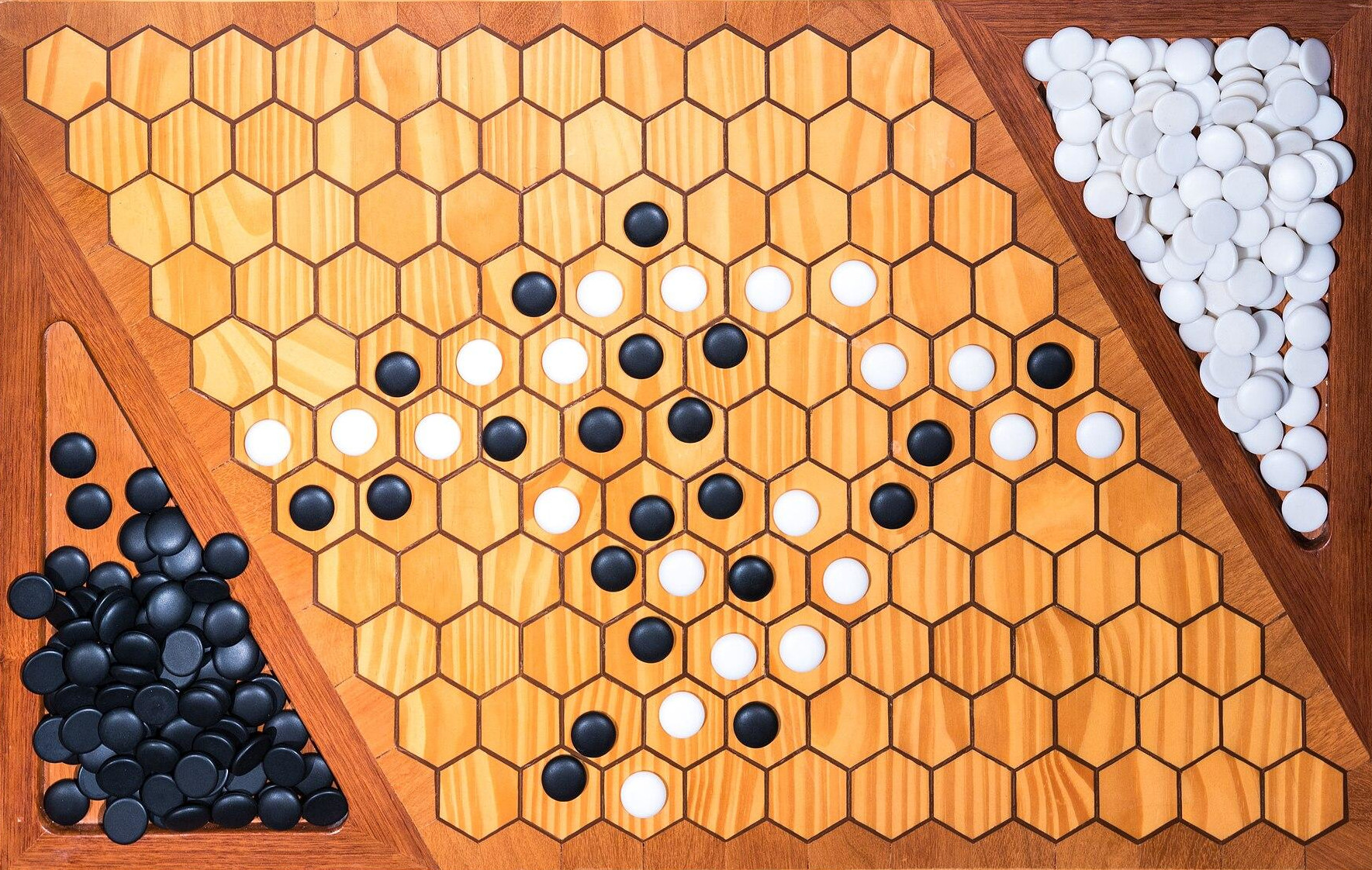}}
\label{Hex game}
\caption{
In hex, there is no draw.
}
\end{figure}

\paragraph{}
{\bf Hex on a torus} is considered to be harder than the
standard hex. Also here, the advantage of the first player
can be balanced with the pie rule telling that after the first
move the second player can decide whether to take over the
position of the first. Also for Torus Hex
there is no tie because of topology. Also here,
The first player always wins because if the second
had a strategy, it could be used by the first without
the first stone. The topology is related to
Brower fixed point theorem and Sperner lemma.
God number is larger or equal to $n$.
The $3 \times3$ Hex board is isomorphic to Tic-Tac-Toe.
It is also related to the P-NP problem.
\index{Hex on a torus}
\index{Torus hex}
\index{Brower fixed point}
\index{Sperner lemma}

\paragraph{}
Related to Hex is the {\bf Shannon switching game}. As Martin Garner reported in one
of his Scientific American puzzles and diversions articles, the game has first been described
by Claude Shannon, in the early 1950ies. Take an arbitrary graph $\Gamma$ in which 
two subsets $A,B \subset V(\Gamma)$ are selected.  
Two players take turn coloring the edges $E(\Gamma)$ of $\Gamma$. The first one who
can build a connection between $A$ and $B$ wins. When played on a rectangular grid, 
it is the game {\bf Bridg-It} described by David Gale a few years after Shannon and
this was also described in one of the Scientific American columns. 
The players are usually denoted as V="Short" and W="Cut". While ``Hex" is played on 
vertices of a graph, the Shannon switching game is played on edges. Like Hex, there
is always a winner in this game, no draw. 
\index{Bridg-it}
\index{Shannon switching game}

\paragraph{}
{\bf Peg games} are {\bf solitaire games}. The god number is explicitly known
for some.  The topology decides whether there is a win or not. Which
topologies allow a reduction to a point? It can be NP-Hard.
\cite{UeharaIwata,KiyomiMatsui2001}.
In the case of a triangular
peg game with side length $3$, the game is a draw. There is simply
no path to the solution where only one peg remains.
The most common peg solitaire is probably the English board on 44 holes.
The European board has 4 holes more. The general problem is called also $Hi-Q$
and  NP-complete.
\index{Peg games}
\index{Hi-Q}
\index{English board}
\index{European board}

\begin{figure}[!htpb]
\scalebox{0.55}{\includegraphics{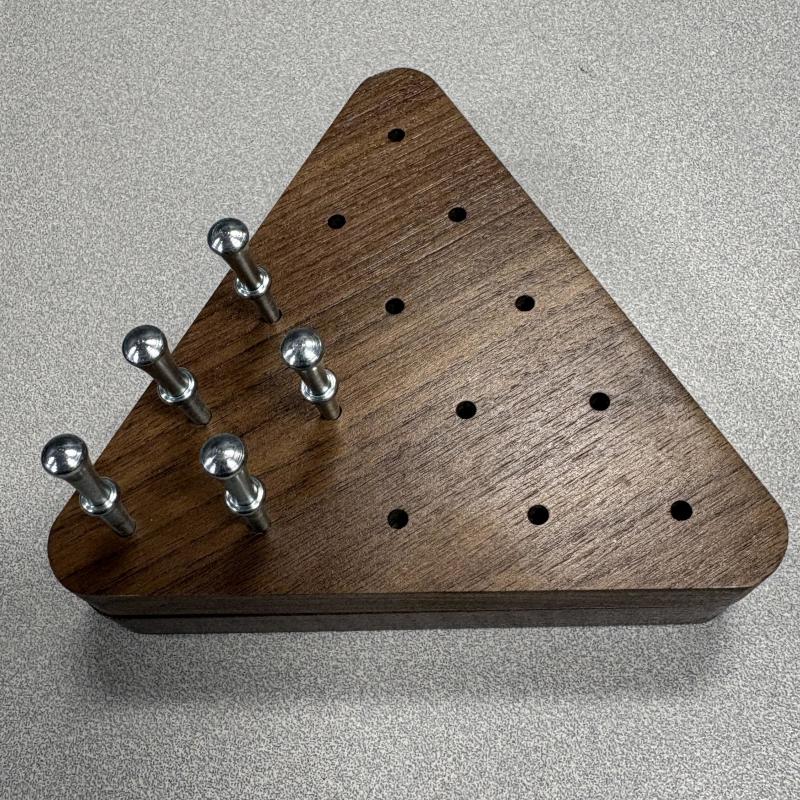}}
\label{Peg Solitaire}
\caption{
A peg solitaire where we play on a triangular grid
with 6 holes. The 9 holes to the right are considered off grid. In this
case the game has no solution with this initial condition $v$, where
$A$ is the set of configurations with one peg. 
}
\end{figure}
\index{peg solitaire}

\begin{figure}[!htpb]
\scalebox{0.85}{\includegraphics{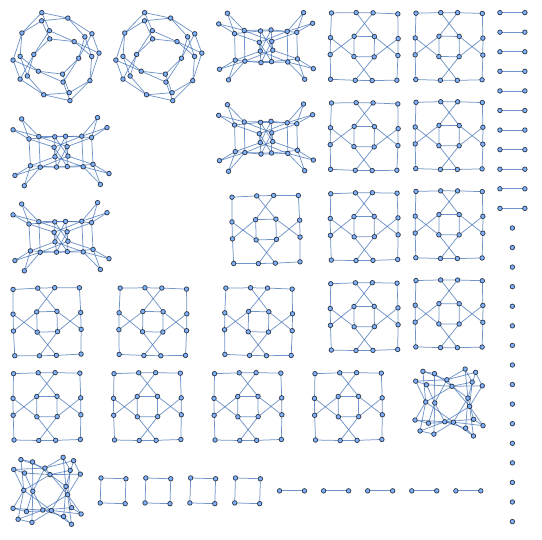}}
\scalebox{0.85}{\includegraphics{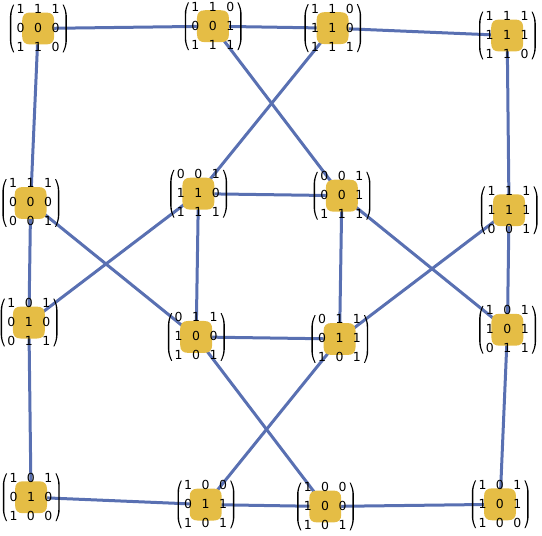}}
\label{Peg Solitaire}
\caption{
To the left we see the full game graph on a $3 \times 3$ lattice Peg solitaire.
It has $2^9=512$ vertices. 
To the right we see the connected component of an initial
condition with one peg missing (upper right vertex). 
The game can not be solved. The code producing this graph is below.
}
\end{figure}

\paragraph{}
The 80ies were already in a frenzy about Rubik type games. 
There were Hungarian ring type games like the {\bf Nintendo drum}
in which 25 balls could be rotated around or then the 
{\bf masterball} in which a sphere is cut into $4*8=32$ pieces. 
\footnote{Also called GeoMaster. Not to be confused by much later 
considered Hanayama Polekmon puzzle}
The user could rotate by 180 degrees along a longitudinal line
or then by 360/8 degrees any of the 4 latitude layers. 
The permutation group has $16!^2=437763136697395052544000000$ 
which is a subgroup of $S_{32}$. 
\index{masterball}

\begin{figure}[!htpb]
\scalebox{0.39}{\includegraphics{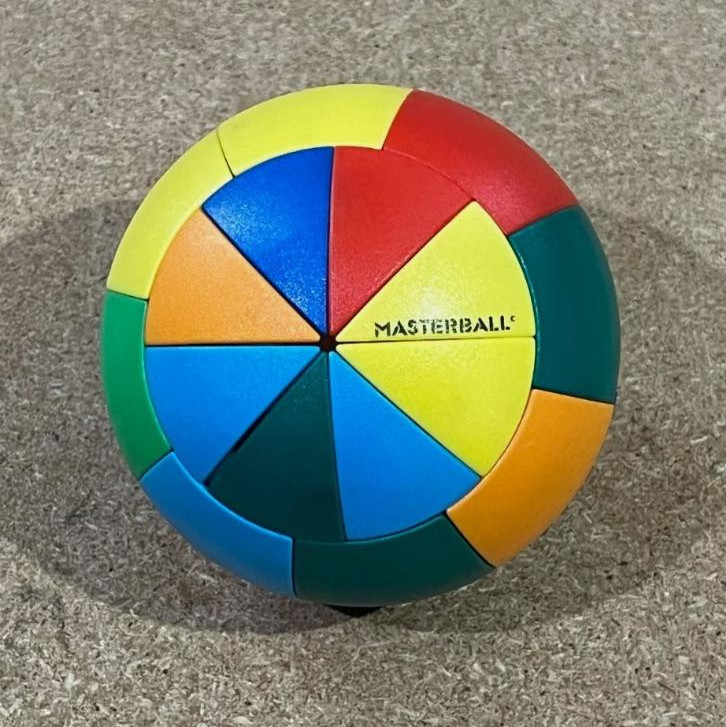}}
\label{Masterball}
\caption{
The masterball puzzle is a Cayley graph in the group 
$S_{16} \times S_{16}$. The puzzle is still sold on ebay (picture).
}
\end{figure}

\paragraph{}
{\bf NIM} is 2-player game in which the goal is to 
take the last object. Each player can take an arbitrary number of
objects from one pile. For two piles, there is a winning strategy for
the first player if the pile sizes are different. Otherwise, the second
player wins. If both piles have n pieces for example, then the god number
is $n$.  
\index{NIM}

\paragraph{}
Shannon writes in \cite{Shannon1950}:
{\it``In some games, there is a simple evaluation function $f(P)$ which can be applied to a
position P and whose value determines to which category (won, lost, etc.) the position P
belongs. In the game of NIM (Hardy and Wright, 1938), for example, this can be
determined by writing the number of matches in each pile in binary notation. These
numbers are arranged in a column (as though to add them). If the number of ones in each
column is even, the position is lost for the player about to move, otherwise won.}
\index{NIM}
\index{value function}

\paragraph{}
An other class of arithmetic games are {\bf Zeckendorf games}. 
Like NIM or alpha-zero \cite{BeneshCampbell} it is based on arithmetic. 
The Zeckendorf game is based on the fact that every integer can be written as a sum of
non-adjacent Fibonacci numbers. The now starts with a number and ends when the number is
Zeckendorf decomposed. The game always ends as a win for one of the players.
For example, for n=13, player 2 has a winning strategy. For n=32, player 1 has a winning strategy. 
\cite{CashmanMillerShuffeltonSon} displays an example where a full game tree can be displayed. 
\index{Zeckendorf games}
\index{game tree}

\paragraph{}
For {\bf checkers} we have a forced draw.
A position with 24 kings could be reached in 211 moves.
On a $10 \times 10$ board, a proof game can last up to $389$ moves.
On an $n \times n$ board, it is PSPACE-hard to determine whether
a specified player has a winning strategy.
\index{checkers}
\index{PSPACE hard}

\paragraph{}
There are 3,6 or 9 Men's {\bf Morris} versions. It is known {\bf M\"uhle}=mill
in German.  In {\bf 9 Men's morris}, the players place and move nine men
to form "mills" which are three in a row, allowing to remove
an opponents piece. The game ends when a player is 
reduced to two pieces. The Nine Men's morris goes back to the Roman empire.
This game is solved in the sense that we know it is a "draw" case in the
Zermelo trichotomy. The Twelfe Men's morris adds four diagonal lines and
is known as {\bf Morabaraba}.
\index{Morris}
\index{9 mens's morris}
\index{mills}
\index{M\"uhle}
\index{Morabaraba}

\begin{figure}[!htpb]
\scalebox{0.84}{\includegraphics{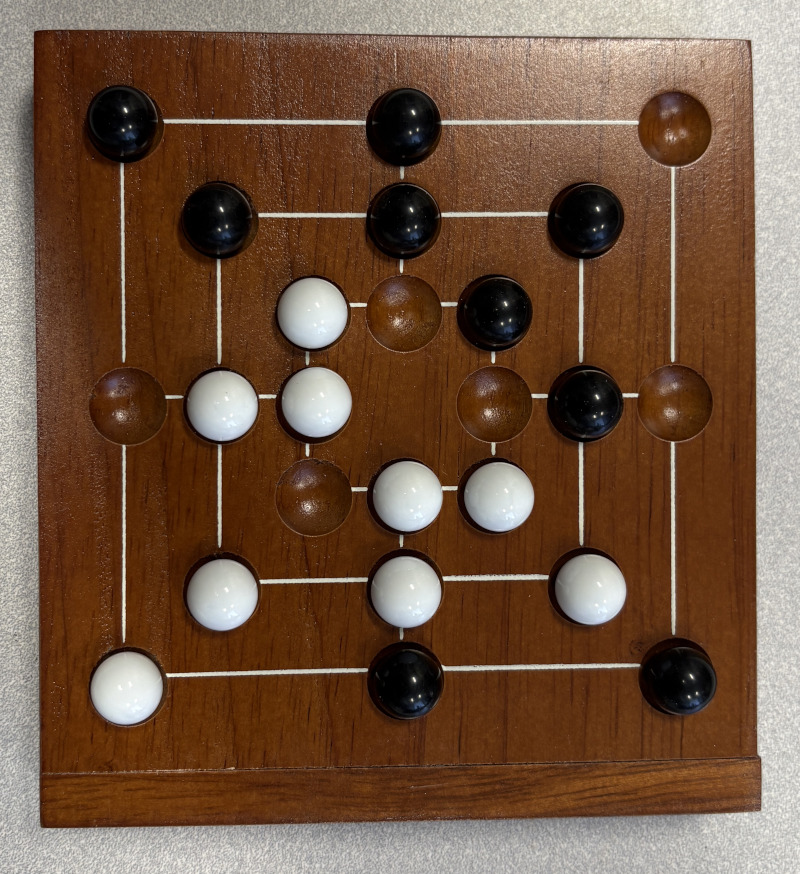}} 
\label{Morris}
\caption{
Nine Men's Morris is known to be a draw case in the Zermelo theorem. 
}
\end{figure}

\paragraph{}
{\bf Mankala} is a 2-player game that is also known as {\bf Serrata}. 
Its transport nature makes it resemble the more popular dice game 
{\bf Backgammon} but it is completely  deterministic.
The player who takes the last object loses. A game is
encoded as a vector $n=(n_1,\dots,n_k)$. The parity of the
binary digital sum $n_1 \oplus n_2 \oplus n_k$ decides then,
who wins. On smaller boards and less stones, one can write
down the entire game graph. 
\index{Mankala}
\index{Serrata}
\index{Backgammon}

\begin{figure}[!htpb]
\scalebox{0.79}{\includegraphics{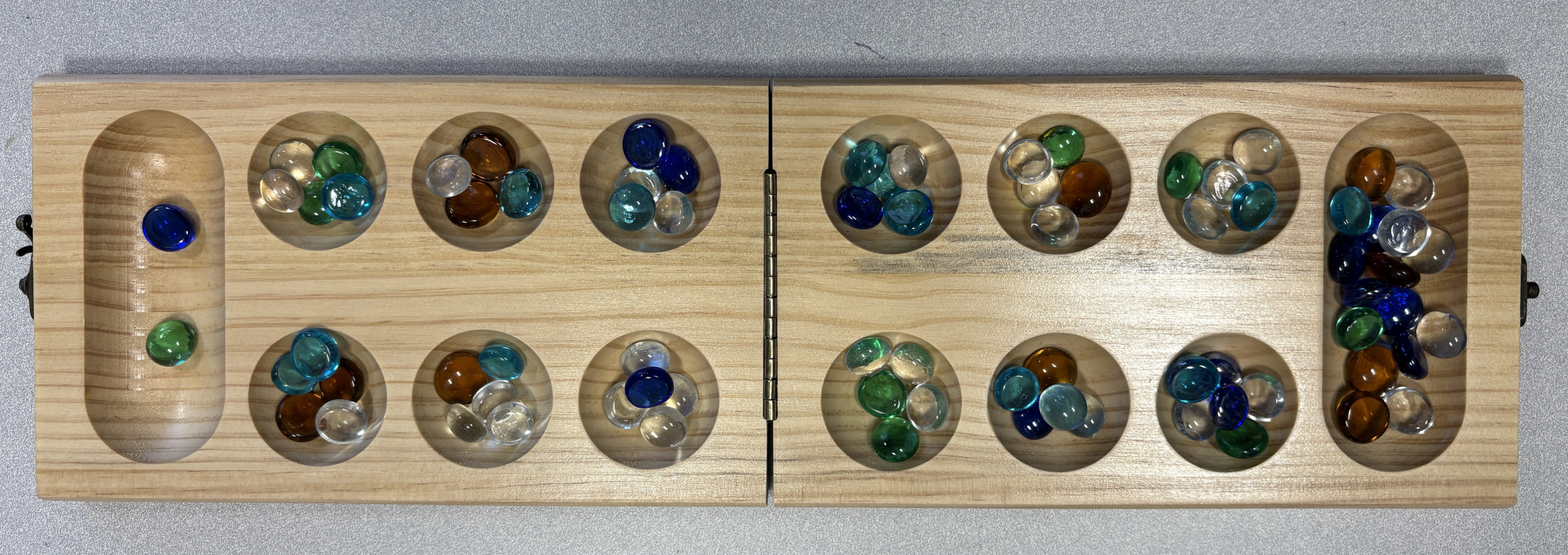}} 
\label{Mankala}
\caption{
Mankala is also known as Serrata. Unlike Backgammon, it is deterministic. Each player has pits and a bank 
to the right. The goal is to cash in as many stones as possible. To play,a  
player takes a pile from its own side and distributes the stones anti-clockwise, one by one. When ending in 
the bank, the player can play again. Play until no stones are any more in the pits. 
The game produces a complex game tree. As in chess, one can for small $n$ and $m$ (n pits, m stones
initially in each pit), write down the game tree. For $n=2,m=2$ for example, the first player wins. 
}
\end{figure}

\paragraph{}
{\bf Tetris} leads to polyomino packing problems. These
are solitaire games.  Polynomios are also called {\bf lattice animals}.
Tetris problems are also a source of nice 1-player
games. They are related to lattice animals (polyominos).
The first interesting puzzle can be done with the 
{\bf 5 tetrominoes}. By area they could fit into a 4x5 box
but a checkerboard argument shows that we would need
11 colors of one color and 9 of an other because the
T animal has 3 of one color and 1 of the other while
all other animals have 2 of one color and 2 of the other.
Famous are the 12 free pentominoes: there are 2339 ways to
arrange them in a $10 \times 6$ rectangle.
\index{pentominoes}
\index{polyonimos}
\index{tetris}
\index{tetrominoes}
\index{lattice animals}

\begin{figure}[!htpb]
\scalebox{0.4}{\includegraphics{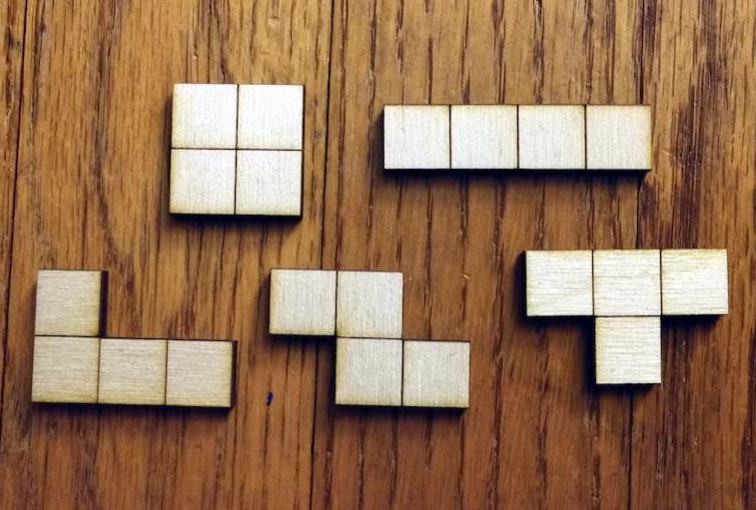}}
\label{tetrominoes}
\caption{
Tetrominoes
}
\end{figure}
\index{tetrominoes}

\paragraph{}
A more recent sliding game with larger hole size is the {\bf cross}. It is a permutation
sliding game on $S_{13}$ with 4 holes so that a-priori only $13!/4!$ positions have to be 
considered. Since there is ample of space the sorting can go quite efficient. There is
a difficulty in that the balls are moved by gravity. One can naturally see this game 
however as group game as it can be seen as a finitely presented group with 4 generators.
Each generator $t,b,l,r$ (for top bottom, left right) tilts the game physically in space so 
that one of the ends of the cross is at the bottom. The generators satisfy $tb=bt$ and $lr=rl$. 
A possible element in that presented group is $tltltlbr$. We have not yet looked at the
exercise to find the relations in this finitely presented group.

\begin{figure}[!htpb]
\scalebox{0.9}{\includegraphics{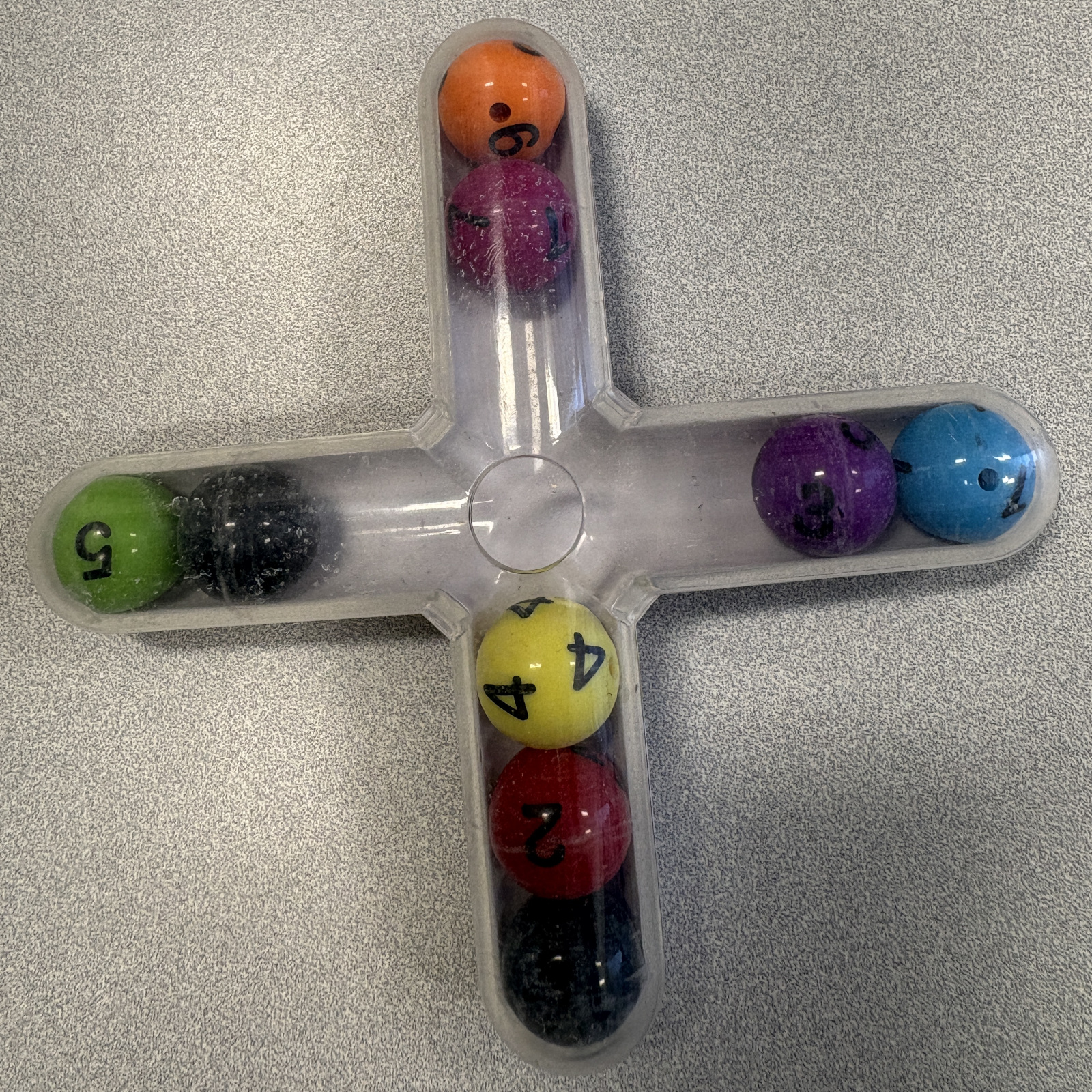}}
\label{cross}
\caption{
Cross parking game. An example of a sliding game with more holes.
(was bought at the MoMath Museum in New York).
}
\end{figure}
\index{cross parking game}

\paragraph{}
New games are also developed in modern times. An example {\bf alpha-zero} described in \cite{BeneshCampbell}.
In one of the games there, one has a stack $Z=\{1, \dots, n\}$. The bank starts with a total $0$.
Player $V$ picks a number from $Z$ and either adds or multiplies it to the total. The number is
removed from the stack. The second player now picks a number from the stack and adds or multiplies
it to the total. The first player who reaches a total that is multiple of n wins. 
This can also be played with other stacks like a collection of cards and a given number $n$. 
\index{alpha-zero game} 

\vfill
\pagebreak

\section{Remarks}

\paragraph{}
Conway constructed numbers as a subclass of games \cite{numbersgames}. 
It is concept allowing to construct {\bf surreal numbers}. 
The axiomatic setup of {\bf Conway's combinatorial game theory} can describe any
two-player game with perfect information and no chance.
Donald Knuth wrote a novel \cite{knuthnumbers} about it, in which two people found themselves on a deserted 
island with a stone telling some of the axioms of the concept. Knuth's title in English was
{\it Surreal Numbers: How Two Ex-Students Turned On to Pure Mathematics and Found Total Happiness.}
The German title of that book was {\bf island of numbers}. Note however that even for simple games like
Tic-Tac-Toe, one would hardly write the game as part of the "numbers of games" paradigm. 
\index{surreal numbers}
\index{numbers and games}
\index{Conway's combinatorial game theory}

\paragraph{}
Many solitaire game can be made into a 2 player game where two-players take turn. 
The 6-card game we have seen is first of all a simple solitaire sorting game.
The one who lands in the win position wins. This is in general not that difficult
to analyze. Such games can be played on any Cayley graph. One could also give each 
player the option to lock one of the generators. 
For transposition games where $G=S_n$ one can naturally take $V=A_n$ the 
set of even permutations and $W=S_n \setminus V$, the set of odd permutations.
After choosing $v$ and $A$ one has a game.

\paragraph{}
God number problems are related to {\bf run length analysis} of algorithms, like sorting or
path finding. A {\bf breadth first search} builds up larger and larger neighborhoods.
To get from a neighborhood of all points in distance $r$, build all
unit balls from all vertices of the boundary to get a ball of radius $r+1$.
This is {\bf Breadth-First Search}. 
Once all vertices are covered, we have the largest distance from $x$.
This can be inefficient. It can be that a vertex $x$ has $S_r(x)=G$
but $r$ is smaller than the spectral radius. A simple example is the wheel
graph, where the center $x$ gives $B_1(x)=G$ but where the graph diameter
is $2$. 
\index{run length analysis}
\index{breadth first search}
\index{BFS}

\paragraph{}
A general strategy to generate 2 player game is to take a graph, select
a certain geometric property for partially 2-colored graphs like
having a path graph of length 4 with one color. Then have two-players
place colors on edges. The first one who realizes the pattern wins. 
Already in the solitaire game, finding patters is NP complete in general. 
Finding complete sub-graphs of a graph for example is NP complete. 

\paragraph{}
The collection of all solitaire games could be made into a {\bf category}. 
The objects are the graphs a homomorphism from $(V,E,v,A)$ and $(W,F,w,B)$ 
would be a map $\phi: V \to W$  that maps edges to edges or vertices
and which maps $v$ to $w$ and $A$ to  a subset of $B$ and for which the inverse $\phi^{-1}(y)$ is
a connected sub-graph of $(V,E)$ for every $y \in W$. Homomorphisms do not honor "win" however in
general.  Two games are {\bf isomorphic} if there exist a graph isomorphism that maps $v$ to $w$
and $A$ to $B$. Isomorphic games have the same god number. An example are the following two $2x2x2$ Rubik implementations. 
\index{category}
\index{isomorphic gaphs}

\begin{figure}[!htpb] 
\scalebox{0.39}{\includegraphics{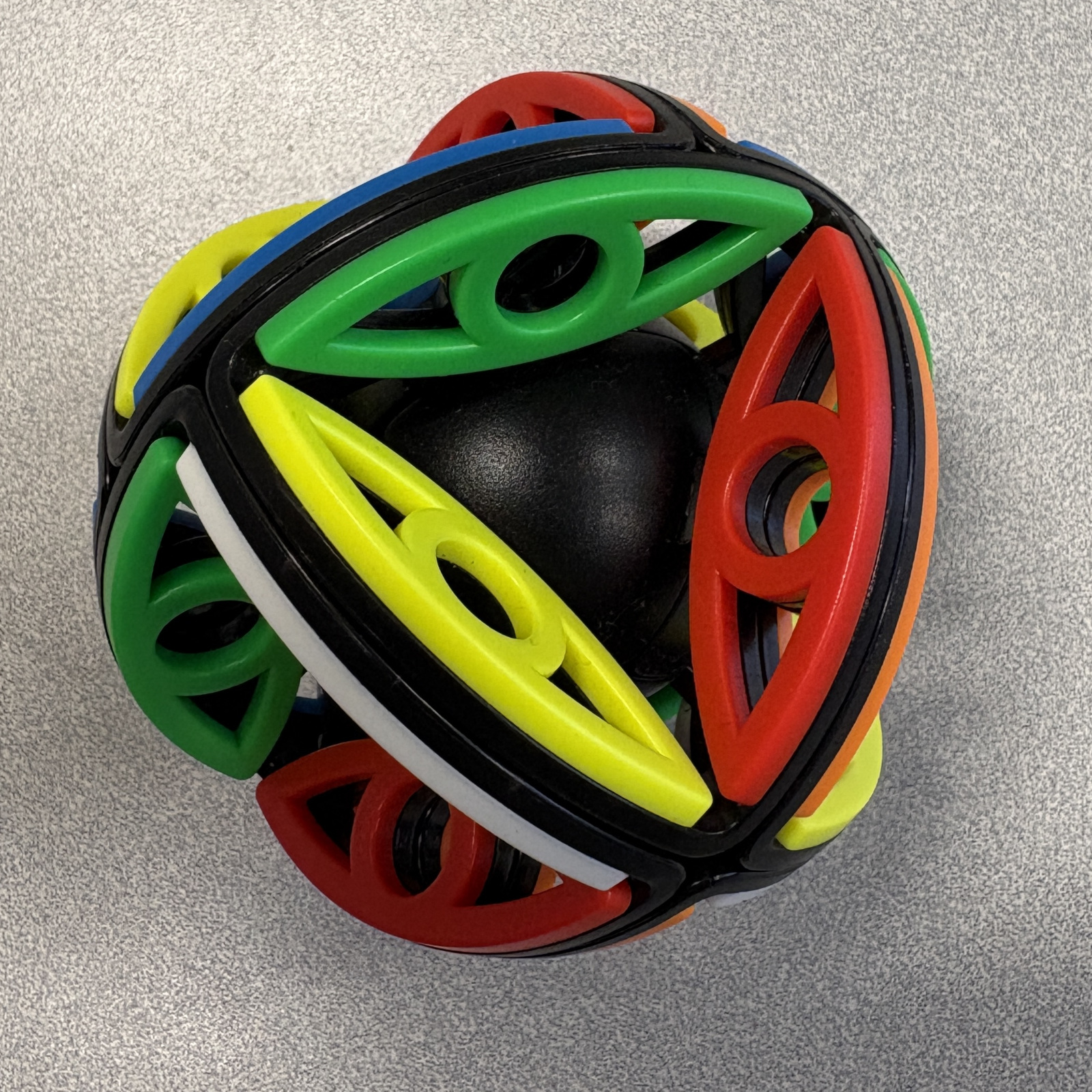}}
\scalebox{0.39}{\includegraphics{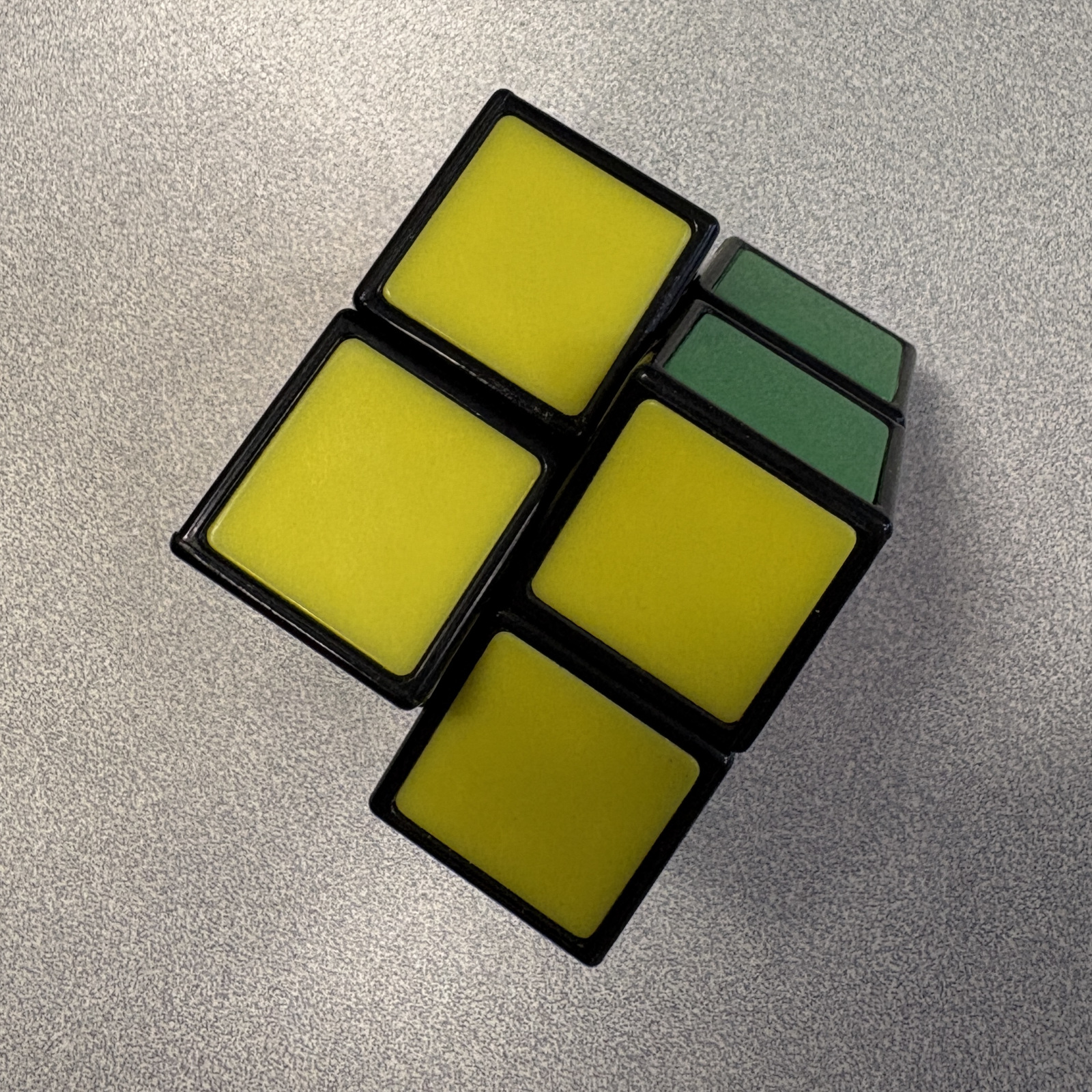}}
\label{game}
\caption{ 
Two puzzles that are isomorphic. 
}
\end{figure} 
\paragraph{}
The disjoint union of two solitary games in which $A$ in the first game is connected
to the initial condition in the second.
This is the disjoint union of the graphs. The god number is additive. 
If we add two Rubik's cubes for example, the god number to solve both of them is 
$20+20=40$. If we add two multiplayer games, the situation is more complicated. If player
$V$ decides to focus on one game, $V$ will probably win that game if the
other player focuses on the other game. An example of addition of games are
competitions where many cubes have to be solved. 

\paragraph{}
{\bf Value systems} could allow to estimate the god number.  This has been pointed out already
in \cite{Shannon1950}. They can serve as the analog of {\bf Lyapunov functions} in dynamical systems theory.
If we can find a value system which has a maximum at the equilibrium set $S$ and design a {\bf gradient method}
with respect to this function, then we have some hope that the gradient ascent will lead to a solution 
path with a decent length. 
In the Rubik cube, the value system could be the number of correct cubes.
In the chess case, it can be the sum of the values of each captured piece. 
In the 15 game it can be a distance between permutations. 
In a transposition puzzle, we have seen the graph distance $d_H(x,y)$ on $G$
which uses the distance matrix from the host graph $\Gamma$ telling which 
transpositions can be used to bring the permutation $x$ to the permutation $y$. 

\paragraph{}
We should always look also whether backtracking is efficient.
Start at the solution set A and take a few steps back to reveal part of the 
basin of attraction, positions which have to end up in A. This can be done
in the case of the {\bf peg solitaire}. 
If we can also analyze a neighborhood of $v$. If $N_n(v) \cap N_m(A)$ 
intersect we might already see the god number. 

\paragraph{}
One could also cover the entire graph by contractible graphs and
build smaller graphs in which the contractible parts are identified. 
This is at the heart of the Hanoi puzzle for example. If we glue two
disks together, we get from a Hanoi $n$ puzzle to a Hanoi $n-1$ puzzle. 
More generally, if we have finitely many points $P_1, \dots P_n$
and radii $r_i$ such that the neighborhood graphs of $N(P_i,r_i)$
intersect, then we can estimate the god number. 

\section{History}

\paragraph{}
The history of game theory starts earlier than 
von Neumann and Morgenstern \cite{NeumannMorgenstern}. A good source reference
is \cite{DimandDimand}. James Waldegrave first described minimax solutions to best 
strategies in two-person zero sum card games.
But as pointed out in the introduction to \cite{DimandDimand} Waldegrave did not 
attract attention so that the first formal reasoning for game theory are attributed
to  Cournot and Edgeworth 1897 who pioneered
{\bf oliopoly theory} to model industries with a few producers.
An important early start in mathematical game theory with 
von Neumann proved in 1928 the {\bf minimax theorem}. 
\index{Minimax theorem}
Von Neumann's {\bf minimax theorem} tells that $P^n$ is the space of probability vectors in 
$\mathbb{R^n}$ and $A$ is a $n \times m$ matrix, then the bi-linear function $f(x,y)=x \cdot Ay$ 
satisfies ${\rm max}_{x \in P^n} {\rm min}_{y \in P^m} f(x,y) =
{\rm min}_{y \in P^m} {\rm max}_{x \in P^n} f(x,y)$. Von Neumann considered simplices
like $P^n$ {\bf strategy sets} and $f(x,y)$ the {\bf expected payoff} when the 
first player plays $x$ and the second player plays $y$. It latter got generalized
to concave-convex functions, functions $f(x,y)$ that are concave in $x$ and convex in $y$. 
\index{oliopoly}

\paragraph{}
If the theory of games is seen from a finitist point of view one can pretend to be able to 
oversee the entire game graph. In this ``microcanonical perspective" probabilistic thinking is not really needed.  
When playing a traditional game like chess or
checkers, we are in a finite setting and do in principle not need probability spaces.
In a discrete finite frame work, in principle, we could do perfectly rational decisions by checking
{\bf all possible game outcomes}. While impractical from a computer science point of view,
mathematically speaking, we are on firm ground. When playing on a $4 \times 4$ chess board with 
a select number of figures, we can draw out the entire game graph. 
Game graphs are usually drawn as trees (as for example in the textbook \cite{OsborneRubinstein1994}).
Indeed, if we look at the game graph $G$ of a game, we can draw in principle build a tree
where each vertices consist of {\bf all possible simple paths starting in $v$} and two 
such paths are connected by an edge, if one is contained in the other and the length differ by $1$. 
This is called the {\bf self-avoiding walk tree} rooted at $v$. This SAW tree contains all the information
about $G$ but different graphs can produce the same tree. 
\index{microcanonical perpective}
\index{self-avoiding walk tree}
\index{SAW tree}

\paragraph{}
Kalmar \cite{Kalmar1928} already used a remarkable clarity. The 1928 article starts with:
{\it ``Ich werde mich mit solchen Spielen besch\"aftigen, welche
von zwei Spielern durch abwechselnde Z\"uge gef\"uhrt werden, deren
Wahl und Durchf\"uhrbarkeit nur von dem Entschl\"usse des am Zug
befindlichen Spielers bzw. von den Spielregeln, aber keineswegs
vom Zufall oder von der Handfertigkeit des Spielers abh\"angt,
und welche der Reihe nach Positionen herbeif\"uhren, die beiden
Spielern vollst\"andig bekannt sind. 
(Spiele von der Art der Kartenspiele werden also durch diese 
Festsetzung ausgeschlossen). Das bekannteste Beispiel bietet das 
Schachspiel; andere Beispiele: das Damenspiel, M\"uhlenspiel, Nimspiel.)"} which translates to
{\it ``I shall concern myself with those games which are played by two players 
making alternating moves, whose choice and feasibility depend only on the 
decision of the player whose turn it is, or on the rules of the game, 
but in no way on chance or on the manual dexterity of the player, and 
which successively lead to positions that are completely known to 
both players. (Games of the kind of card games are therefore excluded 
by this stipulation.) The best-known example is chess; other examples 
are checkers, Nine Men's Morris, and Nim."} Note however that a card game like 
Poker that is excluded by Kalmar is still a finite situation. There is a game
position for any possible permutation of the $52!$ permutations of a card deck. 
The game graph is big, but still finite. Poker could be described by a finite
game graph.
\index{Kalmar}

\paragraph{}
Von Neumann and Morgenstern used the idea of game tree already but the modern formal tree
language first at about the same time, notably by
Section 15.7 "Application to Chess" of \cite{NeumannMorgenstern}, the outcome of chess is
encoded in a value function taking values $1,0,-1$,
where $1$ means that $V$ (White) has a strategy to win and
$-1$ means that $W$ (Black) has a strategy to win and $0$ means that
both have a strategy such that the outcome is a tie. The mini-max principle comes in 
as this value is ${\rm max}_x {\rm min}_y H(x,y)={\rm min}_y {\rm max}_x H(x,y)$. 
The $x,y$ are however not probability vectors in this case. If we look at the axioms
\cite{NeumannMorgenstern} (included in a figure below), this is hardly used when thinking
about chess. 

\paragraph{}
Even in \cite{NeumannMorgenstern} the notion of ``strategy" is rather vague when applied to 
chess. Seeing a strategy as a subgraph of the game graph is not only more intuitive it is 
purely combinatorial. To illustrate the language comparision, 
look at section 15.1.1 in \cite{NeumannMorgenstern}:
{\it x,y in $H(x,y)$ stand for ultima analysi
not for a choice (in a move) but for the players' strategies; i.e. their entire
"theory" or "plan" concerning the game. For the moment we leave it at that.}
What does "theory" or "plan" mean? Defining "strategies" as subgraphs of
the game graph does not need to be explained. 

\paragraph{}
Non-deterministic features come in for two-player games already if both players have to make
decisions at the same time like for a {\bf prisoner dilemma} or {\bf rock-paper-scissor game},
not knowing what the other is doing. An other source of 
randomness can enter in from external parameters, like the result of a {\bf dice roll} like
in Backgammon. 
In economic situations, it occurs also in complex setups, where many parameters can only 
be estimated. Also if the game graphs are two large, a deep search through all the possible
game outcomes is impossible. 
\index{prisoner dilemma} 
\index{rock-paper-scissor game}
\index{backgammon}

\paragraph{}
The mathematics of household games and puzzles which are most popular can be dealt with 
mathematics that is discrete: graph theory or combinatorial group theory. It might not be the most
practical approach for usual games, especially if dice calls are invoked, where the game graph 
would have to include all the possible dice outcomes $\{1, \dots, 6\}^n$ when considering moves
of length $n$. In a card game like {\bf black jack} we would have to look at the finite set
of all $(k*52)!$ shufflings if a shoe of k card decks is used. Drawing out the entire {\bf game graph}
is obviously impractical already for much smaller card decks. 
\index{household games}

\paragraph{}
Zermelo's paper \cite{Zermelo1913} is considered the first mathematical analysis of strategies in 
games. The notes of Zermelo (not even 4 pages long), were not very detailed and 
Denes Koenig fixed some gaps in 1927. Zermelo for example did not restrict to finite games.
At that time, a ``strategy" specified a specific move
for each position meaning that one has not to take the maximum over all game events.
A ``quasi strategy" specifies an acceptable set of moves. 
The restriction to finite games dates back to von Neumann and Morgenstern from 1944.
We can achieve that by assuming that a game position can not be reached twice. 

\paragraph{}
The first popular group games was the 15 puzzle, a sliding puzzle that is modeled by the
group $S_{16}$ of all permutations of tiles in a $4 \times 4$  board. A parity functional
shows that the game group is the alternating group $A_{16}$. The most popular group game
still today are Rubik type puzzles where the game graph is the Cayley graph of a group. 
The natural description is the language of finitely presented groups. As the groups are 
finite and every finite group is a permutation group we can restrict to permutation 
groups. From the book \cite{Cameron}: 
"until about 1850, the term 'group' referred to a set G of transformations
of a set F, such that G is closed under composition, contains the identity
transformation, and contains the inverse of each of its elements. This is what
would now be called a `permutation group'." 
\cite{Cayley1854}.

\paragraph{}
For the axiomatic set up of Von Neumann and Morgenstern axiom system for n-player games,
see section 10.1.1 in \cite{NeumannMorgenstern}. As mentioned before, the bulk of game
theory deals with multiplayer games but this is difficult as different partitions of 
players must be considered. We cite Roth in \cite{Shapleyvalue}: 
{\it A theory of how players in a game might choose to organize themselves remains
one of the most difficult and important problems in game theory.} Shapley \cite{Shapley1952}
(who had studied at Harvard as an undergraduate) became the founder of {\bf 
cooperative game theory}.
In one of the most cited works in game theory thinks a ``game" as a set of rules with 
specified players $U$ in the playing positions and introduces then
super additive valuation $\nu$ from $2^U$ to $\mathbb{R}$ meaning 
$\nu(A) \geq \nu(A \cap B) + \nu(A \setminus B)$ with the special additive case
where $\nu$ comes from a value $\nu(A) = \sum_{i \in A} \phi(i)$ etc. 

\paragraph{}
Obviously, economists look more at situations where games have lots of players. The analysis
of ``value functions" can then become important. Shapley and Nash knew each
other as they were both young game theorists in Princeton. Nash 
developed ``non-cooperative equilibrium theory", while Shapley focused on 
{\bf cooperative games}. The umbrella was {\bf Albert Tucker} who supervised
Gale, Nash and Shapley and coined the term {\bf "Prisoner Dilemma"} in 1950
and is sometimes called the {\bf ``midwife of game theory"}. It is interesting
that Tucker got his PhD under Solomon Lefschetz working on manifolds. Lefschetz is
known among other things for his fixed point theorem generalizing the 
Brouwer fixed point theorem. Both of his academic grandsons Gale and Nash 
got involved in fixed point theorems again, Gale giving a spectacular proof of
Brouwer's fixed point theorem using game theory and Nash giving an elegant proof of the 
von Neumann minimax theorem using the Brouwer fixed point theorem, and then 
developing the concept of Nash equilibria.
\index{cooperative games}
\index{non-cooperative games}
\index{Shapley}
\index{Prisoner Dilemma}
\index{midwife of game theory}

\vfill
\pagebreak

\section{Some Code}

\paragraph{}
We see first a simple 5 line implementation of the {\bf game graph construction}. The input is 
a connected graph $\Gamma$ with $n$ vertices. The output is the {\bf game graph graph} $G$, the Cayley graph 
generated by the transpositions generated by the edges of $\Gamma$. We have mentioned the folklore
result that $G$ is $S_n$ and so has $n!$ vertices.
Our code is slightly slower than the {\bf breadth first search construction} which builds up the balls $B_r(x)$
layer by layer from a starting point $x$. We should note however that breadth first search does not necessarily 
get the graph diameter, as the center $x$ of the ball has been chosen at random. While it is true for
Cayley graphs that we always get the graph diameter with breadth first search, it is not true in general.
The graph diameter is not always ${\rm min} \{ r, B_r(x)=G \}$, where $B_r(x)$ is the ball of radius $r$.
The wheel graph $G$ for example has diameter $2$ but $B_1(x)=G$ already with radius $r=1$. 
The following construction it is almost as fast, as for connected graphs, the game graph has
always $S_n$ as vertex list. And most importantly, since we construct the game graph, we get
its diameter for certain and not just a lower bound. We observe experimentally that all cases
seen so far are Hamiltonian. 
\index{game graph construction}
\index{Hamiltonian}

\begin{tiny} 
\lstset{language=Mathematica} \lstset{frameround=fttt}
\begin{lstlisting}[frame=single]
GameGraph[gamma_]:=Module[{e,v,A,B,K},
  e=EdgeList[gamma]; e=Table[Sort[{e[[k,1]],e[[k,2]]}],{k,Length[e]}]; 
  v=VertexList[gamma]; n=Length[v]; A=Permutations[v]; K=Length[A];
  T[a_,{p_,q_}]:=Table[If[k==p,a[[q]],If[k==q,a[[p]],a[[k]]]],{k,n}];
  B={};Do[a=A[[k]]; Do[b=T[a,e[[l]]]; B=Append[B,a->b],{l,Length[e]}],{k,K}];
  UndirectedGraph[Graph[B]]];       TranspositionPuzzle=GameGraph; 

S=GraphPlot3D[GameGraph[CompleteGraph[3]]]; Export["gamegraph_k3.pdf",S,"PDF"]
S=GraphPlot3D[GameGraph[PathGraph[Range[4]]]]; Export["gamegraph_p4.pdf",S,"PDF"]
S=GraphPlot3D[GameGraph[StarGraph[4]]]; Export["gamegraph_s4.pdf",S,"PDF"]
Table[GraphDiameter[GameGraph[PathGraph[Range[k]]]],{k,1,7}]
Table[GraphDiameter[GameGraph[CycleGraph[k]]],{k,1,7}]
Table[GraphDiameter[GameGraph[WheelGraph[k]]],{k,1,7}]
Table[GraphDiameter[GameGraph[StarGraph[k]]],{k,1,7}]
GraphDiameter[GameGraph[CompleteGraph[{2,2,2}]]]        (* Gives 7, the same as GAP *)
Do[G=GameGraph[RandomGraph[{6,8}]]; Print[GraphDiameter[G]];
   Print[HamiltonianGraphQ[G]],{100}]

(* the following takes hours, look at the BFS below which is faster  *)
moebius=UndirectedGraph[Graph[{1->2,2->3,3->4,4->5,5->6,6->7,7->8,8->1,
                               1->5,5->2,2->6,6->3,3->7,7->4,4->8,8->5}]];
cylinder=UndirectedGraph[Graph[{1->2,2->3,3->4,4->1,5->6,6->7,7->8,8->5,
                                1->5,5->2,2->6,6->3,3->7,7->4,4->8,8->1}]];
GraphDiameter[GameGraph[moebius]]
GraphDiameter[GameGraph[cylinder]]
\end{lstlisting}
\end{tiny}

\paragraph{}
Here is some mathematica code which computes the lengths of the BFS layers $|S_r(v)|$
(wave front layers)  and compute it in the case of the Moebius strip and in the case of the 
Cylinder
\index{BFS layer size}
\index{wave front length}
\index{Moebius strip}
\index{cylinder}

\begin{tiny}
\lstset{language=Mathematica} \lstset{frameround=fttt}
\begin{lstlisting}[frame=single]
T[a_,X_]:=Module[{b=a}, b[[X[[1]]]]=a[[X[[2]]]]; b[[X[[2]]]]=a[[X[[1]]]]; b];
BFS[gamma_]:=Module[{e=EdgeList[gamma],v=VertexList[gamma],A,A0,A1,L={}},
  n=Length[v];A0={Range[n]};A=A0; While[Length[A0]>0,
    A1=Flatten[Table[T[A0[[k]],e[[l]]],{l,Length[e]},{k,Length[A0]}],1]; 
    A1=Complement[A1,A]; A=Union[A,A1]; A0=A1; L=Append[L,Length[A1]] ]; L];
moebius =UndirectedGraph[Graph[{1->2,2->3,3->4,4->5,5->6,6->7,7->8,8->1,
                                1->5,5->2,2->6,6->3,3->7,7->4,4->8,8->5}]];
cylinder=UndirectedGraph[Graph[{1->2,2->3,3->4,4->1,5->6,6->7,7->8,8->5,
                                1->5,5->2,2->6,6->3,3->7,7->4,4->8,8->1}]];
{BFS[moebius],BFS[cylinder]}
\end{lstlisting}
\end{tiny}

\paragraph{}
The result $BFS(moebius) = (0, 16, 137, 781, 3038, 7818, 12375, 10845, 4581, 700, 28, 0)$ and 
$BFS(cylinder) = (0, 16, 136, 784, 3104, 8008, 12480, 10606, 4383, 744, 56, 2, 0)$ confirm
that the diameter of the Cayley graphs are different. And here is the computation in GAP:

\begin{tiny}
\lstset{language=Mathematica} \lstset{frameround=fttt}
\begin{lstlisting}[frame=single]
LoadPackage("grape");  # we use the moebius strip
a:=(1,2); b:=(2,3); c:=(3,4); d:=(4,5); e:=(5,6); f:=(6,7); g:=(7,8); h:=(8,1); 
i:=(1,5); j:=(5,2); k:=(2,6); l:=(6,3); m:=(3,7); n:=(7,4); o:=(4,8); p:=(8,5);
G:= Group(a,b,c,d,e,f,g,h,i,j,k,l,m,n,o,p);
S:=GeneratorsOfGroup(G); C:=CayleyGraph(G,S);; Diameter(C);      #  10

a:=(1,2);b:=(2,3);c:=(3,4);d:=(4,1);e:=(5,6);f:=(6,7);g:=(7,8);h:=(8,5);
i:=(1,5);j:=(5,2);k:=(2,6);l:=(6,3);m:=(3,7);n:=(7,4);o:=(4,8);p:=(8,1);
G:= Group(a,b,c,d,e,f,g,h,i,j,k,l,m,n,o,p);
S:=GeneratorsOfGroup(G); C:=CayleyGraph(G,S);; Diameter(C);      #  11
\end{lstlisting}
\end{tiny}

\paragraph{}
The cycle graph case $\Gamma=C_n$ is covered by the {\bf theorem of Adin-Alon-Roichman}: \cite{AdinAlonRoichman}
giving ${\rm God}(C_n) = \langle (n^2+1)/4 \rangle$. We have given upper and lower bounds in the wheel case
(in the section about transposition games) 
The path case is the {\bf bubble sort case} with $n(n-1)/2$ diameter. Of course if 
$\Gamma_1  \subset \Gamma_2$, then the god number of $\Gamma_1$ is larger or equal than the 
god number of $\Gamma_2$. 
\index{theorem of Adin-Alon-Roichman}
\index{bubble sort}

\begin{center}
\begin{tabular}{|c|c|c|c|c|c|} \hline
Size n &  Complete &  Path     &  Cycle    & Wheel  & Star \\ \hline
3 &        2  &     3     &   2       &   2    &    3 \\
4 &        3  &     6     &   4       &   3    &    4 \\
5 &        4  &    10     &   6       &   5    &    6 \\
6 &        5  &    15     &   9       &   6    &    7 \\
7 &        6  &    21     &  12       &   9    &    9 \\
8 &        7  &    28     &  16       &  10    &   10 \\
9 &        8  &    36     &  20       &  12    &   12 \\ 
10&        9  &    45     &  25       &  13    &   13 \\
11&       10  &    55     &  30       &  15    &   15 \\ \hline
\end{tabular}
\end{center} 

\paragraph{}
We implemented the algorithm both in Mathematica as well as in the GAP software \cite{GAP}
(both in OS X and linux). Both cases computed in Cayley graphs. Here is an example
how to compute the god number. We take $\Gamma=K_{2,2,2}$, which is 
the {\bf octahedron graph}, for which the god number is $7$. The following
GAP code confirms this. We used GAP also to compute the n=10,n=11 wheel god numbers
where it took already hours in the case of $n=11$. 
Note that no formula is yet known for the wheel graph. It seems to agree with 
the star graph case for larger $n$. This means that sorting through a central
node suffices and no sorting along the circular rim is needed.  
We have also Mathematica code for transposition and 
sliding puzzles defined by any graph $\Gamma$ in the code section. For the 
{\bf magic octahedron}, the case when we have a sliding puzzle with one hole, the
god number is $7$.  Here is the GAP code for the {\bf octahedron transposition puzzle},
giving graph diameter is $7$. 
\index{GAP}
\index{Octahedron graph}
\index{octahedron sliding puzzle}
\index{magic octahedron}

\begin{tiny}
\lstset{language=Mathematica} \lstset{frameround=fttt}
\begin{lstlisting}[frame=single]
LoadPackage("grape");
a:=(1,2); b:=(2,3); c:=(3,4); d:=(4,1); e:=(1,5); f:=(2,5); 
g:=(3,5); h:=(4,5); i:=(1,6); j:=(2,6); k:=(3,6); l:=(4,6);
G:=Group(a,b,c,d,e,f,g,h,i,j,k,l);
S:=GeneratorsOfGroup(G);
C:=CayleyGraph(G,S);
Diameter(C);
\end{lstlisting}
\end{tiny}

\paragraph{}
Let us remark that if the diameter of the game graph is larger than 2, then the game graph $G$ 
obtained from $\Gamma$ with $m$ edges one can try to produce a {\bf $m$-manifold} by placing a $m$-simplex at every vertex 
of $G$ and attach at every of the $(m-1)$ faces an other m-simplex. In the case of the path graph $P_4$ with 
$4$ vertices, it is an {\bf Archimedian solid}, the {\bf great rhom-bicuboctahedron} a graph with god number $6$. 
Its dual is a {\bf Catalan solid}, the {\bf disdyakis dodecahedron} with 45 faces. It is a 2-sphere: every 
unit sphere $S_1(x)$, the boundary of $B_1(x)$ is a circular graph. 
It is not always the case that we get a manifold. An example of a 3-regular graph that is not the dual of a 2-manifold
is the $4$-M\"obius ladder (Wagner graph). 
\index{Wagner graph}
\index{Moebius ladder}
\index{Catalan solid}
\index{Archimedian solid}
\index{disdyakis dodecahedron}
\index{great rhom-bicuboctahedron}

\begin{figure}[!htpb]
\scalebox{0.49}{\includegraphics{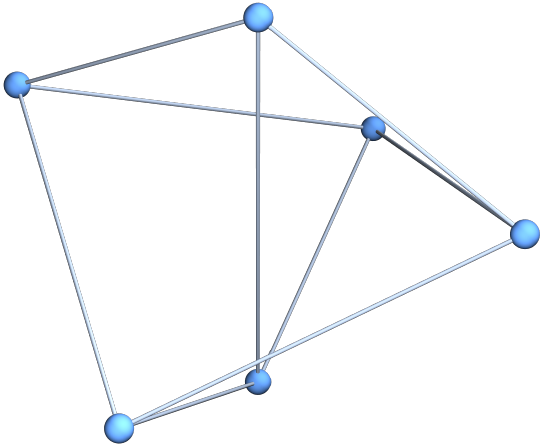}}
\scalebox{0.49}{\includegraphics{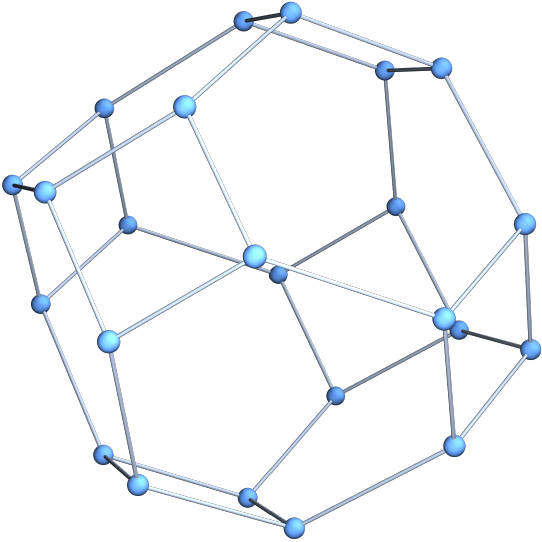}}
\scalebox{0.49}{\includegraphics{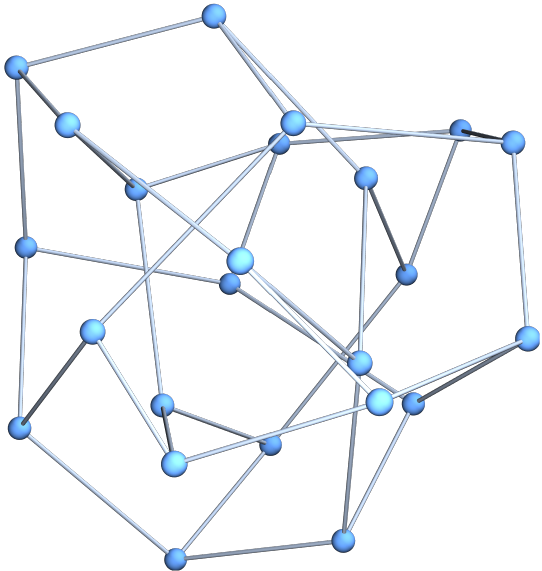}}
\label{game}
\caption{
The game graph of $K_3$ is the {\bf utility graph} $K_{3,3}$. 
The game graph of the path graph $P_4$ of length 3 is the dual of a 2-sphere (a Catalan solid in this case)
The game graph of the star graph $S_4$ is the dual of of a 2-manifold.  The pictures were
generated with the program given above. 
}
\end{figure}
\index{utility graph}
\index{path graph}

\paragraph{}
Below is GAP code for the Rubik cube. It has been used in the 80ies. 
(It was used in student projects for a mathematical software course, where one of us (O.K) was
an undergraduate course assistant \cite{KnillMaederSandersSimon}. The assignment given to the students in 
that "mathematical software" course was to write a program in CAYLEY (now MAGMA) to build a strategy that
would solve the Rubik cube from any initial condition, build on a theory of Schreier-Sims.)
\footnote{When kids got the Rubik cube in the late 1970'ies, the idea of the game was
to build your own strategy. For O.K., this had required weeks during high school. When looking and comparing
such algorithms later, it became clear that everybody (without knowing any group theory) developed independently 
some sort of {\bf Schreier-Sims algorithm}, working in a tower of {\bf stabilizer groups},
fixing larger and larger parts of the cube, then developing
generators for that group. It is interesting that kids are often better than adults in that, maybe because
they are not afraid to play around, even if they do not know what they are doing. } 
The students had to solve this homework assignment in the algebra programming language CAYLEY
(which is syntax close to GAP today and now called MAGMA). The computer would randomly have to find generators of the 
stabilizer group, until a way is found to find a larger stabilizer group etc. 
\index{Schreier-Sims algorithm}
\index{Magma}
\index{stabilizer group}

\paragraph{}
We compute here the god number of the subgroup of the Rubik cube, generated by face 
rotations by 180 degrees. This is again an example of a Cayley graph generated by involutions. 
But each generator is a product of two transpositions, not one. The god number is 15. 
The game graph has 663'552 vertices and 1'990'656 edges, girth 4. \\
As for an other example, the group generated by three generators $a^2,b^2,c^2$ has only 2592 vertices
and 3888 edges, girth 12 and diameter 14. 
The group Rubik itself produces a game graph $43252003274489856000$ that is too large for the 
machine to get the diameter for. It is 20, if one uses the generators $a^k,b^k, \dots, f^k$ for $k=1,2,3$. 

\index{girth}
\index{GAP}

\begin{tiny}
\lstset{language=gap} \lstset{frameround=fttt}
\begin{lstlisting}[frame=single]
LoadPackage("grape");
a:=( 1, 3, 8, 6)( 2, 5, 7, 4)( 9,33,25,17)(10,34,26,18)(11,35,27,19);
b:=( 9,11,16,14)(10,13,15,12)( 1,17,41,40)( 4,20,44,37)( 6,22,46,35);
c:=(17,19,24,22)(18,21,23,20)( 6,25,43,16)( 7,28,42,13)( 8,30,41,11);
d:=(25,27,32,30)(26,29,31,28)( 3,38,43,19)( 5,36,45,21)( 8,33,48,24);
e:=(33,35,40,38)(34,37,39,36)( 3, 9,46,32)( 2,12,47,29)( 1,14,48,27);
f:=(41,43,48,46)(42,45,47,44)(14,22,30,38)(15,23,31,39)(16,24,32,40);
Rubik := Group(a,b,c,d,e,f);
Size(Rubik);
G:=Group(a^2,b^2,c^2,d^2,e^2,f^2);
S:=GeneratorsOfGroup(G);
C:=CayleyGraph(G,S);
Size(UndirectedEdges(C));
Size(G);
Girth(C);
Diameter(C);
\end{lstlisting}
\end{tiny}

\paragraph{}
Here is the GAP program for the floppy:
\cite{GAP}
\index{GAP}
\index{floppy}

\begin{tiny}
\lstset{language=Mathematica} \lstset{frameround=fttt}
\begin{lstlisting}[frame=single]
# Gap code for Floppy group,  God number 13,   2304 elements 
#                 +--------------+
#                 |  1    2    3 |
#                 |  4  top    5 |
#                 |  6    7    8 |
#  +--------------+--------------+--------------+--------------+
#  |  9   10   11 | 12   13   14 | 15   16   17 | 18   19   20 | 
#  +--------------+--------------+--------------+--------------+
#                 | 21   22   23 |
#                 | 24 bottom 25 |
#                 | 26   27   28 |
#                 +--------------+
LoadPackage("grape");
a:=(1,28)(2,22)(3,26)(18,20);
b:=(6,23)(8,21)(7,22)(12,14);
c:=(1,26)(6,21)(4,24)(12,20);
d:=(3,28)(8,23)(5,25)(14,18);

G:=Group(a,b,c,d);
S:=GeneratorsOfGroup(G);
C:=CayleyGraph(G,S);
Diameter(C);
Size(G);
\end{lstlisting}
\end{tiny}

\paragraph{}
Here is a Mathematica implementation of the floppy with one cube fixed (space
symmetries are factored out):

\begin{tiny}
\lstset{language=Mathematica} \lstset{frameround=fttt}
\begin{lstlisting}[frame=single]
(* Floppy Rubik Cube implementation: from Oliver Knill 2011 for Math E 320   *)
T1[{x_,y_}]:={{x[[2]],x[[1]],x[[3]],x[[4]]},{-y[[1]],-y[[2]], y[[3]], y[[4]]}};
T2[{x_,y_}]:={{x[[1]],x[[3]],x[[2]],x[[4]]},{ y[[1]],-y[[2]],-y[[3]], y[[4]]}};
T3[{x_,y_}]:={{x[[1]],x[[2]],x[[4]],x[[3]]},{ y[[1]], y[[2]],-y[[3]],-y[[4]]}};
T4[{x_,y_}]:={{x[[4]],x[[2]],x[[3]],x[[1]]},{-y[[1]], y[[2]], y[[3]],-y[[4]]}};
T5[{x_,y_}]:={{x[[3]],x[[2]],x[[1]],x[[4]]},{-y[[1]], y[[2]],-y[[3]], y[[4]]}};
T6[{x_,y_}]:={{x[[1]],x[[4]],x[[3]],x[[2]]},{ y[[1]],-y[[2]], y[[3]],-y[[4]]}};

P=Permutations[{1,2,3,4}]; T={T1,T2,T3,T4,T5,T6};
s=Partition[Flatten[Table[{a,b,c,a*b*c},{a,-1,1,2},{b,-1,1,2},{c,-1,1,2}]],4];
v=Partition[Flatten[Table[{P[[k]],s[[l]]},{k,Length[P]},{l,Length[s]}]],8];
v=Table[Partition[v[[k]],4],{k,Length[v]}];
e=Flatten[Table[Table[v[[k]]->T[[l]][v[[k]]],{k,Length[v]}],{l,Length[T]}]];
s=UndirectedGraph[Graph[e]]; S=GraphPlot3D[s]
Print["Is Hamltonian: ",HamiltonianGraphQ[s]];
Print["God Number:   ",GraphDiameter[s]]
\end{lstlisting}
\end{tiny}

\paragraph{}
Here is the gap code for the $2 \times 2 \times 2$ Rubik (pocket cube) which 
produces a graph with $8!*3^7 = 88179840$ vertices. The diameter of the Cayley
is 14. The code needed a couple of hours to run. 

\begin{tiny}
\lstset{language=Mathematica} \lstset{frameround=fttt}
\begin{lstlisting}[frame=single]
LoadPackage("grape");
a:=(1,2,3,4)(5,17,13,9)(6,18,14,10);
b:=(5,6,7,8)(1,9,21,19)(4,12,24,18);
c:=(13,14,15,16)(2,20,22,10)(3,17,23,11);
d:=(9,10,11,12)(4,13,22,7)(3,16,21,6);
e:=(17,18,19,20)(1,8,23,14)(2,5,24,15);
f:=(21,22,23,24)(11,15,19,7)(12,16,20,8);
G:= Group(a,b,c,d,e,f);
S:=GeneratorsOfGroup(G);
C:=CayleyGraph(G,S);
Diameter(C);
Size(G);
\end{lstlisting}
\end{tiny}

\paragraph{}
And here is a GAP implementation of the masterball. It had
been advertised as a challenging puzzle because it has more
positions than Rubik. But it is actually easier as a,b,c,d 
commute and the second and third layers can not mix with the
first and fourth reflecting that the group is $S_{16} \times S_{16}$
a subgroup of $S_{32}$. 
\footnote{masterball}

\begin{tiny}
\lstset{language=Mathematica} \lstset{frameround=fttt}
\begin{lstlisting}[frame=single]
LoadPackage("grape");
a:=(1 , 2, 3, 4, 5,  6,  7,  8 );
b:=(9 ,10,11,12,13, 14, 15, 16 );
c:=(17,18,19,20,21, 22, 23, 24 );
d:=(25,26,27,28,29, 30, 31, 32 );
e:=(1,25)(9,17)(2,26)(10,18)(3,27)(11,19)(4,28)(12,20);
G:=Group(a,b,c,d,e);
S:=GeneratorsOfGroup(G);
C:=CayleyGraph(G,S);
Diameter(C);
Size(G);
\end{lstlisting}
\end{tiny}

\paragraph{}
The group above is the product of two groups of size $16! = 20922789888000$. 
The Cayley graph has diameter $9$.  
\footnote{When the rainbow master ball was released during the 1980ies, 
the company organized a contest in the Swiss town of Bern, where the puzzle had to be
solved competitively in front of a larger audience. The task was to perform a specific flip in the puzzle.
O.K. used then Cayley (now Magma) on a Sun Workstation to come up with a path that would do the flip. 
The winner of the competition and winner to a trip to the US was a Swiss farmer and cheese-maker from 
Emmental. He had the best understanding to walk around in that Cayley graph, without computer!}

\begin{tiny}
\lstset{language=Mathematica} \lstset{frameround=fttt}
\begin{lstlisting}[frame=single]
LoadPackage("grape");
a:=(1 , 2, 3, 4, 5,  6,  7,  8 );
b:=(9 ,10,11,12,13, 14, 15, 16 );
c:=(1,9)(2,10)(3,11)(4,12);
G:=Group(a,b,c);
S:=GeneratorsOfGroup(G);
C:=CayleyGraph(G,S);
Diameter(C);
Size(G);
\end{lstlisting}
\end{tiny}

\paragraph{}
The situation is interesting as small subgroups can be understood
better. If one allows only rotations by $90$ degrees one works
in the smaller group of $8!$ elements.

\begin{tiny}
\lstset{language=Mathematica} \lstset{frameround=fttt}
\begin{lstlisting}[frame=single]
LoadPackage("grape");
a:=(1 , 2, 3, 4);
b:=(5 , 6, 7, 8);
c:=(1,5)(2,6);
G:=Group(a,b,c);
S:=GeneratorsOfGroup(G);
C:=CayleyGraph(G,S);
Diameter(C);
Size(G);
\end{lstlisting}
\end{tiny}

\paragraph{}
Here is a Mathematica implementation for $n \times m$  sliding puzzles:

\begin{tiny}
\lstset{language=Mathematica} \lstset{frameround=fttt}
\begin{lstlisting}[frame=single]
SlidingPuzzleGrid[n_,m_]:=Module[{}, R=Range[n*m]; R[[n*m]]=x;
  f[X_]:=Partition[X,m]; A=Map[f,Permutations[R]]; K=Length[A];            
  T[a_,b_]:=Module[{p,q},p=Position[a,x]; q=Position[b,x]; 
      Total[Flatten[p-q]]==1 && Delete[a,p]==Delete[b,q]];
  B={}; Do[a=A[[k]]; b=A[[l]]; If[T[a,b] || 
           T[Transpose[a],Transpose[b]],B=Append[B,a->b]],{k,K},{l,K}];
  G=UndirectedGraph[Graph[A,B]]];
G=SlidingPuzzleGrid[2,3]; (* 5-puzzle, |V|=720, |E|=840 Edges. The god number is 21   *)
Print[Map[GraphDiameter,ConnectedGraphComponents[G]]];
H1=First[ConnectedGraphComponents[G]]
\end{lstlisting}
\end{tiny}

\paragraph{}
We also programmed the sliding puzzle for a general graph $\Gamma$. In this case, the ``hole" is 
the vertex with label $1$. Unfortunately, for the {\bf rainbow puzzle}, the case where $\Gamma$ is 
the icosahedron graph, the program is not strong enough yet to get us the god number. As sliding
puzzles are not Cayley graphs, we could also not just feed it to GAP. 

\begin{tiny}
\lstset{language=Mathematica} \lstset{frameround=fttt}
\begin{lstlisting}[frame=single]
SlidingPuzzleGraph[gamma_]:=Module[{e,v,A,B,K,r,s},
  e=EdgeList[gamma]; e=Table[Sort[{e[[k,1]],e[[k,2]]}],{k,Length[e]}];
  v=VertexList[gamma]; n=Length[v]; A=Permutations[v]; K=Length[A];
  T[a_,{p_,q_}]:=Table[If[k==p,a[[q]],If[k==q,a[[p]],a[[k]]]],{k,n}];
  B={};Do[a=A[[k]]; Do[{r,s}=e[[l]];
      If[a[[r]]==1 || a[[s]]==1, b=T[a,e[[l]]]; B=Append[B,a->b]],{l,Length[e]}],{k,K}];
  UndirectedGraph[Graph[B]]]; 

gamma=CompleteGraph[{2,2,2}]; G=SlidingPuzzleGraph[gamma]; 
GraphDiameter[First[ConnectedGraphComponents[G]]]        (* 11  *)
gamma=GridGraph[{2,2}]; G=SlidingPuzzleGraph[gamma]; 
GraphDiameter[First[ConnectedGraphComponents[G]]]        (*  6  *)
gamma=GridGraph[{3,2}]; G=SlidingPuzzleGraph[gamma]; 
GraphDiameter[First[ConnectedGraphComponents[G]]]        (* 21  *)
gamma=PolyhedronData["Icosahedron","Skeleton"]; 
G=SlidingPuzzleGraph[gamma]                              (* too large *)
\end{lstlisting}
\end{tiny}

\paragraph{}
Here is Mathematica code to compute the Hanoi graphs:

\begin{tiny}
\lstset{language=Mathematica} \lstset{frameround=fttt}
\begin{lstlisting}[frame=single]
check[{X_,Y_,Z_}]:=(X==Sort[X] && Y==Sort[Y] && Z==Sort[Z]);
T12[{X_,Y_,Z_}]:=If[Length[X]==0,{X,Y,Z},{Delete[X,1],Prepend[Y,X[[1]]],Z}];
T13[{X_,Y_,Z_}]:=If[Length[X]==0,{X,Y,Z},{Delete[X,1],Y,Prepend[Z,X[[1]]]}];
T21[{X_,Y_,Z_}]:=If[Length[Y]==0,{X,Y,Z},{Prepend[X,Y[[1]]],Delete[Y,1],Z}];
T23[{X_,Y_,Z_}]:=If[Length[Y]==0,{X,Y,Z},{X,Delete[Y,1],Prepend[Z,Y[[1]]]}];
T31[{X_,Y_,Z_}]:=If[Length[Z]==0,{X,Y,Z},{Prepend[X,Z[[1]]],Y,Delete[Z,1]}];
T32[{X_,Y_,Z_}]:=If[Length[Z]==0,{X,Y,Z},{X,Prepend[Y,Z[[1]]],Delete[Z,1]}];
extend[x_]:=Module[{},
 y=T12[x];If[check[y]&&Not[x==y],v=Union[v,{y}];e=Append[e,x->y]];
 y=T13[x];If[check[y]&&Not[x==y],v=Union[v,{y}];e=Append[e,x->y]];
 y=T21[x];If[check[y]&&Not[x==y],v=Union[v,{y}];e=Append[e,x->y]];
 y=T23[x];If[check[y]&&Not[x==y],v=Union[v,{y}];e=Append[e,x->y]];
 y=T31[x];If[check[y]&&Not[x==y],v=Union[v,{y}];e=Append[e,x->y]];
 y=T32[x];If[check[y]&&Not[x==y],v=Union[v,{y}];e=Append[e,x->y]];];
extendgraph:=Module[{},Do[extend[v[[i]]],{i,Length[v]}]];

HanoiGraph[n_]:=Module[{},x0={Range[n],{},{}}; v={x0}; e={};
Do[extendgraph,2^n];s=UndirectedGraph[Graph[e]];
  Print[{n,Length[VertexList[s]],Length[EdgeList[s]]}]; 
  Print[GraphDiameter[s]]; s];
s=HanoiGraph[4]
\end{lstlisting}
\end{tiny}

\paragraph{}
Here is the code which generates the game graph of the peg solitaire on a $3 \times 3$ lattice.
There is a component with maximal diameter 5. The component with one missing peg at the 
boundary has diameter 4. 

\begin{tiny}
\lstset{language=Mathematica} \lstset{frameround=fttt}
\begin{lstlisting}[frame=single]
n=3;A=Tuples[{0,1},n^2];p[x_]:=Partition[x,n];A=Map[p,A];
V[x_,y_]:=x[[1]]==y[[1]]&&x[[2]]==y[[2]]&&
    ((x[[3]]=={1,1,0} && y[[3]]=={0,0,1}) ||
     (x[[3]]=={0,1,1} && y[[3]]=={1,0,0})); 
X[x_,y_]:=V[x,y]||V[RotateRight[x,1],RotateRight[y,1]] ||
                  V[RotateRight[x,2],RotateRight[y,2]];
Y[x_,y_]:=X[Transpose[x],Transpose[y]];W[x_,y_]:=X[x,y]||Y[x,y];
B={}; Do[If[W[A[[k]],A[[l]]],
  B=Append[B,A[[k]]->A[[l]]]],{k,Length[A]},{l,Length[A]}];
G=UndirectedGraph[Graph[A,B]]
G1 = First[ConnectedGraphComponents[G]]; GraphDiameter[G1];    (* 5 *)
s=NeighborhoodGraph[G, A[[Length[A]-1]], 5]; 
S=GraphPlot[s,GraphStyle->"SmallNetwork"]
\end{lstlisting}
\end{tiny}

\pagebreak

\paragraph{}
Here is code we used to illustrate 
the $2 \times 4$ chess game with kings and queens only. It is highly inefficient
but self-contained.

{\fontsize{2}{6}\selectfont
\lstset{language=Mathematica} \lstset{frameround=fttt}
\begin{lstlisting}[frame=single]
B={{K,Q},  {0,0}, {0,0}, {k,q}};  A=Reverse[Transpose[Reverse[B]]]; q="q"; k="k"; Q="Q"; K="K";
R=MemberQ;  P=Append;

WhiteMoves[A_]:=Module[{G={}},
(* White moves without catching            *)
Do[F=A;If[R[{q,k},A[[i,j]]]&&A[[i,j+1]]==0,F[[i,j+1]]=A[[i,j]];F[[i,j]]=0; G=P[G,A->F]],{i,1,2},{j,1,3}];
Do[F=A;If[R[{q,k},A[[i,j]]]&&A[[i,j-1]]==0,F[[i,j-1]]=A[[i,j]];F[[i,j]]=0; G=P[G,A->F]],{i,1,2},{j,2,4}];
Do[F=A;If[R[{q,k},A[[i,j]]]&&A[[i+1,j]]==0,F[[i+1,j]]=A[[i,j]];F[[i,j]]=0; G=P[G,A->F]],{i,1,1},{j,1,4}];
Do[F=A;If[R[{q,k},A[[i,j]]]&&A[[i-1,j]]==0,F[[i-1,j]]=A[[i,j]];F[[i,j]]=0; G=P[G,A->F]],{i,2,2},{j,1,4}];
Do[F=A;If[R[{q,k},A[[i,j]]]&&A[[i-1,j+1]]==0,F[[i-1,j+1]]=A[[i,j]];F[[i,j]]=0; G=P[G,A->F]],{i,2,2},{j,1,3}];
Do[F=A;If[R[{q,k},A[[i,j]]]&&A[[i-1,j-1]]==0,F[[i-1,j-1]]=A[[i,j]];F[[i,j]]=0; G=P[G,A->F]],{i,2,2},{j,2,4}];
Do[F=A;If[R[{q,k},A[[i,j]]]&&A[[i+1,j+1]]==0,F[[i+1,j+1]]=A[[i,j]];F[[i,j]]=0; G=P[G,A->F]],{i,1,1},{j,1,3}];
Do[F=A;If[R[{q,k},A[[i,j]]]&&A[[i+1,j-1]]==0,F[[i+1,j-1]]=A[[i,j]];F[[i,j]]=0; G=P[G,A->F]],{i,1,1},{j,2,4}];
(* White queen moves without catching      *)
Do[F=A;If[A[[i,j]]==q&&A[[i,j+1]]==A[[i,j+2]]==0, F[[i,j+2]]=q; F[[i,j]]=0; G=P[G,A->F]],{i,1,2},{j,1,2}];
Do[F=A;If[A[[i,j]]==q&&A[[i,j-1]]==A[[i,j-2]]==0, F[[i,j-2]]=q; F[[i,j]]=0; G=P[G,A->F]],{i,1,2},{j,3,4}];
Do[F=A;If[A[[i,j]]==q&&A[[i,j+1]]==A[[i,j+2]]==A[[i,j+3]]==0, F[[i,j+3]]=q; F[[i,j]]=0; G=P[G,A->F]],{i,1,2},{j,1,1}];
Do[F=A;If[A[[i,j]]==q&&A[[i,j-1]]==A[[i,j-2]]==A[[i,j-3]]==0, F[[i,j-3]]=q; F[[i,j]]=0; G=P[G,A->F]],{i,1,2},{j,4,4}];
(* White king or queen catches straight   *)
Do[F=A;If[R[{k,q},A[[i,j]]]&&R[{Q,K},A[[i+1,j  ]]],F[[i+1,j  ]]=A[[i,j]];F[[i,j]]=0; G=P[G,A->F]],{i,1,1},{j,1,4}];
Do[F=A;If[R[{k,q},A[[i,j]]]&&R[{Q,K},A[[i-1,j  ]]],F[[i-1,j  ]]=A[[i,j]];F[[i,j]]=0; G=P[G,A->F]],{i,2,2},{j,1,4}];
Do[F=A;If[R[{k,q},A[[i,j]]]&&R[{Q,K},A[[i  ,j+1]]],F[[i  ,j+1]]=A[[i,j]];F[[i,j]]=0; G=P[G,A->F]],{i,1,1},{j,1,3}];
Do[F=A;If[R[{k,q},A[[i,j]]]&&R[{Q,K},A[[i  ,j-1]]],F[[i  ,j-1]]=A[[i,j]];F[[i,j]]=0; G=P[G,A->F]],{i,1,1},{j,2,4}];
(* White king or  queen catches diagonally *)
Do[F=A;If[R[{k,q},A[[i,j]]]&&R[{Q,K},A[[i+1,j+1]]],F[[i+1,j+1]]=A[[i,j]];F[[i,j]]=0; G=P[G,A->F]],{i,1,1},{j,1,3}];
Do[F=A;If[R[{k,q},A[[i,j]]]&&R[{Q,K},A[[i+1,j-1]]],F[[i+1,j-1]]=A[[i,j]];F[[i,j]]=0; G=P[G,A->F]],{i,1,1},{j,2,4}];
Do[F=A;If[R[{k,q},A[[i,j]]]&&R[{Q,K},A[[i-1,j+1]]],F[[i-1,j+1]]=A[[i,j]];F[[i,j]]=0; G=P[G,A->F]],{i,2,2},{j,1,3}];
Do[F=A;If[R[{k,q},A[[i,j]]]&&R[{Q,K},A[[i-1,j-1]]],F[[i-1,j-1]]=A[[i,j]];F[[i,j]]=0; G=P[G,A->F]],{i,2,2},{j,2,4}];
(* White more queen catches                *)
Do[F=A;If[A[[i,j]]==q&&A[[i,j+1]]==0&&R[{Q,K},A[[i,j+2]]],F[[i,j+2]]=q;F[[i,j]]=0; G=P[G,A->F]],{i,1,2},{j,1,2}];
Do[F=A;If[A[[i,j]]==q&&A[[i,j-1]]==0&&R[{Q,K},A[[i,j-2]]],F[[i,j-2]]=q;F[[i,j]]=0; G=P[G,A->F]],{i,1,2},{j,3,4}];
Do[F=A;If[A[[i,j]]==q&&A[[i,j+1]]==0&&A[[i,j+2]]==0&&R[{Q,K},A[[i,j+3]]],F[[i,j+3]]=q;F[[i,j]]=0;G=P[G,A->F]],{i,2},{j,1,1}];
Do[F=A;If[A[[i,j]]==q&&A[[i,j-1]]==0&&A[[i,j-2]]==0&&R[{Q,K},A[[i,j-3]]],F[[i,j-3]]=q;F[[i,j]]=0;G=P[G,A->F]],{i,2},{j,4,4}]; 
G];

BlackMoves[A_]:=Module[{G={}},
(* Black moves without catching            *)
Do[F=A;If[R[{Q,K},A[[i,j]]]&&A[[i,j+1]]==0,F[[i,j+1]]=A[[i,j]];F[[i,j]]=0; G=P[G,A->F]],{i,1,2},{j,1,3}];
Do[F=A;If[R[{Q,K},A[[i,j]]]&&A[[i,j-1]]==0,F[[i,j-1]]=A[[i,j]];F[[i,j]]=0; G=P[G,A->F]],{i,1,2},{j,2,4}];
Do[F=A;If[R[{Q,K},A[[i,j]]]&&A[[i+1,j]]==0,F[[i+1,j]]=A[[i,j]];F[[i,j]]=0; G=P[G,A->F]],{i,1,1},{j,1,4}];
Do[F=A;If[R[{Q,K},A[[i,j]]]&&A[[i-1,j]]==0,F[[i-1,j]]=A[[i,j]];F[[i,j]]=0; G=P[G,A->F]],{i,2,2},{j,1,4}];
Do[F=A;If[R[{Q,K},A[[i,j]]]&&A[[i-1,j+1]]==0,F[[i-1,j+1]]=A[[i,j]];F[[i,j]]=0; G=P[G,A->F]],{i,2,2},{j,1,3}];
Do[F=A;If[R[{Q,K},A[[i,j]]]&&A[[i-1,j-1]]==0,F[[i-1,j-1]]=A[[i,j]];F[[i,j]]=0; G=P[G,A->F]],{i,2,2},{j,2,4}];
Do[F=A;If[R[{Q,K},A[[i,j]]]&&A[[i+1,j+1]]==0,F[[i+1,j+1]]=A[[i,j]];F[[i,j]]=0; G=P[G,A->F]],{i,1,1},{j,1,3}];
Do[F=A;If[R[{Q,K},A[[i,j]]]&&A[[i+1,j-1]]==0,F[[i+1,j-1]]=A[[i,j]];F[[i,j]]=0; G=P[G,A->F]],{i,1,1},{j,2,4}];
(* Black queen moves without catching      *)
Do[F=A;If[A[[i,j]]==Q&&A[[i,j+1]]==A[[i,j+2]]==0, F[[i,j+2]]=Q; F[[i,j]]=0; G=P[G,A->F]],{i,1,2},{j,1,2}];
Do[F=A;If[A[[i,j]]==Q&&A[[i,j-1]]==A[[i,j-2]]==0, F[[i,j-2]]=Q; F[[i,j]]=0; G=P[G,A->F]],{i,1,2},{j,3,4}];
Do[F=A;If[A[[i,j]]==Q&&A[[i,j+1]]==A[[i,j+2]]==A[[i,j+3]]==0, F[[i,j+3]]=Q; F[[i,j]]=0; G=P[G,A->F]],{i,1,2},{j,1,1}];
Do[F=A;If[A[[i,j]]==Q&&A[[i,j-1]]==A[[i,j-2]]==A[[i,j-3]]==0, F[[i,j-3]]=Q; F[[i,j]]=0; G=P[G,A->F]],{i,1,2},{j,4,4}];
(* Black king or queen catches straight   *)
Do[F=A;If[R[{K,Q},A[[i,j]]]&&R[{q,k},A[[i+1,j  ]]],F[[i+1,j  ]]=A[[i,j]];F[[i,j]]=0; G=P[G,A->F]],{i,1,1},{j,1,4}];
Do[F=A;If[R[{K,Q},A[[i,j]]]&&R[{q,k},A[[i-1,j  ]]],F[[i-1,j  ]]=A[[i,j]];F[[i,j]]=0; G=P[G,A->F]],{i,2,2},{j,1,4}];
Do[F=A;If[R[{K,Q},A[[i,j]]]&&R[{q,k},A[[i  ,j+1]]],F[[i  ,j+1]]=A[[i,j]];F[[i,j]]=0; G=P[G,A->F]],{i,1,1},{j,1,3}];
Do[F=A;If[R[{K,Q},A[[i,j]]]&&R[{q,k},A[[i  ,j-1]]],F[[i  ,j-1]]=A[[i,j]];F[[i,j]]=0; G=P[G,A->F]],{i,1,1},{j,2,4}];
(* Black king or  queen catches diagonally *)
Do[F=A;If[R[{K,Q},A[[i,j]]]&&R[{q,k},A[[i+1,j+1]]],F[[i+1,j+1]]=A[[i,j]];F[[i,j]]=0; G=P[G,A->F]],{i,1,1},{j,1,3}];
Do[F=A;If[R[{K,Q},A[[i,j]]]&&R[{q,k},A[[i+1,j-1]]],F[[i+1,j-1]]=A[[i,j]];F[[i,j]]=0; G=P[G,A->F]],{i,1,1},{j,2,4}];
Do[F=A;If[R[{K,Q},A[[i,j]]]&&R[{q,k},A[[i-1,j+1]]],F[[i-1,j+1]]=A[[i,j]];F[[i,j]]=0; G=P[G,A->F]],{i,2,2},{j,1,3}];
Do[F=A;If[R[{K,Q},A[[i,j]]]&&R[{q,k},A[[i-1,j-1]]],F[[i-1,j-1]]=A[[i,j]];F[[i,j]]=0; G=P[G,A->F]],{i,2,2},{j,2,4}];
(* Black more queen catches                *)
Do[F=A;If[A[[i,j]]==Q&&A[[i,j+1]]==0&&R[{q,k},A[[i,j+2]]],F[[i,j+2]]=q;F[[i,j]]=0; G=P[G,A->F]],{i,1,2},{j,1,2}];
Do[F=A;If[A[[i,j]]==Q&&A[[i,j-1]]==0&&R[{q,k},A[[i,j-2]]],F[[i,j-2]]=q;F[[i,j]]=0; G=P[G,A->F]],{i,1,2},{j,3,4}];
Do[F=A;If[A[[i,j]]==Q&&A[[i,j+1]]==0&&A[[i,j+2]]==0&&R[{q,k},A[[i,j+3]]],F[[i,j+3]]=Q;F[[i,j]]=0;G=P[G,A->F]],{i,2},{j,1,1}];
Do[F=A;If[A[[i,j]]==Q&&A[[i,j-1]]==0&&A[[i,j-2]]==0&&R[{q,k},A[[i,j-3]]],F[[i,j-3]]=Q;F[[i,j]]=0;G=P[G,A->F]],{i,2},{j,4,4}]; 
G];

SetOptions[GraphPlot, BaseStyle -> {FontFamily -> "Times", FontSize -> 6}];
Gall={}; Vold={A}; G=WhiteMoves[A];      Gall=Union[Gall,G];s=UndirectedGraph[Graph[G]];
  Vnew=VertexList[s]; V=Complement[Vnew,Vold]; A1=A; i0=1;j0=1;A1[[i0+1,j0+1]]=A[[i0,j0]];A1[[i0,j0]]=0;
Vold=Vnew;Do[If[V[[k]]!=A1,G=BlackMoves[V[[k]]]; Gall=Union[Gall,G];s=UndirectedGraph[Graph[Gall]];
  Vnew=Union[Vnew,VertexList[s]]],{k,Length[V]}];V=Complement[Vnew,Vold];
S=GraphPlot[s,GraphStyle -> "SmallNetwork",GraphLayout -> "RadialDrawing"]
\end{lstlisting}
}

\paragraph{} And here is the BFS layer computation in the example where the layer sequence
is not concave down. 

\begin{tiny}
\lstset{language=Mathematica} \lstset{frameround=fttt}
\begin{lstlisting}[frame=single]
n=713; a=48; b=88; c=138; v=Range[n]-1; e={}; R={a,b,c,-a,-b,-c};
e=Flatten[Table[Table[ v[[k]]->Mod[v[[k]]+R[[l]],n],{l,Length[R]}],{k,Length[v]}]];
s=UndirectedGraph[Graph[e]]; v0=v[[1]]; d=GraphDiameter[s]; Print[d];  (* 10 *)
A=Table[Length[VertexList[NeighborhoodGraph[s,v0,k]]],{k,0,d}]; A=Prepend[A,0];
diff[X_]:=Table[X[[k+1]]-X[[k]],{k,Length[X]-1}]; Print[diff[A]];
\end{lstlisting}
\end{tiny}

\vfill
\pagebreak

\section{Summary of axioms}

\paragraph{}
Here is a summary of the axiomatic setup for 1-player games = solitaire games.  \\

\begin{center}
\fbox{ \parbox{17cm}{ 
1) The {\bf game graph} $G$ of a {\bf solitaire game} is a finite simple directed graph $(V,E)$.  \\
2) A distinguished vertex $v \in V$ is the {\bf initial condition} of the game.  \\
3) In the set $S=\{ x \in V, {\rm outdeg}_G(x)=0 \}$ of {\bf end positions}, we distinguish a set $A \subset S$ of {\bf win positions}. \\
4) A {\bf strategy} is a sub-graph $H$ of $G$. Let $\mathcal{H}$ be the set of strategies.\\
5) A {\bf game event} $\gamma=(v=v_0,v_1, \dots, v_k)$ in $H$ is a maximal simple path in $H$ starting in $v$. Call them $\mathcal{E}(H)$. \\
6) The length $|\gamma|$ of a game event is the number of edges in $\gamma$.  \\
7) $G$ is {\bf winnable} if there exists $H \in \mathcal{H}$ such that all game events in $H$ end in $A$. \\
8) The {\bf god number} of a winnable game is ${\rm God}(G,v,A)=d(v,A)={\rm min}_{a \in A} d(v,a)$. \\
9) The {\bf maximal play length} of a non-winnable game is $\max_{\gamma \in \mathcal{E}(G)} |\gamma|$. Define ${\rm God}(G,v,A)=\infty$ then. \\
}}
\end{center} 

\paragraph{}
Without specifying $v,A$ the corresponding global invariant is ${\rm max}_{x,y \in V, y \; {\rm reachable} \; from \; x} d(x,y)$.
The god number of the graph, its diameter is an upper bound for the god number of the game (which has additional 
structure in the form of initial condition $v$ and win positions $A$). The diameter of the game graph can
be larger than the god number of the game as the starting position $v$ can be close to the solution set $A$. 
If $v \in A$, then ${\rm God}(G)=0$. The set of game events consists now of a single path of length $|\gamma|=0$. 
In non-winnable cases, the god number is usually not evaluated. It is the maximal length of a simple path 
starting at $v$.  For a non-winnable game, the maximal play length is $(n-1)$ if and only if $G$ is 
{\bf Hamiltonian} meaning that $G$ admits a {\bf Hamiltonian path}, 
a simple path covering all vertices. The god number computation leads in general to 
{\bf NP complete problems} for solvable solitaire games, as the $n \times m$ generalization of the 15-puzzle shows.  
\index{Hamiltonian graph}

\paragraph{}
Here is the setup for 2-player games: \index{axiomatic setup} for {\bf normal play}.
In {\bf mis\`ere play}, the player who moves last loses, and $A_V=A \cap V$ would be the win positions. In that
case the last move would just not be done and $A_V = A \cap W$. But misere play could easily
be transformed to normal play. 
\index{mis\`ere play}
\index{normal play}

\begin{center}
\fbox{ \parbox{17cm}{ 
1) The {\bf game graph} is a finite simple bipartite directed graph $G=(V \cup W,E)$. Vertices $V \cup W$
   represent possible game positions. \\
2) Player $V$ plays from $V$ to $W$ and player $W$ plays from $W$ to $V$ meaning
   $E = E_V \cup E_W$ with $E_V \subset V \times W$ and  $E_W \subset W \times V$. \\
3) Denote by $S=\{ x \in V \cup W, {\rm outdeg}_G(x)=0 \}$ the set of {\bf end positions}.  \\
4) A distinguished vertex $v \in V$ is the {\bf initial condition} of the game. $v \in W$ can be arranged by switching labels. \\
5) There is a {\bf win set} $A \subset S$. Terminal positions in $A \cap W$ are wins for $V$ and terminal positions in $A \cap V$ are wins for $W$. \\
6) A {\bf game event} $\gamma$ in $H \subset G$ is a maximal simple path in $H$ that starts at $v$. \\
7) The {\bf length} of a game event is defined as $|\gamma|=|E_V \cap E(\gamma)| = | \{ 0 \leq j < k : v_j \in V, (v_j,v_{j+1}) \in E \}|$. 
   The length is measured in moves of the starting player $V$, regardless who eventually wins. \\
8) A {\bf strategy} is a sub-graph $H$ of $G$ containing $v$. The strategy $H=\{v\}$ would mean, to not play. \\
9)  A {\bf V-strategy} is a strategy $H$ with $E_W \subset E(H)$ meaning player $V$ can limit moves of $V$ but not of $W$. \\
10) A {\bf W-strategy} is a strategy $H$ with $E_V \subset E(H)$ meaning player $W$ can limit moves of $W$ but not of $V$. \\
11) {\bf G is a win for V}, if there is a V-strategy $H$ such that every game event in $H$ ends in $A \cap W$.  \\
12) {\bf G is a win for W}, if there is a W-strategy $H$ such that every game event in $H$ ends in $A \cap V$.  \\
13) {\bf G is a draw} if neither $V$ nor $W$ has a winning strategy. \\
14) If $G$ is a win for $V$, the {\bf god number} is ${\rm God}_V(G,v)=\min_{H \in \mathcal{H}_V^{win}} \max_{\gamma \in \mathcal{E}(H)} |\gamma|$. \\
15) If $G$ is a win for $W$, the {\bf god number} is ${\rm God}_W(G,v)=\min_{H \in \mathcal{H}_W^{win}} \max_{\gamma \in \mathcal{E}(H)} |\gamma|$. \\
16) If $G$ is a draw, define ${\rm God}(G,v,A)=\infty$ and define maximal play length $L_{max}(G,v)={\rm max}_{\gamma \in \mathcal{E}(G)} |\gamma|$.
}}
\end{center}

\paragraph{}
{\bf Zermelo's theorem} logically follows from the axioms: {\bf a game is either a win 
for V, a win for W or a draw}. Since a win for $V$ excludes a win for $W$ 
the fourth logical possibility that both do win can  not happen.
Indeed, if $H_V$ is a winning $V$-strategy and $H_W$ is a winning $W$-strategy,
then their intersections would be a strategy such that only game events that end both in $A \cap W$
as well as $A \cap V$, which is impossible due to the bipartite structure implying $V$ and $W$
being disjoint. One can define the combined play strategy 
$H= H_V \cap H_W$. Any maximal event in $H$ would be forced to go into two disjoint 
sets $A \cap V$ and $A \cap W$.  
{\bf Chess problems} ``mate in $k$" (with a slightly modified game rule
so that a {\bf repeated position immediately lead to a draw}) has god number $k$. 
We have assumed $v \in V$. In a chess problem ``Black to move and mate in 3", player
Black is V. Note that $|\gamma|$ always counts the number of moves of the first player $V$. 
\index{Zermelo's theorem}
\index{chess problems}
\index{mate in k moves}

\paragraph{}
There are various statements which can be deduced immediately from the axioms. Examples: \\
{\bf 1)} The set of strategies forms a $\pi$-system. This means that it a set of objects that
is closed under intersection. If $H_1,H_2$ are strategies, then $H_1 \cap H_2$
are a strategy. Also V-strategies and W-strategies form $\pi$ systems.  \\
\index{$\pi$-system}

{\bf 2)} The set of game events in $H$ naturally form a new graph that is a tree $T(H)$.
This is the {\bf self-avoiding walk tree} of $H$ rooted at $v$. This tree $T(H)$
allows to reconstruct $H$, but do not necessarily characterize $H$ as there could be 
different trees that represent the same graph. 
\index{self-avoiding walk tree}

{\bf 3)} The length function $|\gamma|$ counting the number of moves of the first player 
defines a distance notion $d(v,y) = {\rm inf}_{\gamma} |\gamma|$ of points to the initial condition $v$.
It is not a metric because $d(y,v)$ is not defined as
 $\gamma$ runs over all game events starting at $v$ and ending at $y$. One can assume $d(v,y)=\infty$ if
there is no game path from $v$ to $y$. Computer scientists usually look at balls $B_r(v)$
in BFS search and look only at $G'=B_{{\rm God}(G)}(v) \subset G$ as these are all the reachable
configurations. As this sub-graph of $G$ is again a game graph, we could just restrict
$G'$ and get the new equilibrium set $S'=S \cap G'$ and win set $A'= A \cap G'$. \\
\index{BFS search}

{\bf 4)} An other possibility is to shift the initial condition $v$ to an other starting point $w$ 
and look at the game $G'=\bigcup_r B_r(w)$ with initial condition $v'=w$ and winning set
$A'=A \cap G'$ and equilibrium set $S'=S \cap G'$. An example is to look at an initial chess position, where
only a few figures are present. A god number $k$ for a win of $V$ means then ${\rm God}=k$.  \\

\paragraph{}
We see with such examples that the setup is close to {\bf down to earth game situations} encountered
when playing traditional games or solving traditional puzzles. In order
to appreciate that, one has to look and compare with other axiomatic setups. The above graph theoretical 
axioms are {\bf down to earth axioms}. 
\index{axiomatic setup}
\index{down to each axioms}

\begin{figure}[!htpb]
\scalebox{0.59}{\includegraphics{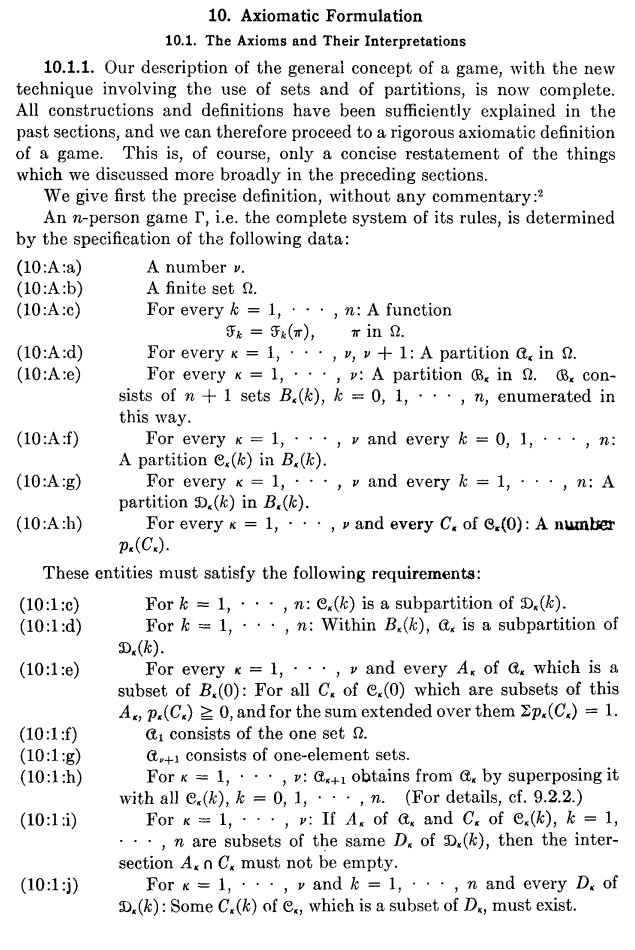}}
\label{game}
\caption{
A page of \cite{NeumannMorgenstern} listing the axioms.
}
\end{figure}
\index{von Neumann Morgenstern axioms}

\begin{figure}[!htpb]
\scalebox{0.4}{\includegraphics{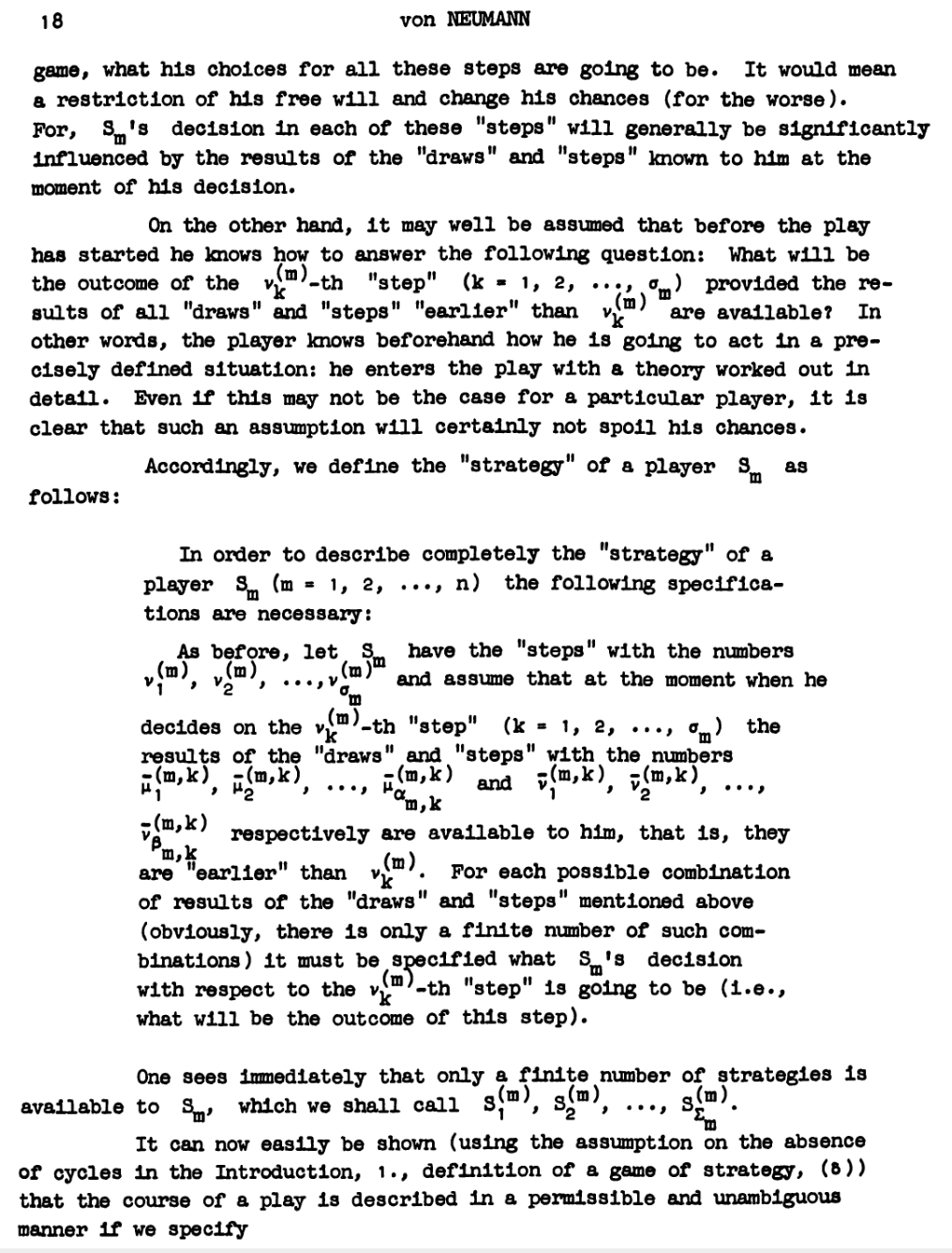}}
\label{game}
\caption{
A page of \cite{Neumann1928} defining "strategy"
}
\end{figure}
\index{strategy by von Neumann in 1928}

\vfill
\pagebreak

\section{Some questions}

\paragraph{}
The beauty about games is that the provide an endless source for problems. In general, for any 
family of games like like $n \times m \times k$ Rubik cubes or $n \times m$ tic-tac-toe games, 
one can ask asymptotic questions or extend things to 
in higher dimensions. Four dimensional Rubik cubes based on tesseracts have been implemented 
on computers in the 90'ies already. Given a puzzle(n), not only the growth rate of $God(puzzle(n))$ 
but the construction of the solution path is interesting.  
\index{tesseract rubik}
\index{higher dimensional rubik}

\paragraph{}
Problem A) Is there a closed formula for the god number for the transposition game of the
wheel graphs $W_n$? Are the god numbers of $W_n$
and $S_n$ the same for large $n$ or are there infinitely large $n$, where 
${\rm God}(S_n) = \lfloor 3n/2 \rfloor>{\rm God}(W_n)$? For the star graph, see
\cite{AkersKrishnamurthy} (where a proof sketch appears in Theorem 7.)
Unlike for path graphs $P_n$ or circular graphs $C_n$, where the god number grows quadratically with $n$, the god
number grows linearly for wheel or star graphs because the diameter of
such graphs is $2$. If a graph has diameter $O(n)$ like for $P_n$ or $S_n$, the God number is expected to grow like
$O(n^2)$.

\paragraph{}
Problem B) Given a random connected graph in Erdoes-Renyi probability space,
or a random tree obtained by some branching process. How does the God number of the transposition game graph 
or the sliding game graph behave in average? In particular, how does the 
expected God number behave in the limit $n \to \infty$? The answer will depend
on the parameters of the probability models.
\index{random puzzles}

\paragraph{}
Problem C) Sliding or transposition problems look interesting on cylinder $C_n \times P_m$
or torus $C_n \times C_n$. One can play such puzzles on any discrete manifold. We have seen that the
god number different on a M\"obius strip or a cylinder even so the Cayley graphs $G_1,G_2$ have the same vertex
degree and same size. It would be interesting to study the topological properties of these graphs. They
have the same size but different diameter. 
\index{cylinder puzzle}
\index{torus puzzle}
\index{M\"obius puzzle}

\paragraph{}
Problem D) It would be nice to have more examples of non-square $n \times m$ chess boards
set-ups where the game graph has a reasonable size.
There are infinitely many variants one can consider like on 3D cubes
or on different manifolds such as tori. If we take a standard chess and move each army
of 16 pieces one step forward, it can be played on a torus and is mate in 1. There are
likely also other games situations with unexpected outcome. Take a standard chess for
example and play it on the Moebius strip where left and right part are identified in an
opposite way. White can as a first move immediately capture a rook, but then also black 
can do that. What happens on a $8 \times 4$ board with the standard set-up with the middle
4 rows taken away so that all figures are covered. Now play this on the Moebius strip. 

\paragraph{}
Problem E) Mankala, where each player has exactly $m$ stones in each of the $n=6$ pits has a 
chance to be analyzed for small $n$ and $m$. For $n=1$ and $2$, the first player obviously wins. 
For a larger number of pits, unexpected strategies might occur due to the rule that 
if one has a pit with a lot of stones so that one can go around the entire game and land in 
an empty of its own pits, one gets all the stones from the opponent. For small $n$, one can 
draw out the game graphs. 

\paragraph{}
Problem F) What happens with the 4 peg tower of Hanoi (Reeves puzzle)
in the limit $n \to \infty$. Does the shape also converge to a fractal
object if realized suitably in $\mathbb{R}^3$, similarly as the Hanoi
graph converges to a Sirpinski fractal in $\mathbb{R}^2$. 
In the Hanoi case, the geodesic is the boundary line when embedded as the
Sirpinski fractal. What are the geodesics in the Reeves puzzle? 
\index{Sirpinski fractal}

\paragraph{}
Problem G) For transposition puzzles, there is the interesting question about the
size of the {\bf antipodal set}, the set of points which have maximal
distance to some vertex. In any Cayley graph, this is independent of the 
vertex. In general the {\bf BFS layer profile} is of interest. This is the length 
of the wave front starting at $v$. This layer profile is most of the time
concave down (unimodal) but we have seen exceptions. What BFS layer profiles can occur?
A concrete question to ask is whether the BFS layer profile determines
the graph $\Gamma$ in a transposition graph set-up or whether there are non-isomorphic
graphs $\Gamma$ with $n$ vertices such that their BFS layer profiles are the same. 
\index{antipodal set}
\index{unimodal}

\printindex

\bibliographystyle{plain}

\end{document}